\documentclass[11pt,reqno]{amsart}

\usepackage{amsmath, amsthm, amssymb}
\usepackage{amsfonts}
\usepackage{mathtools}

\usepackage[ansinew]{inputenc}
\usepackage[dvips]{epsfig}
\usepackage{graphicx}
\usepackage[english]{babel}

\usepackage{aliascnt}

\theoremstyle{plain}
\newtheorem{theorem}{Theorem}[section]

\newaliascnt{lemma}{theorem}
\newtheorem{lemma}[lemma]{Lemma}
\aliascntresetthe{lemma}

\newaliascnt{proposition}{theorem}
\newtheorem{proposition}[proposition]{Proposition}
\aliascntresetthe{proposition}

\newaliascnt{corollary}{theorem}

\aliascntresetthe{corollary}

\theoremstyle{definition}
\newaliascnt{definition}{theorem}

\aliascntresetthe{definition}

\newaliascnt{remark}{theorem}
\newtheorem{remark}[remark]{Remark}
\aliascntresetthe{remark}

\newaliascnt{example}{theorem}
\newtheorem{example}[example]{Example}
\aliascntresetthe{example}

\theoremstyle{definition}
\newtheorem*{remark*}{Remark}
\newtheorem*{assumption*}{Assumption}

\numberwithin{equation}{section}

\usepackage[backgroundcolor=white, bordercolor=blue,
linecolor=blue]{todonotes}

\usepackage{hyperref}
\usepackage{cleveref}

\usepackage{dsfont}
\usepackage{bbm}

\newcommand{\N}{\mathbb{N}}
\newcommand{\R}{\mathbb{R}}

\newcommand{\cP}{\mathcal{P}}

\newcommand{\eps}{\epsilon}

\newcommand{\ddiv}{\mathrm{div}_v}

\newcommand{\bbR}{\mathbb{R}}

\newcommand{\average}{{\mathchoice {\kern1ex\vcenter{\hrule height.4pt
width 6pt depth0pt} \kern-9.7pt} {\kern1ex\vcenter{\hrule
height.4pt width 4.3pt depth0pt} \kern-7pt} {} {} }}
\newcommand{\dashint}{\average\int}

\begin{document}
\allowdisplaybreaks
\title{Boundary Harnack estimates of optimal order for kinetic Fokker-Planck equations}

\author{Kyeongbae Kim}
\address{Institute for Applied Mathematics, University of Bonn, Endenicher Allee 60, 53115, Bonn, Germany}
\email{kkim@uni-bonn.de}
\urladdr{https://sites.google.com/view/kyeongbaekim/}

\author{Marvin Weidner}
\address{Institute for Applied Mathematics, University of Bonn, Endenicher Allee 60, 53115, Bonn, Germany}
\email{mweidner@uni-bonn.de}
\urladdr{https://sites.google.com/view/marvinweidner/}

\keywords{kinetic, Fokker-Planck, Kolmogorov, boundary, Harnack}

\subjclass[2020]{35Q84, 35B65, 82C40}

\allowdisplaybreaks

\begin{abstract}
We establish higher order boundary Harnack estimates for solutions to kinetic Fokker-Planck equations with absorbing incoming boundaries. Unlike classical elliptic and parabolic equations with Dirichlet data, we show that the quotient of two solutions for kinetic equations is not $C^{\infty}$ up to the boundary. Instead, we develop a general theory showing that, near the grazing set, the quotient of two solutions is $C^{3/2}$ if the domain and data are sufficiently smooth, and $C^{1,1}$ in the absence of source terms. These exponents are optimal.
\end{abstract}

\maketitle

\section{Introduction}  

The goal of this article is to establish higher order boundary Harnack-type estimates for solutions to kinetic Fokker-Planck equations 
    \begin{align}
    \label{eq:PDE}
        \partial_tf+v\cdot\nabla_xf-\ddiv(A\nabla_vf)=B\cdot\nabla_vf+F \quad\text{for } (t,v,x) \in (-1,0] \times \Omega \times \R^n,
    \end{align}
subject to absorbing incoming boundaries, namely
\begin{align}
\label{eq:bdry}
    f = 0\quad\text{in } (-1,0] \times \gamma_-, 
\end{align}
where $\gamma_{\pm} := (-1,0] \times \{ (x,v) \in \partial \Omega \times \R^n : \pm v \cdot n_x > 0 \}$.
This class of equations plays a central role in kinetic theory where it arises naturally as the linearization of the Landau equation from plasma physics \cite{CSS18,GHJO20,HeSn20,ImMo21,DGY22,GuSi25,FRW25}. Moreover, it models Langevin dynamics \cite{Kra40,McK63,IsWa94,GJW99,Ram22} and is connected to the pricing of Asian options in mathematical finance \cite{FNPP10,AnRe24,Bow25}.

From a purely mathematical perspective, \eqref{eq:PDE} is of high interest due to the hypoellipticity of the governing degenerate operator $f \mapsto (v \cdot \nabla_x  - \Delta_v) f$, an effect that was first described by Kolmogorov \cite{Kol34}. 
Due to its relevance in all of these fields, the regularity theory for kinetic Fokker-Planck equations has been an intensely active research area over the past two decades \cite{PaPo04,PaPo06,GIMV19,AAMN24,KLN25,ImSi20,ImMo21,ImSi21,GuMo22,Loh23,WaZh24,RoWe25}. 

The starting point of this article is the striking phenomenon that the hypoelliptic smoothing of \eqref{eq:PDE} deteriorates at the boundary of the solution domain. Specifically, near the grazing set $\gamma_0$, where velocities are tangent to the boundary,
\begin{align*}
    \gamma_{0} := (-1,0] \times \{ (x,v) \in \partial \Omega \times \R^n : v \cdot n_x = 0 \}
\end{align*}
the kinetic transport operator degenerates and therefore, the structure of the equation no longer generates full kinetic regularity in all variables. For sufficiently smooth coefficients and domains, solutions of \eqref{eq:PDE}--\eqref{eq:bdry} are known to be $C^{1/2}$ near $\gamma_0$ (see \cite{HJV14,Zhu25,KiWe26}), however they generally fail to be more regular as is indicated by explicit counterexamples $\phi_0 \in C^{1/2}$ (see \eqref{eq.phi0}).

This loss of regularity raises the natural question whether the $C^{1/2}$ regularity is a particular feature of certain special solutions to \eqref{eq:PDE}--\eqref{eq:bdry} or if the interplay between the geometry at $\gamma_0$ and the kinetic equation affects all solutions similarly. More generally, it is natural to ask:
\begin{align*}
    \textit{Do all solutions to \eqref{eq:PDE}-\eqref{eq:bdry} exhibit the same boundary behavior near the grazing set $\gamma_0$?}
\end{align*}

The goal of this article is precisely to answer this question. To measure the similarity of the boundary behavior of two solutions $f_1,f_2$ to \eqref{eq:PDE}-\eqref{eq:bdry} we establish a higher order boundary Harnack estimate which amounts to analyzing the boundary regularity of the quotient $f_1/f_2$. This is a natural approach, since if both solutions exhibit the same leading-order singular boundary behavior near $\gamma_0$, then that singularity should cancel, making $f_1/f_2$ more regular than $C^{1/2}$. 

\subsection{Main result}

In \autoref{thm.mainhar.intro} we establish a \textit{complete characterization of the uniform boundary behavior} of solutions to \eqref{eq:PDE}-\eqref{eq:bdry} by proving optimal regularity estimates of their quotient $f_1/f_2$. Our result seems to be the first kinetic boundary Harnack estimate at spatial boundaries. We show that the regularity of $f_1/f_2$  depends heavily on the direction from which $\gamma_0$ is approached. To understand this crucial aspect of our result, we decompose the interior of the solutions domain into three regions:
\begin{align*}
   \mathcal{R}_+\coloneqq \{d_{\Omega}(x) \leq (n_x\cdot v)^3\},  \quad \mathcal{R}_-\coloneqq \{d_{\Omega}(x) \leq -(n_x\cdot v)^3\}, \quad \mathcal{R}_0\coloneqq \{d_{\Omega}(x) \ge |n_x\cdot v|^3\}.
\end{align*}
Here, $d_{\Omega}(x) := \mathrm{dist}(x,\partial \Omega)$ and $n_x$ stands for the outer unit normal vector at the projection $\bar{x} \in \partial \Omega$ of $x$ to $\partial \Omega$. Note that near the boundary $\partial \Omega$ the regions $\mathcal{R}_{+}$ and $\mathcal{R}_-$ naturally extend the outgoing and incoming boundaries $\gamma_{\pm}$ into the domain and can hence be seen as the outgoing and incoming boundary layers, respectively. In contrast, $\mathcal{R}_0$ only contains velocities that are nearly tangent to $\partial \Omega$ and therefore represents the grazing region\footnote{Geometrically, projecting a phase-space point $(x,v)$ to the kinetic boundary $\gamma = \partial \Omega \times \R^n$ via a suitably scaled kinetic distance $d_{\Omega}(x) + |v \cdot n_x|^3$ maps $\mathcal{R}_{+}$, $\mathcal{R}_{-}$, and $\mathcal{R}_0$ exactly to $\gamma_{+}$, $\gamma_-$, and $\gamma_0$, respectively.}.

\begin{theorem}\label{thm.mainhar.intro}
   Let $\Omega$ be a smooth domain with $0 \in \partial \Omega$. Let $f_1,f_2$ be solutions to
    \begin{equation*}
\left\{
\begin{alignedat}{3}
\partial_tf_i+v\cdot\nabla_{x}f_i-\ddiv(A\nabla_vf_i)&=B\cdot\nabla_vf_i+F_i&&\qquad \mbox{in  $Q_2 \cap \big((-4,0] \times \Omega \times \R^n \big) =: H_2$},  \\
f_i&=0&&\qquad  \mbox{in $\gamma_-\cap Q_2$}.
\end{alignedat} \right.
\end{equation*} 
with $f_2,F_2\geq0$ and $A \ge \Lambda^{-1}I$. Then, there is $\rho_1 = \rho_1(\Lambda,\Omega) > 0$ such that if $f_2(-1,-\rho_1^3 n_0,0)=c_0>0$  for some $c_0 > 0$, then the following hold true:

\begin{itemize}
    \item[(i)] There is $\mathcal{M}=\mathcal{M}(n,\Lambda,\Omega)$ such that if for some $\eps > 0$,  $A \in C^{\frac{1}{2} + \eps}(H_2)$, $B \in L^{\infty}(H_2)$, and $F_1, F_2 \in C^{\eps}(H_2)$ with
    \begin{align}
    \label{eq:F-small}
    \|F_2\|_{L^\infty(H_2)}\leq \frac{c_0}{\mathcal{M}},
\end{align}
then
\begin{align*}
    \left[\frac{f_1}{f_2}\right]_{C^{3/2}(H_{\frac{\rho_1}8}\setminus \mathcal{R}_-)}\leq c \sum_{i=1}^2 \big( \|f_i\|_{L^\infty(H_2)} + \|F_i\|_{C^{\eps}(H_2)} \big),
\end{align*}
where $c=c(n,\Lambda,\eps,c_0,\|A\|_{C^{\frac{1}{2} + \eps}(H_2)},\|B\|_{L^{\infty}(H_2)},\Omega)$.

\item If $F_1 \equiv F_2 \equiv 0$ and for some $\eps > 0$, $A \in C^{1+\eps}(H_2)$, $B \in C^{\eps}(H_2)$, then 
\begin{align*}
     \left[\frac{f_1}{f_2}\right]_{C^{1,1}(H_{\frac{\rho_1}8}\setminus\mathcal{R}_-)}\leq c\sum_{i=1}^2 \|f_i\|_{L^\infty(H_2)} 
\end{align*}
for some constant $c=c(n,\Lambda,\eps,c_0,\|A\|_{C^{1+\eps}(H_2)},\|B\|_{C^{\eps}(H_2)},\Omega)$.
\end{itemize}
\end{theorem}

Our result indicates that solutions exhibit a universal leading-order boundary profile near the grazing set $\gamma_0$ within $\mathcal{R}_0 \cup \mathcal{R}_+$. Consequently, taking the quotient of two solutions factors out their dominant singularities, yielding regularity up to order $C^{3/2}$ (or $C^{1,1}$ if $F_i \equiv 0$). We show in \autoref{ex:sharpness} that the exponents $C^{3/2}$ in (i) and $C^{1,1}$ in (ii) are optimal, even for the Kolmogorov equation in the half-space. This proves that the universal boundary behavior does not extend to higher-order asymptotic terms, which remain singular and are highly solution-dependent.

We emphasize that \autoref{thm.mainhar.intro} does not hold for the incoming boundary layer  $\mathcal{R}_-$, where the quotient $f_1/f_2$ is generally \textit{unbounded}, even away from $\gamma_0$. In \autoref{ex.r-} we provide explicit examples of nonnegative solutions to \eqref{eq:PDE}--\eqref{eq:bdry} that remain smooth up to $\gamma_-$, but whose quotient is unbounded\footnote{We can show higher regularity of $f_1/f_2$ as in \autoref{thm.mainhar.intro} outside the smaller set $\mathcal{R}_-^M = \{ M d_{\Omega}(x) \le -  (n_x \cdot v)^3)\}$ for any $M > 1$. In that case, the constants depend on $M$ and explode as $M \to \infty$. See also the discussion in \autoref{rmk.bdryhar}. }. This failure of the boundary Harnack principle can be attributed to the infinite order decay of solutions near $\gamma_-$ (see \cite{Sil22,Zhu25,KiWe26}).

Inside the outgoing boundary layer $\mathcal{R}_+$, but away from $\gamma_0$, we can prove that the quotient $f_1/f_2$ is actually smooth (see \autoref{lem.bdry+}). This is due to the fact that solutions are smooth up to $\gamma_+$ (see \cite{RoWe25}) but do not vanish at $\gamma_+$, as no boundary condition can be prescribed there.

Ultimately, we emphasize that \autoref{thm.mainhar.intro} stands in stark contrast to classical higher order boundary Harnack estimates for elliptic and parabolic equations (see \cite{DeSa15,BaGa16,Kuk22,TTV24,Zha24,JeVi24}), where quotients of solutions always remain $C^{\infty}$. Remarkably, boundary Harnack estimates of any order even hold for nonlocal problems (see \cite{AbRo20}), where solutions are in general only H\"older continuous up to the boundary. In this light, our result reveals a distinguishing feature of kinetic equations: the degeneracy of the hypoelliptic effect at the grazing set governs not only the absolute behavior of solutions, but their relative behavior as well. To our knowledge, \autoref{thm.mainhar.intro} is the first result to capture this phenomenon.

We make several further comments on our result.

\begin{remark}
    \begin{itemize}
    \item[(i)] In light of higher order boundary Harnack estimates for elliptic equations (see \cite{TTV24}), our assumptions on the coefficients $A,B$ seem to be sharp within the class of H\"older spaces.
        \item[(ii)] In \autoref{thm.mainhar.gen} we prove boundary Harnack estimates of order $\alpha \in (0,3/2]$ under weaker assumptions on $A,B,F_i,\Omega$, and of order $\alpha \in (3/2,2]$ in case $F_i = 0$. This is the most general version of our main result and it can be seen as a Schauder-type theory for $f_1/f_2$.

        \item[(iii)] The smallness of $F_2$ in \eqref{eq:F-small} is a natural assumption in the literature (see for instance {\cite{ Tor24}}) and is required in order to propagate the positivity of $f_2$ from $(-1,-\rho_1^3 n_0,0)$ to $H_{\rho_1/8}$ via Harnack chains (see \autoref{lem.exp.posright}). It is possible to drop the assumption \eqref{eq:F-small} if we impose that 
        \begin{align*}
            \inf_{H_{c_0}(-1,-\rho_1^3 n_0,0)} f_2 \ge c_0 ,
        \end{align*}
        which allows to propagate positivity of $f_2$ as in \autoref{lem.exp.posright}  using the weak Harnack inequality.
        \item[(iv)] Note that \autoref{thm.mainhar.intro} yields regularity only in a $\rho_1$-neighborhood of the boundary. It seems that this neighborhood can in general not be enlarged since the propagation of positivity is distorted by the kinetic transport operator (see the discussion in {Section \ref{sec:propagation}}).
        \item[(v)] We emphasize that our main result holds true for solutions in kinetic cylinders $Q_1$ that are one-sided in time (see Section \ref{sec:prelim} for the precise definition). 
    \end{itemize}
\end{remark}

\begin{remark}
    In \cite[Theorem 1.7]{KiWe26}, we have established that $f/\Phi \in C^{0,1}(\mathcal{R}_0 \cup \mathcal{R}_+)$ for an explicitly constructed profile $\Phi \in C^{1/2}$. While this result characterizes the boundary behavior of an individual solution $f$, the fact that $\Phi$ is not a solution \textit{limits the regularity of the quotient $f/\Phi$}. Indeed, the $C^{0,1}$ regularity in \cite[Theorem 1.7]{KiWe26} is optimal (see \autoref{ex:Phi}). In contrast, the higher order estimates in $C^{3/2}$ and $C^{1,1}$ obtained in \autoref{thm.mainhar.intro} demonstrate that boundary Harnack estimates are essential for capturing the higher order solution behavior. Moreover, note that, as we explain below, the proof of \autoref{thm.mainhar.intro} differs fundamentally from the one of \cite[Theorem 1.7]{KiWe26}. Most importantly, the proof of \autoref{thm.mainhar.intro} not only hinges on fine boundary expansions, but also on lower solutions bounds such as in \autoref{lem.hopf.curve}. 
\end{remark}

\subsection{Discussion of the literature}

Boundary Harnack principles have a rich history due to their various applications in potential theory \cite{Kem72,Dah77}, free boundary problems \cite{CFMS81,ACS96}, \cite{AlSh19,DeSa15,DeSa20,RoTo21,Zha24}, and probability theory \cite{BBB91,BaBu94}. Classical results establish that quotients of elliptic or parabolic solutions are H\"older continuous up to the boundary, even in rough domain (Lipschitz or worse). Our main result belongs to the distinct class of higher order boundary Harnack estimates, which provide a  Schauder-type regularity theory for quotients of two solutions and have been established in \cite{DeSa15,BaGa16,Kuk22,TTV24,Zha24,JeVi24} for elliptic and parabolic equations. 

For kinetic equations, \autoref{thm.mainhar.intro} provides the first regularity estimate for quotients of solutions near spatial boundaries. Previous results {\cite{CNP12,CNP13}} have established Carleson-type estimates near Lipschitz boundaries, but even the H\"older regularity of $f_1/f_2$ has remained open. On the other hand, let us mention the seminal series of articles \cite{NyPo16,Nys17,LiNy21,LiNy21b,LiNy22}, where the authors establish fine properties of the kinetic harmonic measure, as well as comparison principles, and boundary Harnack estimates of class $C^\alpha$. Crucially, however, their analysis is restricted to domains of the form
\begin{align}
\label{eq:Nystrom-domain}
    \{(t,x,v) \in \R^{1+2n} : v_n > \psi(t,x,v')\}, \quad \text{where } \psi \text{ is Lipschitz continuous}.
\end{align}
The geometry of \eqref{eq:Nystrom-domain} is fundamentally different from the one of spatial domains as considered in \autoref{thm.mainhar.intro}. In fact, the boundary of \eqref{eq:Nystrom-domain} is everywhere \textit{non-characteristic} in the sense that it is ``seen'' by the diffusion in $v$ and therefore a boundary condition needs to be prescribed everywhere. This leads to a solution behavior resembling much more the one of elliptic and parabolic equations. Spatial boundaries, in contrast, are ``invisible'' to the diffusion in $v$ and all interaction with the boundary is governed by the transport operator. The analysis of spatial domains requires different tools which we develop in this article. 

\begin{remark}
In \cite[Remark 1.5]{NyPo16}, the authors ask whether boundary Harnack estimates can also hold in domains different from \eqref{eq:Nystrom-domain}. In this paper, we provide a partial answer to this question for the class of spatial domains. It is an interesting research topic to investigate boundary Harnack estimates in spatial Lipschitz domains or in domains with boundaries jointly in $(t,x,v)$. The latter case could be relevant for the study of the free boundary in kinetic obstacle-type problems (see \cite{Bow25}).
\end{remark}

\subsection{Strategy of the proof}

The key difficulty in the proof of our main result arises from the singular solution behavior near the grazing set $\gamma_0$, which has recently been investigated in \cite{KiWe26} (see also \cite{Ram22,HLW24,Zhu25}). The main work in this article consists in establishing a Hopf-type lemma and novel fine expansions of solutions to \eqref{eq:PDE}--\eqref{eq:bdry} near the grazing set, as we explain below. 

The central object in this analysis is the explicit solution $\phi_0$
\begin{equation}\label{eq.phi0}
\left\{
\begin{alignedat}{3}
v\partial_x\phi_0-\partial_{vv}\phi_0&=0&&\qquad \mbox{in  $\{x>0\} \times \R$}, \\
\phi_0&= 0&&\qquad\mbox{in $\{x=0\}\times \{v>0\}$},
\end{alignedat} \right.
\end{equation} 
which was computed in \cite{Gol71, GJW99} (see also \cite{KiWe26})
\begin{align*}
\phi_0(x,v) = 
\begin{cases}
    x^{1/6}U\left(-\frac16,\frac23,-\frac{v^3}{9x}\right)&\quad\text{if }v<0,\\
             \frac{\Gamma(7/6)}{\Gamma(1/6)} x^{1/6}e^{-\frac{v^3}{9x}}U\left(\frac56,\frac23,\frac{v^3}{9x}\right)&\quad\text{if }v\geq0.
             \end{cases}
\end{align*}
The function $\phi_0$ has the following asymptotic boundary behavior (see \cite{Gol71,GJW99,HLW24,KiWe26})
\begin{align}
\label{eq:phi_0-behavior}
\phi_0(x,v)\eqsim\begin{cases}
       x^{\frac16}\left(\frac{v^3}x\right)^{-\frac56}e^{-\frac{v^3}{9x}}\quad&\text{in } \{x < v^3\} = \mathcal{R}_-,\\
        x^{\frac16}\quad&\text{in } \{x \ge |v|^3\} = \mathcal{R}_0,\\
        |v|^{\frac12}\quad&\text{in } \{x < - v^3\} = \mathcal{R}_+,
        \end{cases}
\end{align}
and it is easy to see that $\phi_0 \in C^{\frac{1}{2}} \setminus C^{\frac{1}{2} + \eps}$ and is homogeneous of degree $\frac{1}{2}$, i.e., $\phi_0(\lambda^3 x , \lambda v) = \lambda^{\frac{1}{2}} \phi_0(x,v)$.

The first main ingredient in the proof of \autoref{thm.mainhar.intro} is the following lower bound for solutions to \eqref{eq:PDE}--\eqref{eq:bdry}. This lemma plays the same role as the classical Hopf lemma for elliptic or parabolic equations. To our knowledge, this is the first result of its kind for kinetic equations with spatial boundaries, and we believe it is of independent interest.

\begin{theorem}\label{lem.hopf.curve}
   Let $\Omega$ be a bounded $C^{2,\eps}$ domain for some $\eps\in(0,1)$. Let $f\geq0$ be a weak solution to 
    \begin{equation*}
\left\{
\begin{alignedat}{3}
\partial_tf+v\cdot\nabla_{x}f-\ddiv(A\nabla_vf)&=B\cdot\nabla_vf+F&&\qquad \mbox{in $H_2$}, \\
f&=0&&\qquad  \mbox{in $\gamma_-\cap H_2$},
\end{alignedat} \right.
\end{equation*} 
where $A\in C^{\eps}(H_2)$ with $A \ge \Lambda^{-1}I$, $B\in L^\infty(H_2)$, and $0\leq F\in L^{\infty}(H_2)$.

Then, there are $\rho_1=\rho_1(\Lambda,\Omega) > 0$ and  $\mathcal{M}=\mathcal{M}(n,\Lambda,\Omega)$ such that if
\begin{align*}
    f(-1,-\rho_1^3 n_0,0)=c_0>0, \qquad \|F\|_{L^\infty(H_2)}\leq \frac{c_0}{\mathcal{M}},
\end{align*}
for some $c_0 > 0$, then 
\begin{align*}
    \frac{c_0}{c}\leq \inf_{H_{\frac{\rho_1}8}\setminus \mathcal{R}_-}\frac{f}{\Phi}
\end{align*}
for some constant $c=c(n,\Lambda,\eps,\|A\|_{C^\eps(H_2)},\|B\|_{L^\infty(H_2)},\Omega)$, where ${\Phi(x,v)\coloneqq \phi_0(d_\Omega(x),v\cdot n_x)}$.
\end{theorem}

The proof of the Hopf lemma proceeds in three main steps. First, we establish a propagation of positivity result for solutions to kinetic equations (see \autoref{lem.exp.posright}). To do so, we employ suitable Harnack chains to propagate the pointwise information  at $(-1,-\rho_1^3 n_0,0)$ to points inside $H_{\rho_1/8} \cap {\{ d_{\Omega} \ge (\rho_1/32)^3\}}$. This is quite delicate in the kinetic setting due to the specific geometry induced by the transport term and the fact that the waiting time in the kinetic Harnack inequality cannot be chosen arbitrarily small (see \cite{CNP10}). 
In the second step we prove the Hopf lemma for  equations with constant coefficients in the half-space by constructing suitable barrier functions (see \autoref{lem.hopf}) and using the propagation of positivity result. 
Finally, we develop a novel perturbation argument to generalize the Hopf lemma from constant to H\"older continuous coefficients (see Subsection \ref{subsec:Hopf-general}). This step is technically demanding since it relies on fine pointwise expansions of solutions at $\gamma_0$ in terms of the highly anisotropic function $\phi_0$.

The second main ingredient in the proof of \autoref{thm.mainhar.intro} is to establish new expansions of solutions at $\gamma_0$ in the half-space, i.e., when $\Omega = \{ x_n > 0 \}$. In \cite[Proposition 5.1]{KiWe26}, we have already shown that any solution $f_1$ to \eqref{eq:PDE}--\eqref{eq:bdry} satisfies for any $r \in (0,\frac{1}{2}]$
\begin{align}
\label{eq:expansion}
    |f_1(z) - p_{0}(z) \phi_0(z) - p_1(z) \varphi_1(x_n,v_n) - b_2 \varphi_2(x_n,v_n) - b_3 \varphi_3(x_n,v_n)| \lesssim r^{3 - \eps} ~~ \forall z  \in H_r,
\end{align}
where $p_0,p_1 \in \cP_2$, $b_2,b_3 \in \R$, and the functions $\varphi_1,\varphi_2,\varphi_3$ are explicit solutions to
\begin{equation*}
\left\{
\begin{alignedat}{4}
v\partial_x\varphi_1 - \partial_{vv}\varphi_1 &= \partial_{vv} \phi_0, &\quad 
v\partial_x\varphi_2 - \partial_{vv}\varphi_2 &= 1, &\quad 
v\partial_x\varphi_3 - \partial_{vv}\varphi_3 &= \phi_0 &\qquad &\text{in } \{x>0\} \times \R, \\
\varphi_1 &= 0, &\quad 
\varphi_2 &= 0, &\quad 
\varphi_3 &= 0 &\qquad &\text{in } \{x=0\} \times \{v>0\}
\end{alignedat}
\right.
\end{equation*}
which satisfy $\varphi_1(\lambda^3 x, \lambda v) = \lambda^{\frac{1}{2}} \varphi_1(x,v)$, $\varphi_2(\lambda^3 x, \lambda v) = \lambda^{2} \varphi_2(x,v)$, $\varphi_3(\lambda^3 x, \lambda v) = \lambda^{2 + \frac{1}{2}} \varphi_3(x,v)$ (see also Subsection \ref{subsec:regularity} and Subsection \ref{subsec:aux} for a more detailed discussion).
 
 The expansion \eqref{eq:expansion} reveals the fine asymptotic behavior of solutions to \eqref{eq:PDE}--\eqref{eq:bdry} up to order three. 
 To prove a higher order boundary Harnack estimate, we need to significantly improve \eqref{eq:expansion}. The main idea is to transform \eqref{eq:expansion} into an expansion for $f_1$ in terms of another solution $f_2$, as follows:
\begin{align}
\label{eq:expansion-bh}
    |f_1(z) - q_0(z) f_2(z) - q_1(z) \varphi_1(x_n,v_n) - b'_2 \varphi_2(x_n,v_n) - b'_3 \varphi_3(x_n,v_n)| \lesssim r^{3 - \eps} ~~ \forall z \in H_r,
\end{align}
where $q_0,q_1 \in \cP_2$ and $b'_2, b'_3 \in \R$ (see \autoref{lem.exp.bdry.diff}). We establish \eqref{eq:expansion-bh} via a blow-up argument that is conceptually motivated by \cite{AbRo20}. In the kinetic setting, however, the analysis is substantially more delicate due to the complexity of the error terms. Hence, the proof requires several new elements (see for instance the role of \autoref{lem.aux.direction}). Another crucial feature of \eqref{eq:expansion-bh} is that we actually have to prove that $q_1 \in \cP_2$ satisfies $q_1(0) = \nabla_v q_1(0) = 0$ which requires a very careful analysis.

Subsequently \autoref{thm.mainhar.intro} is deduced upon dividing \eqref{eq:expansion-bh} by $f_2$ and transferring the regularity of the individual quotients
\begin{align*}
    q_0, \qquad q_1 \frac{\varphi_1}{f_2}, \qquad \frac{\varphi_2}{f_2}, \qquad \frac{\varphi_3}{f_2}
\end{align*}
to $f_1/f_2$. The key observation here is that the most singular function $\phi_0$ from \eqref{eq:expansion} has been eliminated in \eqref{eq:expansion-bh} while the singular behavior of $\varphi_1$ near the origin is compensated by the conditions $q_1(0) = \nabla q_1(0) = 0$. Precisely these two features are responsible for the higher regularity of the quotient $f_1/f_2$.

In light of the Hopf lemma in \autoref{lem.hopf.curve} and the homogeneity of the functions $\varphi_i$, it becomes apparent that $q_1 \frac{\varphi_1}{f_2} \asymp |z|^{2}$, $\frac{\varphi_2}{f_2} \asymp |z|^{\frac{3}{2}}$, and $\frac{\varphi_3}{f_2} \asymp |z|^2$ as the kinetic distance $|z| = |x|^{\frac{1}{3}} + |v| \to 0$. This asymptotic behavior yields the $C^{3/2}$ regularity of $f_1/f_2$ established in \autoref{thm.mainhar.intro}(i). Moreover, we show that $\tilde{b}_2 = 0$ in case $F_i \equiv 0$, which allows to improve the regularity of $f_1/f_2$ to $C^{1,1}$ in \autoref{thm.mainhar.intro}(ii).

\subsection{Outline}
This paper is structured as follows. In Section \ref{sec:prelim} we introduce the kinetic H\"older spaces used in the article and recall several regularity results established in previous works. In Section \ref{sec:propagation} we establish a general propagation of positivity result for kinetic equations near spatial boundaries (see \autoref{lem.exp.posright}). The Hopf lemma (see \autoref{lem.hopf.curve}) is proved in Section \ref{sec:Hopf}. Finally, Section \ref{sec:bdry-Harnack} is dedicated to the proof of the higher order boundary Harnack estimate and establishes our main result \autoref{thm.mainhar.intro}. Moreover, we provide several counterexamples indicating the sharpness of our results in this section, as well as in the Appendix. 

\subsection{Acknowledgments}
Kyeongbae Kim and Marvin Weidner were supported by the Deutsche Forschungsgemeinschaft (DFG,
German Research Foundation) under Germany's Excellence Strategy - EXC-2047/1 - 390685813 and through the CRC 1720 ``Analysis of criticality: from complex phenomena to models and estimates'', 53930965.

\pagebreak

\section{Preliminaries}
\label{sec:prelim}
Throughout the entire paper we always assume that for some constant $\Lambda\geq1$, it holds
\begin{align*}
    \Lambda^{-1}I \leq A\leq \Lambda I\quad\text{and}\quad |B|\leq \Lambda.
\end{align*}

Next, let us collect some notation.
\begin{itemize}
\item Let us write $X=(x,v)$. Given $X\in \bbR^n\times \bbR^n$, we write $X'=(x',v')\in\bbR^{n-1}\times \bbR^{n-1}$, where $X=(x',x_n,v',v_n)$.
\item Define a stationary cylinder by 
\begin{align}\label{stationary cylinder}
    \mathrm{Q}_r\coloneqq B_{r^3}\times B_r\quad\text{and}\quad \mathrm{H}_r\coloneqq B_{r^3}\times B_r\cap \{x_n>0\}.
\end{align}
In particular, we write $\mathrm{Q}'_{r}\coloneqq B'_{r^3}\times B'_r$, where $B'_{r}$ is the $(n-1)$-dimensional ball with radius $r$.
\item Let us write $z=(t,x,v)$. Given $z\in \bbR\times\bbR^n\times \bbR^n$, we write $z'=(t,x',v')\in\bbR\times\bbR^{n-1}\times \bbR^{n-1}$, where $z=(t,x',x_n,v',v_n)$.
    \item We set $I_{r^2} = (-r^2,0]$. Define a kinetic cylinder by
    \begin{align*}
        Q_r\coloneqq  I_{r^2}\times B_{r^3}\times B_r\quad\text{and}\quad H_r\coloneqq Q_r\cap\big(\R \times \{x_n>0\} \times \R^n).
    \end{align*}
Furthermore, we write $Q'_{r}\coloneqq I_{r^2}\times B'_{r^3}\times B'_r$, where $B'_{r}$ is the $(n-1)$-dimensional ball with radius $r$.
\item For any $R>0$ and $z \in \R^{2n+1}$, we define $S_Rz\coloneqq (R^2t,R^3x,Rv)$. For $z_0 \in \R^{2n+1}$, we set
\begin{align*}
    z_0 \circ z = (t_0 + t, x_0 + x + t v_0 , v_0 + v)
\end{align*}
\item Given a domain $(a,b)\times \Omega\times V\subset \bbR\times \bbR^n\times \bbR^n$ with $\Omega $ and $V$ being bounded sets, we define the kinetic boundary of $(a,b)\times \Omega\times V$ by 
    \begin{align*}
        \partial_{\mathrm{kin}}((a,b)\times \Omega\times V)=(\{t=a\}\times \Omega\times V) \cup((a,b)\times \Omega\times\{v\in\partial V\})\cup \gamma_-.
    \end{align*}
\item For a given vector $v\in\bbR^n$, we write $v^{\mathrm{t}}$ as a transpose of $v$, hence $v\cdot w=v^{\mathrm{t}}w$.
\end{itemize}

\subsection{Kinetic H\"older spaces}

We recall the definition of the kinetic H\"older space and its properties. 

First, we introduce a space of kinetic polynomials. For any $k\in\N\cup\{0\}$, we define $\cP_k$ by
\begin{align*}
    \cP_k\coloneqq \mathrm{span}\left\{\sum_{|\beta|\leq k}t^{\beta_t}x^{\beta_{x_1}}\cdots x_{n}^{\beta_n}v_1^{\beta_{v_1}}\cdots v_n^{\beta_{v_n}}\Bigg|\begin{array}{l} \beta=(\beta_t,\beta_{x_1},\ldots,\beta_{x_n},\beta_{v_1},\ldots,\beta_{v_n})\\
    |\beta|=2|\beta_t|+\sum\limits_{i=1}^{n}(3|\beta_{x_i}|+|\beta_{v_i}|)
    \end{array}\right\}.
\end{align*}
Moreover, we write $\cP_{-1} = \{ 0 \}$.

Now, we define the kinetic H\"older semi-norm by
\begin{align*}
    [f]_{C^{k,\delta}(D)}\coloneqq\sup_{z_0\in D}\inf_{p_{z_0}\in\cP_k}\sup_{z\in D}\frac{|(f-p_{z_0})(z)|}{\mathrm{dist}_{\mathrm{kin}}(z,z_0)^{k+\delta}},
\end{align*}
for any $k\in\N\cup\{0\}$ and $\delta\in(0,1]$, where 
\begin{align*}
    \mathrm{dist}_{\mathrm{kin}}(z,z_0)\coloneqq \max\{|t-t_0|^{\frac12},|x-(x_0+(t-t_0)v_0)|^{\frac13},|v-v_0|\}.
\end{align*}
In particular, we write for any $\delta\leq0$, $[f]_{C^{\delta}(H_R)}\coloneqq\|f\|_{L^\infty(H_R)}$.

Furthermore, we denote by $D^k$ a $k$-th order differential operator. For example, $D$ may represent $\partial_{v_i}$, $D^{2}$ may denote operators such as $\partial_{v_i,v_j}$ or $\partial_t+v\cdot\nabla_x$, and $D^3$ can be $\partial_{v_i,v_j,v_k}$, $\partial_{v_i}(\partial_t+v\cdot\nabla_v)$ or $\partial_{x_i}$.

We now provide several auxiliary lemmas for kinetic H\"older spaces.

\begin{lemma}\label{lem.interpolhol}\cite[Lemma 2.10]{RoWe25} Let $r > 0$ and $z_0 \in \R^{2n+1}$. Let $k,l\in\N\cup\{0\}$ and $\eps,\delta\in(0,1]$. Suppose $k+\eps<l+\delta$. Then we have
\begin{align*}
    [f]_{C^{k,\eps}(Q_r(z_0))}\leq c\left(r^{l+\delta-(k+\eps)}[f]_{C^{l,\delta}(Q_r(z_0))}+r^{-(k+\eps)}\|f\|_{L^\infty(Q_r(z_0))}\right)
\end{align*}
    for some constant $c=c(n,k,l,\eps,\delta)$.
\end{lemma}
\begin{lemma}\label{lem.grabdd.hol}\cite[Lemma 2.7]{ImSi21} Let $r > 0$ and $z_0 \in \R^{2n+1}$. Let $k,l\in\N\cup\{0\}$ and $\eps\in(0,1]$. Then 
\begin{align*}
    [D^lf]_{C^{k,\eps}(Q_r(z_0))}\leq c[f]_{C^{k+l,\eps}(Q_r(z_0))}
\end{align*}
for some constant $c=c(n,k,l,\eps)$.
\end{lemma}
Next, we give the converse version of \autoref{lem.grabdd.hol}. 
\begin{lemma}\label{lem.holspace}
    Let $r > 0$ and $z_0 \in \R^{2n+1}$. Let $k\in\{1,2\}$ and $\eps\in(0,1]$. Then we have 
    \begin{align*}
     [f]_{C^{k+\eps}(Q_{r/2}(z_0))}\leq c\left( [\nabla_v^kf]_{C^\eps(Q_r(z_0))}+\left[(\partial_t+v\cdot\nabla_x)^{\left[\frac{k+\eps}{2}\right]}f\right]_{C^{{k+\eps-2[(k+\eps)/2]}}_{\mathcal{T}}(Q_r(z_0))}+[f]_{C^{{k+\eps}}_x(Q_r(z_0))}\right)
    \end{align*}
    where we write for any $\alpha\in(0,2]$ and $\beta\in(0,3]$,
    \begin{equation}\label{defn.hol.xt}
    \begin{aligned}
         &[f]_{C^{{\alpha}}_{\mathcal{T}}(Q_r(z_0))}\coloneqq\sup_{z_1,z=(t_2,x_1+(t_2-t_1)v_1,v_1)\in Q_r(z_0)}\frac{|f(z_1)-f(z)|}{\mathrm{dist}(z_1,z)^{\alpha}},\\
         &[f]_{C^{\alpha}_x(Q_r(z_0))}\coloneqq \sup_{z_1,z=(t_1,x_2,v_1)\in Q_r(z_0)}\frac{|f(z_1)-f(z)|}{\mathrm{dist}(z_1,z)^{\alpha}}.
    \end{aligned}
    \end{equation}
\end{lemma}
A similar result can also be established for $k\geq3$. Since it is not needed in this paper, we omit its statement and proof.
\begin{proof}
We assume $r=2$ and $z_0=0$. We observe from the definition of the H\"older space that  
\begin{align*}
    \sup_{z_0\neq z_1\in Q_1}\frac{|f(z_1)-p_{z_0,k}(z_1)|}{\mathrm{dist}(z_0,z_1)^{k+\eps}}=[f]_{C^{k+\eps}(H_1)},
\end{align*}
where $p_{z_0,k}$ is the $k$-th order Taylor polynomial of $f$ at $z_0$. In particular, $p_{z_0,2}$ is given by
\begin{align*}
    p_{z_0,2}(z)\coloneqq f(z_0)+\nabla_vf(z_0)\cdot (v-v_0)+\frac12(v-v_0)^{\mathrm{t}}\nabla_v^2f(z_0) (v-v_0)+(\partial_t+v_0\cdot\nabla_x)f(z_0)(t-t_0).
\end{align*}
Let us write
 \begin{align*}
     z_1=(t_1,x_1,v_1), \quad \overline{z}_1\coloneqq(t_0,x_0,v_1),\quad\text{and}\quad \widetilde{z}_1\coloneqq (t_1,x_0+(t_1-t_0)v_1,v_1)
 \end{align*}
 to see that $\overline{z}_1,\widetilde{z}_1\in Q_2$ and 
 \begin{align}\label{dist.ineq}
     \min\big\{ \mathrm{dist}(\overline{z}_1,\widetilde{z}_1),|v_1-v_0|,\mathrm{dist}(\widetilde{z}_1,z_1) \big\}\leq \mathrm{dist}(z_0,z_1).
 \end{align}
 Thus, we have 
 \begin{align*}
     |f(z_1)-p_{z_0,k}(z_1)|&\leq |f(\overline{z}_1)-p_{z_0,k}(\overline{z}_1)|+|f(\overline{z}_1)-f(\widetilde{z}_1)+(p_{z_0,k}(\overline{z}_1)-p_{z_0,k}({z}_1))|+|f(z_1)-f(\widetilde{z}_1)|\\
     &\eqqcolon J_1+J_2+J_3.
 \end{align*}
 For the two terms $J_1$ and $J_2$, we use the fundamental theorem of calculus to see that
 \begin{align*}
     J_1\leq c[\nabla^k_vf]_{C^{\eps}(Q_2)}|v_1-v_0|^{k+\eps},\quad J_2\leq c\left[(\partial_t+v\cdot\nabla_x)^{\left[\frac{k+\eps}{2}\right]}f\right]_{C^{{k+\eps-2[(k+\eps)/2]}}_{\mathcal{T}}(Q_2)}\mathrm{dist}(\overline{z}_1,\widetilde{z}_1)^{{k+\eps-2[(k+\eps)/2]}}.
 \end{align*}
 Next, by the definition of the H\"older space given in \eqref{defn.hol.xt}, we have 
 \begin{align*}
     J_3\leq c\mathrm{dist}(\tilde{z}_1,z_1)^{k+\eps}[f]_{C_x^{k+\eps}(Q_2)}.
 \end{align*}
 Thus, using the estimates $J_1,J_2$ and $J_3$ together with \eqref{dist.ineq}, we derive 
 \begin{align*}
     \frac{|f(z_1)-p_{z_0,k}(z_1)|}{\mathrm{dist}(z_0,z_1)^{k+\eps}}\leq c\left([\nabla_v^kf]_{C^\eps(Q_2)}+\left[(\partial_t+v\cdot\nabla_x)^{\left[\frac{k+\eps}{2}\right]}f\right]_{C^{{k+\eps-2[(k+\eps)/2]}}_{\mathcal{T}}(Q_2)}+[f]_{C^{{k+\eps}}_x(Q_2)}\right),
 \end{align*}
 which gives 
 \begin{align*}
     [f]_{C^{k+\eps}(Q_1)}\leq c\left([\nabla_v^kf]_{C^\eps(Q_2)}+\left[(\partial_t+v\cdot\nabla_x)^{\left[\frac{k+\eps}{2}\right]}f\right]_{C^{{k+\eps-2[(k+\eps)/2]}}_{\mathcal{T}}(Q_2)}+[f]_{C^{{k+\eps}}_x(Q_2)}\right).
 \end{align*}
By this with $f$ replaced by $f_0(z)\coloneqq f(z_0\circ S_{r/2}z)$, we have the desired estimate.
\end{proof}

\subsection{Regularity estimates}
\label{subsec:regularity}

In this subsection, we collect several regularity estimates for kinetic equations (in the interior and up to the boundary) that have been established in previous works. In addition, we provide three explicit 1D solutions to Kolmogorov's equation in the half space which appear in expansion estimates at the grazing set $\gamma_0$.

First, we recall interior Schauder type estimates of solutions to kinetic equations (see \cite[Lemma 2.23]{RoWe25} or \cite[Lemma 2.5]{KiWe26}). 
\begin{lemma}\label{lem.intsch}
    Let $r > 0$ and $z_0 \in \R^{2n+1}$. Let $k\in \N\cup\{0\}$ and $\eps\in(0,1)$. Let $A\in C^{k-1,\eps}(Q_r(z_0))$, $B\in C^{k-2,\eps}(Q_r(z_0))$, $F\in C^{k-2,\eps}(Q_r(z_0))$ and $G\in C^{k-1,\eps}(Q_r(z_0))$. For any weak solution $f$ to
    \begin{align*}
        \partial_tf+v\cdot\nabla_xf-\ddiv(A\nabla_vf)=B\cdot\nabla_vf+F-\ddiv(G)\quad\text{in }Q_{r}(z_0),
    \end{align*}
    we have 
    \begin{align*}
        [f]_{C^{k,\eps}(Q_{r/2}(z_0))}\leq cr^{-(k+\eps)}\left(\|f\|_{L^\infty(Q_{r}(z_0))}+r^{k+\eps}[F]_{C^{k-2,\eps}(Q_r(z_0))}+r^{k+\eps}[G]_{C^{k-1,\eps}(Q_r(z_0))}\right)
    \end{align*}
    for some constant $c=c(n,\Lambda,\eps,\|A\|_{C^{k-1,\eps}(Q_r(z_0))},\|B\|_{C^{k-2,\eps}(Q_r(z_0))})$.
\end{lemma}

Next, we provide the H\"older regularity of the solution at $\gamma_+$.
\begin{lemma}\label{lem.bdrysch}
 Let $r > 0$ and $z_0 \in \R^{2n+1}$. Let $k\in\{1,2\}$ and $\eps\in(0,1)$. Let $A\in C^{k-1,\eps}(H_r(z_0))$, $B\in C^{k-2,\eps}(H_r(z_0))$, $F\in C^{k-2,\eps}(H_r(z_0))$ and $G\in C^{k-1,\eps}(H_r(z_0))$ with $x_{0,n}=0$, $v_{0,n}<0$ and $H_{r}(z_0)\cap\gamma_0=\emptyset$. Assume that $f$ is a weak solution to 
 \begin{align*}
     \partial_tf+v\cdot\nabla_xf-\ddiv(A\nabla_vf)=B\cdot\nabla_vf+F-\ddiv(G)\quad\text{in }H_{r}(z_0).
 \end{align*}
 Then we have 
    \begin{align*}
        [f]_{C^{k,\eps}(H_{r/2}(z_0))}\leq cr^{-(k+\eps)}\left(\|f\|_{L^\infty(H_{r}(z_0))}+r^{k+\eps}[F]_{C^{k-2,\eps}(H_r(z_0))}+r^{k+\eps}[G]_{C^{k-1,\eps}(H_r(z_0))}\right)
    \end{align*}
    for some constant $c=c(n,\Lambda,\eps,\|A\|_{C^{k-1,\eps}(H_r(z_0))},\|B\|_{C^{k-2,\eps}(H_r(z_0))})$.
\end{lemma}
\begin{proof}
     First note that the regularity at $\gamma_+$ of kinetic equations in  non-divergence form was established in \cite[Lemma 4.1]{RoWe25} by using a blow-up argument. The exact same argument also works for kinetic equations in divergence form as was explained in \cite[Lemma 6.1]{KiWe26} and also the term $\ddiv G$ can treated by the same approach (see \cite[Lemma 2.5]{KiWe26}). Note that all of these arguments also work for one-sided cylinders in $t$ without any major modifications.
\end{proof}

Next, we recall the expansion estimates at the grazing set $\gamma_0$ which were established in \cite{KiWe26}. Here we argue that they remain true in one-sided cylinders. 

Before stating the expansion estimates, we give three explicit solutions $\phi_0$, $\psi_0$ and $\phi_1$ to Kolmogorov's equation in the half-space in 1D, which were introduced and analyzed in \cite[Appendix C]{KiWe26}. As they appear in the expansion and they will also play an important role in Section \ref{sec:bdry-Harnack}, we collect some of their basic properties below. 

The functions $\phi_0 = \phi_0(x_n,v_n)$, $\psi_0=\psi_0(x_n,v_n)$ and $\phi_1 = \phi_1(x_n,v_n)$ are defined as follows:
    \begin{align*}
        \phi_0(x_n,v_n) &\coloneqq \begin{cases}
            x_n^{1/6}U\left(-\frac16,\frac23,-\frac{v_n^3}{9x_n}\right)&\quad\text{if }v_n<0,\\
            \frac{\Gamma(7/6)}{\Gamma(1/6)}x_n^{1/6}e^{-\frac{v_n^3}{9x_n}}U\left(\frac56,\frac23,\frac{v_n^3}{9x_n}\right)&\quad\text{if }v_n\geq0,
        \end{cases}
        \\ 
                \phi_1(x_n,v_n)&\coloneqq \begin{cases}
            -x_n^{7/6}U\left(-\frac76,\frac23,-\frac{v_n^3}{9x_n}\right)&\quad\text{if }v_n<0,\\
            -\frac{\Gamma(13/6)}{\Gamma(-5/6)}x^{7/6}e^{-\frac{v_n^3}{9x_n}}U\left(\frac{11}{6},\frac23,\frac{v_n^3}{9x_n}\right)&\quad\text{if }v_n\geq0,
        \end{cases}
        \\
         \psi_0(x_n,v_n) &\coloneqq x_n^{\frac{2}{3}}\left( M\left( - \frac{2}{3} , \frac{2}{3} , -\frac{v_n^3}{9x_n} \right) +\frac{\Gamma\left(\frac23\right)}{\Gamma\left(-\frac{2}3\right)}\frac{\Gamma\left(-\frac{1}3\right)}{\Gamma\left(\frac43\right)} \frac{v_n}{(9x_n)^{1/3}} M\left( - \frac{1}{3} , \frac{4}{3} ,-\frac{v_n^3}{9x_n} \right)\right) - a v_n^2,
    \end{align*}
for a suitable $a \in \R$ (determined in \cite[Lemma C.6]{KiWe26}), where $U$ denotes the Tricomi confluent hypergeometric function and $M$ denotes a  hypergeometric function of the first kind.

\begin{itemize}
\item The function $\psi_0=\psi_0(x_n,v_n)$ solves 
 \begin{equation}\label{eq.psi0}
\left\{
\begin{alignedat}{3}
v_n\partial_{x_n}\psi_0-\partial_{v_n,v_n}\psi_0&=1&&\qquad \mbox{in  $\{x_n>0\} \times \R^n$}, \\
\psi_0&=0&&\qquad  \mbox{in $\{x_n=0\}\times \{v_n>0\}$}.
\end{alignedat} \right.
\end{equation} 
\item The function $\partial_{v}\phi_1=\partial_{v_n}\phi_1(x_n,v_n)$ solves 
\begin{equation}\label{eq.dphi1}
\left\{
\begin{alignedat}{3}
v_n\partial_{x_n}(\partial_{v_n}\phi_1)-\partial_{v_n,v_n}(\partial_{v_n}\phi_1)&=m_0\phi_0&&\qquad \mbox{in  $\{x_n>0\} \times \R^n$}, \\
\partial_{v}\phi_1&=0&&\qquad  \mbox{in $\{x_n=0\}\times \{v_n>0\}$}
\end{alignedat} \right.
\end{equation} 
for some constant $m_0\neq0$. 
    \item The function $v_n\partial_{v_n}\phi_0$ solves 
   \begin{equation}\label{eq.vdphi0}
\left\{
\begin{alignedat}{3}
v_n\partial_{x_n}(v_n\partial_{v_n}\phi_0)-\partial_{v_n,v_n}(v_n\partial_{v_n}\phi_0)&=-3\partial_{v_n,v_n}\phi_0&&\qquad \mbox{in  $\{x_n>0\} \times \R^n$}, \\
v_n\partial_{v_n}\phi_0&=0&&\qquad  \mbox{in $\{x_n=0\}\times \{v_n>0\}$}.
\end{alignedat} \right.
\end{equation} 
\item We observe the homogeneity of the functions.
\begin{equation}\label{homogen}
    \begin{aligned}
       \phi_0(S_{r}z)=r^{\frac12}\phi_0(z),\quad\psi_0(S_rz)=r^2\psi_0(z),\quad \phi_1(S_rz)=r^{3+\frac12}\phi_1(z).
    \end{aligned}
\end{equation}
\item Combining the regularity results given in \autoref{lem.intsch} and \autoref{lem.bdrysch} and the homogeneity as in \eqref{homogen}, we derive for any  $k\in\{0,1,2\}$ and $\delta\in[0,1]$,
\begin{align}\label{psi0.dd}
    [\phi_0]_{C^{k,\delta}(H_{\rho_0}(z_0))}\leq c\rho_0^{\frac12-(k+\delta)},\quad[\psi_0]_{C^{k,\delta}(H_{\rho_0}(z_0))}\leq c\rho_0^{2-(k+\delta)},\quad [\partial_{v_n}\phi_1]_{C^{k,\delta}(H_{\rho_0}(z_0))}\leq c\rho_0^{\frac52-(k+\delta)}
\end{align}
with $\rho_0\coloneqq \frac{\max\{(x_{0,n})^{\frac13},|v_{0,n}|\}}4$, provided that $x_{0,n}\geq (v_{0,n})^3$.
\end{itemize}

From now on, for the convenience of notation, we always write
\begin{align}\label{notationvarphi}
    \varphi_0(z)\coloneq \phi_0(x_n,v_n),\ \varphi_1(z)\coloneq v_n\partial_{v_n}\phi_0(x_n,v_n),\, \varphi_2(z)\coloneq \psi_0(x_n,v_n),\, \varphi_3(z)\coloneq \partial_{v_n}\phi_1(x_n,v_n).
\end{align}
Lastly, we recall the vector space $\mathcal{T}_{\lambda,\eps}$ defined by
\begin{align*}
    \mathcal{T}_{\lambda,\eps}\coloneqq \left\{p_0{\varphi_0}+p_1{\varphi_1}+p_2{\varphi_2}+p_3{\varphi_3}\,\bigg|\begin{array}{l} p_0,p_1\in \cP_{[\lambda+\eps-\frac12]}\text{ with }p_1(0)=0,\\p_2\in P_{[\lambda+\eps-2]},\,p_3\in P_{[\lambda+\eps-\frac52]}
    \end{array}\right\}.
\end{align*}
We are now ready to give the lemma for the expansion estimates at $\gamma_0$ in one-sided cylinders.
\begin{lemma}\label{lem.hol12}
 Let $\lambda\in\N\cup\{0\}$ and $\eps\in(0,1)$ with $\lambda+\eps<3$, and $\lambda+\eps-\frac12\notin \N\cup\{0\}$. Let $A\in C^{\lambda-\frac12+\eps}(H_r)$, $B\in C^{\lambda-\frac32+\eps}(H_r)$ and $F\in C^{\lambda-2+\eps}(H_r)$ with $r\leq1$ and let $f$ be a weak solution to 
    \begin{equation*}
\left\{
\begin{alignedat}{3}
\partial_tf+v\cdot\nabla_xf-\ddiv({A}\nabla_ v{f})&=B\cdot\nabla_vf+F&&\qquad \mbox{in  $H_r$}, \\
{f}&=0&&\qquad  \mbox{in $\big(\{x_n=0\}\times \{v_n>0\}\big) \cap Q_r$}
\end{alignedat} \right.
\end{equation*} 
with $(A(0))_{n,n}=1$. Then there exists $P\in \mathcal{T}_{\lambda,\eps}$ such that for any $\rho\leq \frac{r}{2}$, 
\begin{align*}
    \|f-P\|_{L^\infty(H_\rho)}\leq c\left( \frac{\rho}{r}\right)^{\lambda+\eps}(\|f\|_{L^\infty(H_r)}+r^{\lambda+\eps}[F]_{C^{\lambda-2+\eps}(H_r)})
\end{align*}
for some constant $c=c(n,\Lambda,\eps,\|A\|_{C^{\lambda-\frac12+\eps}(H_r)},\|B\|_{C^{\lambda-\frac32+\eps}(H_r)})$.
\end{lemma}
\begin{proof}
    Note that in \cite[Proposition 5.1]{KiWe26}, we have proved the same estimate when the domain is given by a two-sided kinetic cylinder $\mathcal{H}_r=(-r^2,r^2)\times (B_{r^3}\cap\{x_n>0\})\times B_r$. Every argument used in \cite[Section 5]{KiWe26} also works for the one-sided kinetic cylinder $H_r$, and the only difference is that the domain of the limiting equation in \cite[(5.38)]{KiWe26} is replaced by $(-\infty,0]\times \{x_n>0\}\times \bbR^n$. Moreover, all the Liouville theorems given in \cite[Section 4]{KiWe26} hold for the domain $(-\infty,0]\times \{x_n>0\}\times \bbR^n$, as we only used global $C^\alpha$ estimates for kinetic equations and Liouville's theorem for stationary solutions given in \cite[Section 3]{KiWe26}. Therefore, without changing any lines of the proof of \cite[Proposition 5.1]{KiWe26}, we can derive expansion estimates when the domain is given by the one-sided cylinder $H_r$.
\end{proof}

We end this section by providing pointwise estimates of the gradient in $v$ of the solution up to the grazing set. This result is essentially a corollary of \autoref{lem.hol12} and the regularity estimates of solutions up to $\gamma_{\pm}$ (see \cite{Sil22,Zhu25} and \cite{RoWe25,KiWe26}).

\begin{lemma}\label{lem.bdry.gamma-}
   Let $r > 0$ and $\eps \in (0,1)$. Let $A\in C^\eps(H_r)$ and $B\in L^\infty(H_r)$. Let $f$ be a weak solution to 
    \begin{equation*}
\left\{
\begin{alignedat}{3}
\partial_tf+v\cdot\nabla_xf-\ddiv({A}\nabla_ v{f})&=B\cdot\nabla_vf&&\qquad \mbox{in  $H_r$}, \\
{f}&=0&&\qquad  \mbox{in $\big(\{x_n=0\}\times \{v_n>0\}\big) \cap H_r$}.
\end{alignedat} \right.
\end{equation*} 
Then for any $z_1\in H_{r/4}$, 
\begin{align}\label{goal.gra}
    |\nabla_vf(z_1)|\leq c\max\{(x_{1,n})^{\frac13},|v_{1,n}|\}^{-\frac12}r^{-\frac12}\|f\|_{L^\infty(H_r)}
\end{align}
for some constant $c=c(n,\Lambda,\eps,\|A\|_{C^\eps(H_r)},\|B\|_{L^\infty(H_r)})$.
\end{lemma}
\begin{proof}
We fix a point $z_1\in H_{r/4}$ and divide the proof into three cases.
\begin{itemize}
    \item Assume $64x_{1,n}\leq (v_{1,n})^3$. Then we have 
    \begin{align*}
        Q_{r_1}(z_1)\subset \R \times \{x_n>0\} \times \R^n,\quad Q_{r}(z_1)\subset H_{\overline{r}_1}(\overline{z}_1),\quad\text{and}\quad H_{\overline{r}_1}(\overline{z}_1)\subset H_{\rho_1}(\overline{z}_1)
    \end{align*}
   where $r_1\coloneqq \frac14\left(\frac{x_{1,n}}{v_{1,n}}\right)^{\frac12}$, $\overline{r}_1\coloneqq2(x_{1,n})^{\frac13}$,  $\overline{z}_1\coloneqq (t_1,x_{1}',0,v_{1})$ and $\rho_1\coloneqq \frac{v_{1,n}}{2}$. Note that the point $\overline{z}_1$ is a natural projection of the point $z_1$ onto $\gamma_-$.
    By a combination of \autoref{lem.intsch} and \autoref{lem.interpolhol}, we deduce
    \begin{align}\label{ineq1.gragamma-}
        |\nabla_vf(z_1)|\leq \|\nabla_vf\|_{L^\infty(Q_{r_1/2}(z_1))}\leq c r_1^{-1}\|f\|_{L^\infty(Q_{r_1}(z_1))}\leq cr_1^{-1}\|f\|_{L^\infty(H_{\overline{r}_1}(\overline{z}_1))}
    \end{align}
    for some $c=c(n,\Lambda,\eps,\|A\|_{C^\eps(H_r(z_0))},\|B\|_{L^\infty(H_r(z_0))})$. Next, we define
    \begin{align*}
        \widetilde{f}(z)\coloneqq f(S_{\rho_1}z),\quad \widetilde{A}(z)\coloneqq A(S_{\rho_1}z),\quad \widetilde{B}(z)\coloneqq \rho_1B(S_{\rho_1}z)
    \end{align*} to see that
    \begin{equation*}
\left\{
\begin{alignedat}{3}
\partial_t\widetilde{f}+v\cdot\nabla_xf-\ddiv(\widetilde{A}\nabla_ v\widetilde{f})&=\widetilde{B}\cdot\nabla_v\widetilde{f}&&\qquad \mbox{in  $H_1(\widetilde{z}_1)$}, \\
\widetilde{f}&=0&&\qquad  \mbox{in $\{x_n=0\}\times \{v_n>0\}$},
\end{alignedat} \right.
\end{equation*} 
where $\widetilde{z}_1\coloneqq (t_1/\rho_1^2,x_1'/\rho_1^3,0,v_1/\rho_1)$. Since 
$\{x_n>0\}$ is a convex domain and $v_{1}/\rho_1\cdot n_{x}=v_{1,n}/\rho_1=2$, by \cite[Theorem 1.3]{Sil22} with $q=\frac32$, we derive
\begin{align}\label{ineq2.gragamma-}
    \|\widetilde{f}\|_{L^\infty(H_{\overline{r}_1/\rho_1}(\widetilde{z}_1))}\leq c(\overline{r}_1/\rho_1)^{\frac32}\|\widetilde{f}\|_{L^\infty(H_{1}(\widetilde{z}_1))}=c(\overline{r}_1/\rho_1)^{\frac32}\|{f}\|_{L^\infty(H_{\rho_1}(\overline{z}_1))}
\end{align}
for some constant $c=c(n,\Lambda)$. Furthermore, by \autoref{lem.hol12}, there is a constant $c_0$ such that for any $\rho\leq r/4$,
\begin{align*}
    \|f-c_0{\varphi_0}\|_{L^\infty(H_{\rho}(z_0))}\leq c\left(\frac{\rho}{r}\right)^{\frac12}\|f\|_{L^\infty(H_{r/2}(z_0))}
\end{align*}
for some constant $c=c(n,\Lambda,\eps,\|A\|_{C^\eps(H_r)},\|B\|_{L^\infty(H_r)})$, where we write $z_0\coloneqq (t_1,x_1',0,v_1',0)$ that is a natural projection of the point $\overline{z}_1$ onto $\gamma_0$. In particular, we deduce that $|c_0|\leq r^{-\frac12}{\|f\|_{L^\infty(H_{r/2}(z_0))}}$. Using this and the fact that $H_{\rho_1}(\overline{z}_1)\subset H_{4\rho_1}(z_0)$, we deduce
\begin{equation}\label{ineq3.gragamma-}
\begin{aligned}
    \|f\|_{L^\infty(H_{\rho_1}(\overline{z}_1))}\leq \|f\|_{L^\infty(H_{4\rho_1}({z}_0))}&\leq  \|f-c_0{\varphi_0}\|_{L^\infty(H_{3\rho_1}({z}_0))}+c_0\|{\varphi_0}\|_{L^\infty(H_{4\rho_1}({z}_0))}\\
    &\leq c({\rho_1}/{r})^{\frac12}\|f\|_{L^\infty(H_{r/2}(z_0))},
\end{aligned}
\end{equation}
where $c=c(n,\Lambda,\eps,\|A\|_{C^\eps(H_r)},\|B\|_{L^\infty(H_r)})$ and we have used the fact that $\|{\varphi_0}\|_{L^\infty(H_{4\rho_1}({z}_0))}\leq c\rho_1^{\frac12}$ by \eqref{psi0.dd}.
A combination of \eqref{ineq1.gragamma-}, \eqref{ineq2.gragamma-} and \eqref{ineq3.gragamma-} leads to 
\begin{align*}
    |\nabla_vf(z_1)|\leq c(rv_{1,n} )^{-\frac12}\|f\|_{L^\infty(H_r)},
\end{align*}
which verifies \eqref{goal.gra}.
\item Assume $64x_{1,n}\geq |v_{1,n}|^3$. Then we have 
\begin{align*}
    Q_{r_1}(z_1)\subset \{x_n>0\},\quad Q_{r_1}(z_1)\subset H_{32r_1}(z_0)\subset H_{r/2}({z}_0),
\end{align*}
where $r_1\coloneqq\frac{(x_{1,n})^{\frac13}}{16}$ and ${z}_0\coloneqq (t_1,x_1',0,v_1',0)$. Therefore, using interior estimates as in \eqref{ineq1.gragamma-} and boundary estimates at $\gamma_0$ as in \eqref{ineq3.gragamma-}, we deduce
\begin{align*}
    |\nabla_vf(z_1)|\leq cr_1^{-1}\|f\|_{L^\infty(Q_{r_1}(z_1))}\leq cr_1^{-1}\|f\|_{L^\infty(Q_{32r_1}(z_0))}\leq c(r_1 r)^{-\frac12}\|f\|_{L^\infty(H_{r/2}(z_0))},
\end{align*}
for some constant $c=c(n,\Lambda,\eps,\|A\|_{C^\eps(H_r)},\|B\|_{L^\infty(H_r)})$, which implies \eqref{goal.gra}.
\item Assume $64x_{1,n}\leq (-v_{1,n})^3$. Then we have 
    \begin{align*}
        Q_{r_1}(z_1)\subset \{x_n>0\},\quad Q_{r}(z_1)\subset H_{\overline{r}_1}(\overline{z}_1),\quad H_{\overline{r}_1}(\overline{z}_1)\subset H_{\rho_1}(\overline{z}_1)\subset H_{4\rho_1}(z_0)\subset H_{r/2}(z_0)
    \end{align*}
   where $r_1\coloneqq \frac{(x_{1,n})^{\frac13}}{16}$, $\overline{r}_1\coloneqq2(x_{1,n})^{\frac13}$,  $\overline{z}_1\coloneqq (t_1,x_{1}',0,v_{1})$, $\rho_1\coloneqq \frac{v_{1,n}}{2}$ and $z_0=(t_1,x_1',0,v_1',0)$. Similarly, by combining interior estimates as in \eqref{ineq1.gragamma-}, boundary estimates at $\gamma_+$ as in \eqref{ineq2.gragamma-} together with \autoref{lem.bdrysch} and boundary estimates at $\gamma_0$ as in \eqref{ineq3.gragamma-}, we have 
   \begin{align*}
       |\nabla_vf(z_1)|\leq cr_1^{-1}\|f-f(\overline{z}_1)\|_{L^\infty(Q_{r_1}(z_1))}\leq c\rho_1^{-1}\|f\|_{L^\infty(H_{\rho_1}(\overline{z}_1))}\leq c\rho_1^{-\frac12}r^{-\frac12}\|f\|_{L^\infty(H_{r/2}(z_0))},
   \end{align*}
   where $c=c(n,\Lambda,\eps,\|A\|_{C^\eps(H_r)},\|B\|_{L^\infty(H_r)})$. Indeed, we have applied interior estimates to the function $f(z)-f(\overline{z}_1)$ and we have used the fact that $\|f-f(\overline{z}_1)\|_{L^\infty(Q_{r_1}(z_1))}\leq cr_1[f]_{C^{0,1}(H_{\overline{r}_1}(\overline{z}_1))}$.
\end{itemize}
Since we have proved \eqref{goal.gra} for each case, this completes the proof.
\end{proof}

\section{Propagation of positivity in kinetic geometry}
\label{sec:propagation}
In this section, we prove the propagation of positivity for kinetic equations via Harnack chains. In particular, we will show that the positivity of the solution $f$ in a kinetic cylinder at time $t=0$ follows from its positivity at a suitable past time (see \autoref{lem.exp.posright}). It turns out that this past time cannot be chosen arbitrarily. This stands in stark contrast with the classical heat equation, where positivity at some time level $t=t_0$ implies positivity of the solution for all later times $t>t_0$ for all points in space.

Such a property fails in the kinetic setting. For example, there is a solution $f$ in $Q_{1}$ such that $f(-1/2,1/2,0)>0$, but $f(-1/4,0,0)=0$ (see \cite[Proposition 4.5]{CNP10}). Moreover, due to the kinetic geometry, we note in particular that propagation in the $x$-direction, from a point $z$ depends on its velocity variable $v$. 
Therefore, proving \autoref{lem.exp.posright} requires a careful investigation. First, we determine how positivity propagates through Harnack inequalities (see \autoref{lem.auxihar}), and then, in a second step, we inductively construct paths from a given point to the kinetic cylinder where we aim to establish the positivity of the solution.

First, we recall Harnack's inequality which was established in \cite[Theorem 5]{GuMo22}. Here, we state it in an appropriately rescaled form.

\begin{lemma}\label{lem.har}
Let $f\geq0$ be a weak solution to
\begin{align*}
    \partial_tf+v\cdot\nabla_xf-\ddiv(A\nabla_vf)=B\cdot\nabla_vf+F\quad\text{in }Q_1.
\end{align*}
Then there are universal small constants $r_0$ and $\eps_0$ such that for any $z \in \R^{2n+1}$ and $\rho>0$,
    \begin{align*}
        \sup_{Q_{r_0}}f(z_1\circ S_\rho z)\leq c\left(\inf_{Q_{r_0}(\eps_0,0,0)}f(z_1\circ S_\rho z)+\rho^2\|F(z_1\circ S_{\rho}z)\|_{L^\infty(Q_{1}(\eps_0,0,0))}\right)
    \end{align*}
    for some constant $c=c(n,\Lambda)$, whenever $Q_{\rho}(z_1\circ(\rho^2\eps_0,0,0))\subset Q_1$. The two constants $r_0, \eps_0$ can be determined explicitly as follows: $r_0=\frac{1}{80}$ and $\eps_0=\frac{19}{8}\frac{1}{400}$.
\end{lemma}

We next introduce an auxiliary lemma, which will be used in the construction of a Harnack chain. Given a set $D$, we describe the propagation set associated with each point $z\in D$ via Harnack's inequality. 
\begin{lemma}\label{lem.auxihar}
    Let $r_0$ and $\eps_0$ be the universal constants determined in \autoref{lem.har}. Let $\rho\in(0,1)$, $r_1,r_2,r_3\geq0$, $a\leq b$, $0<d_1\leq d_2$ and a set  $\widetilde{\mathcal{I}}\times B'_{r_1}\times [d_1,d_2]\times B'_{r_2}\times [-r_3,r_3]$ with $\widetilde{\mathcal{I}}\coloneqq [a,b]$ be given. Suppose
    \begin{align*}
        &(t,x',x_n,v',v_n)\in \mathcal{I}\times \mathcal{B} \times B'_{r_2+\rho r_0}\times [-r_3-\rho r_0,r_3+\rho r_0],\\
        &\mathcal{I}\coloneqq [a+\rho^2(\eps_0-r_0^2),b+\rho^2\eps_0],\\
        &\mathcal{B}\coloneq B'_{r_1-\eps_0\rho^2r_2+(\rho r_0)^3}\times [d_1+\eps_0\rho^2 r_3-(\rho r_0)^3,d_2-\eps_0\rho^2r_3+(\rho r_0)^3],
    \end{align*}
    where we further assume that $r_1-\eps_0\rho^2r_2+(\rho r_0)^3>0$ and $d_1<d_2-2\eps_0\rho^2r_3+2(\rho r_0)^3$, if $r_2>0$ and $r_3>0$, respectively.
    Then there is a point 
    \begin{align*}
        &(t_1,x_{1}',x_{1,n},v_{1}',v_{1,n})\in \widetilde{\mathcal{I}}\times B'_{r_1}\times [d_1,d_2]\times B'_{r_2}\times [-r_3,r_3]
    \end{align*}
    such that 
    \begin{align*}
        (t,x,v)\in Q_{\rho r_0}((t_1,x_1,v_1)\circ (\rho^2\eps_0,0,0)).
    \end{align*}
\end{lemma}
\begin{proof}
   Let us choose $(t_1,x'_1,x_{1,n},v'_1,v_{1,n})$ by
    \begin{align*}
        t_1\coloneqq \begin{cases}
            a&\quad\text{if }t<a+\rho^2\eps_0,\\
            t-\rho^2\eps_0&\quad\text{if } t\geq a+\rho^2\eps_0,
        \end{cases}
    \end{align*}
    \begin{align*}
        v_1'\coloneqq \begin{cases}
            0&\quad\text{if }|v'|< \rho r_0,\\
            v'-\rho r_0(\frac{v'}{\|v'\|})&\quad\text{if }|v'|\geq\rho r_0,
        \end{cases}
    \end{align*}
    \begin{align*}
        v_{1,n}\coloneqq \begin{cases}
            0&\quad\text{if }|v_n|< \rho r_0,\\
            v_n-\rho r_0(\frac{v_n}{\|v_n\|})&\quad\text{if }|v_n|\geq\rho r_0,
        \end{cases}
    \end{align*}
    \begin{align*}
        x'_1\coloneqq \begin{cases}
        0&\quad\text{if }|x'-(t-t_1)v_1|< (\rho r_0)^3,\\
            x'-(t-t_1)v'_1-(\rho r_0)^3\frac{x'-(t-t_1)v'_1}{|x'-(t-t_1)v'_1|}&\quad\text{if } |x'-(t-t_1)v_1|\geq (\rho r_0)^3,
        \end{cases} 
    \end{align*}
    and 
    \begin{align*}
       x_{1,n}\coloneqq \begin{cases}
        d_1&\quad\text{if }x_n\leq d_1+(t-t_1)v_{1,n}+(\rho r_0)^3,\\
            x_n-(t-t_1)v_{1,n}-(\rho r_0)^3\frac{x_n-(t-t_1)v_{1,n}}{|x_n-(t-t_1)v_{1,n}|}&\quad\text{if } x_n> d_1+(t-t_1)v_{1,n}+(\rho r_0)^3.
        \end{cases} 
    \end{align*}
    Then we have $z\in Q_{\rho r_0}(z_1\circ (\rho^2\eps_0,0,0))$ and $z_1\in[a,b]\times B'_{r_1}\times [d_1,d_2]\times B'_{r_2}\times [-r_3,r_3]$, where we have used the fact that $|t-t_1||v'|\leq \eps_0\rho^2r_2$ and $|t-t_1||v_{1,n}|\leq \eps_0\rho^2r_3$. This completes the proof.
\end{proof}

Before stating the main lemma, given $\varsigma\in(0,1/2)$ and $r>0$, we introduce a kinetic cylinder $H_{r,\varsigma}^+$ that stays away from $\{x_n=0\}$ by
    \begin{align}\label{defn.cylinder+}
    H_{r,\varsigma}^+&\coloneqq  \{(t,x',x_n,v',v_n)\in H_r\,:\, (\varsigma r)^3\leq x_n\leq r^3 \} .
\end{align}
In particular, we write for any $t_0\in\bbR$,
\begin{align*}
    H_{r,\varsigma}^+(t_0,0,0)\coloneqq \{(t_0,0,0)\circ z\,:\,z\in H^+_{r,\varsigma}\}.
\end{align*}
Note that such cylinders will also be used to prove the Hopf lemma given in the next section.

Now, we are ready to prove the main result in this section.  More precisely, when $f$ is positive at some point, then its positivity propagates after a certain time gap. First we prove propagation only in the $x$-variables, while fixing the $v$-variables at $v=0$, which allows us to ignore the kinetic geometry. We then investigate propagation in the $v$-direction, which necessarily affects the propagation set in the $x$-direction because of the kinetic transport. However, \autoref{lem.auxihar} shows that when the cylinder on which we will apply Harnack's inequality has a sufficiently small radius $\rho$, the displacement induced by the $v$-direction also stays small (see \textbf{Step 2} and \textbf{Step 4} in the proof of \autoref{lem.exp.posright} for more details).

\begin{lemma}\label{lem.exp.posright}
Let $r\in(0,1]$. For any $\delta\in(0,1]$, there is a constant $M=M(1/\delta)\geq1$ such that for any weak solution $f\geq0$ to
    \begin{align}\label{eq.har}
        \partial_tf+v\cdot\nabla_xf-\ddiv(A\nabla_vf)=B\cdot \nabla_vf+F\quad\text{in }I_{2rM}\times\mathrm{H}_{r(1+\delta)} 
    \end{align}
    with $f(t_0,x_0',x_{0,n},v_0)=f(-(rM)^2,0,r^3,0)=1$ and for any $\varsigma\in(0,\frac14)$, if 
    \begin{align}\label{F.ass.pos.away}
        \|r^2F\|_{L^\infty(I_{2rM}\times\mathrm{H}_{r(1+\delta)} )}\leq 1/\mathcal{M}
    \end{align}
    for some constant $\mathcal{M}=\mathcal{M}(n,\Lambda,1/\delta,\varsigma)\geq1$, then for any $a\in[-\frac{r^2}{16},0]$,
    \begin{align}\label{pos.away}
        \inf_{H^+_{\frac{r}2,{\varsigma}}(a,0,0)}f\geq \frac1c,
    \end{align}
    where $c=c(n,\Lambda,\varsigma,1/\delta)$. 
\end{lemma}

Note that when we apply this lemma later in \autoref{lem.hop.gen0} and \autoref{lem.hop.gen}, we will choose $\delta=1$ or $\delta=\frac1c$ for some constant $c=c(n,\Lambda)$, which has a uniform lower bound.

\begin{proof}
First, note that for any $r\in(0,1]$, the rescaled function $\widetilde{f}(z)\coloneqq f(S_rz)$ solves
\begin{align*}
    \partial_t\widetilde{f}+v\cdot\nabla_x\widetilde{f}-\ddiv(\widetilde{A}\nabla_v\widetilde{f})=\widetilde{B}\cdot\nabla_v\widetilde{f}+\widetilde{F}\quad\text{in }I_{2M}\times H_{1+\delta},
\end{align*}
where $\widetilde{A}(z)\coloneqq A(S_rz)$, $\widetilde{B}(z)\coloneqq rB(S_rz)$, and $\widetilde{F}(z)\coloneqq r^2F(S_rz)$. Therefore, it suffices to prove the result in case $r=1$.

Fix two constants $M\geq1$ and $\mathcal{M}\geq1$, which will be determined at the end of the proof (see \eqref{choi.M} and \eqref{choi.mathcalM}). Note from \autoref{lem.har} that there are universal small constants $r_0$ and $\eps_0$ such that for any $\rho>0$,
    \begin{align}\label{harnack.pre}
        \sup_{Q_{\rho r_0}(z_1)}f(z)\leq C_0\left(\inf_{Q_{\rho r_0}(z_1\circ(\rho^2\eps_0,0,0))}f(z)+\rho^2\|F\|_{L^\infty(Q_\rho(z_1\circ(\rho^2\eps_0,0,0)))}\right)
    \end{align}
    for some constant $C_0=C_0(n,\Lambda)$, whenever $Q_{\rho}(z_1\circ(\rho^2\eps_0,0,0))\subset I_{2M}\times \mathrm{H}_{1+\delta}$. 
    
    We may assume $\delta\leq r_0$, as $f$ is also a solution to \eqref{eq.har} in $I_{2M}\times \mathrm{H}_{1+r_0}$, when $\delta>r_0$. As we will see, this is property would be enough to conclude the proof.

To prove the result, we will apply \eqref{harnack.pre} in sequences of cylinders that are shifted consecutively in $x'$, $v'$, $x_n$, and $v_n$. Now we split the proof into four steps.

    \textbf{Step 1.}  First, we are going to propagate positivity in $x'$ by moving forward in time. We prove that for any $k\in\N$ and $(t,x')\in \mathcal{A}_{k}$,
    \begin{align}\label{ind.x'}
        \frac1{C_0^k}\leq f(t,x',1,0)+\sum_{i=0}^{k-1}\frac{1}{C_0^i\mathcal{M}},
    \end{align}
    where $c=c(n,\Lambda)$ and
    \begin{align*}
        \mathcal{A}_{k}\coloneqq [-M^2+k\delta^2(\eps_0-r_0^2),-M^2+k\delta^2\eps_0]\times B'_{k(\delta r_0)^3}.
    \end{align*}
    The case $k=1$ follows directly from \eqref{harnack.pre} with $\rho=\delta$ and $z_1=(-M^2,0,1,0)\in \bbR\times \bbR^{n-1}\times \bbR\times \bbR^n$. Due to the choice of $\rho$, we have $Q_{\rho}(-M^2,0,1,0)\subset I_{2M}\times \mathrm{H}_{1+\delta}$. When \eqref{ind.x'} with $k=l$ holds, we use \eqref{harnack.pre} with $\rho=\delta$ and $z_1=(t_1,x_1',1,0)$, where  $(t_1,x_1')\in \mathcal{A}_{l}$ to see that \eqref{ind.x'} with $k=l+1$, as for any $(t,x')\in \mathcal{A}_{l+1}$, there is a point $(t_1,x_1')\in \mathcal{A}_{l}$ by \autoref{lem.auxihar} with $a=-M^2+l\delta^2(\eps_0-r_0^2)$, $b=-M^2+l\delta^2\eps_0$, $\rho=\delta$, $r_1=l(\delta r_0)^3$ $r_2=r_3=0 $, $d_1=d_2=1$. Thus, we have inductively proved \eqref{ind.x'} for any $k \in \N$.

Note that at the end of the proof (see \eqref{k.choi}) we will choose the positive integer $k=k(1/\delta)$ to be sufficiently large.

\textbf{Step 2.} Next, using \eqref{ind.x'} and \eqref{harnack.pre}, we will show for any $m\in\N$ and $(t,x',v')\in \mathcal{A}_{k,m}'$,
    \begin{align}\label{ind.v'}
        \frac{1}{C_0^{k+m}}\leq f(t,x',1,v',0)+\sum_{i=0}^{k-1+m}\frac{1}{C_0^i\mathcal{M}},
    \end{align}
    where
    \begin{align*}
        \mathcal{A}'_{k,m}\coloneqq [-M^2+(k+m)\delta^2(\eps_0-r_0^2),-M^2+(k+m)\delta^2\eps_0]\times B'_{r_0\delta^3(kr_0^2-\eps_0\frac{(m-1)m}{2})}\times B_{m\delta r_0}',
    \end{align*}
    i.e. we spread positivity further in the $v'$ direction.
    First, using \eqref{harnack.pre} with $\rho=\delta$ and $z_1=(t,x',1,0)$, where $(t,x')\in \mathcal{A}_{k}$, we have \eqref{ind.v'} for $m=1$. Then we assume \eqref{ind.v'} with $m=l$. Before applying \eqref{harnack.pre}, we note that for any $(t,x',v')\in\mathcal{A}'_{k,l+1}$, there is a point $(t_1,x_1',v_1')\in \mathcal{A}'_{k,l}$ such that 
    \begin{align}\label{inc.cylinder}
        (t,x',1,v',0)\in Q_{\delta r_0}((t_1,x_1',1,v_1',0)\circ (\delta^2\eps_0,0,0)),
    \end{align}
    by \autoref{lem.auxihar} with $\rho=\delta$, $a=-M^2+(k+l)\delta^2(\eps_0-r_0^2)$, $b=-M^2+(k+l)\delta^2\eps_0$, $r_1=r_0\delta^3(kr_0^2-\eps_0\frac{(l-1)l}{2})$, $d_1=d_2=1$, $r_2=l\delta r_0$. 
    Therefore, for any $(t,x',v')\in\mathcal{A}'_{k,l+1}$, we employ \eqref{harnack.pre} with  $\rho=\delta$ and $z_1=(t_1,x_1',1,v_1',0)$, where the point $(t_1,x_1',v_1')\in \mathcal{A}'_{k,l}$ is determined in \eqref{inc.cylinder} so that
    \eqref{ind.v'} is true for $m=l+1$. This proves \eqref{ind.v'}.
    
    To end this step, we finally choose the integer $m=m(1/\delta)$ large enough so that 
    \begin{equation}\label{m.choi}
        1\leq m\delta r_0<2.
    \end{equation}
    As $r_0=\frac1{80}$ by \autoref{lem.har}, $m$ only depends on $1/\delta$.
    This gives that $\eps_0m^2\delta^3r_0\leq \eps_0(m\delta r_0)^2\leq 4\eps_0\leq 1$, as we assumed $\delta\leq r_0$. Thus, we have
    \begin{align}\label{last.v'}
        \frac{1}{C_0^{k+m}}\leq f(t,x',1,v',0)+\sum_{i=0}^{k-1+m}\frac{1}{C_0^i\mathcal{M}},
    \end{align}
    where
    \begin{align}\label{set.akm}
        \mathcal{A}_{k,m}\coloneqq [-M^2+(k+m)\delta^2(\eps_0-r_0^2),-M^2+(k+m)\delta^2\eps_0]\times B'_{k(\delta r_0)^3-1}\times B_{1}'.
    \end{align}

\textbf{Step 3.}
In this step we propagate positivity in the $x_n$-variable from $\{x_n = 1\}$ towards $\{ x_n = 0 \}$. We will do so in two steps. First, we have to apply the Harnack inequality to kinetic cylinders of radius $\delta$, in order to move away from the boundary of the solution domain $\{x_n=(1+\delta)^3\}$. Once we have left the boundary scale, in as second step, we move further by taking cylinders of radius $|x_{n}|^{\frac13}$. This choice guarantees that we stay away from the boundary $\{x_n=0\}$ when $x_n$ is close to $0$. We will divide the proof into two cases. 

\textbf{Step 3-1.} First, we prove the lower bound when $x_n$ is close to $x_n=(1+\delta)^3$.
We will prove for any $\mathbf{n}\in\N$ and $(t,x',x_n,v')\in \mathcal{A}^1_{k,m,\mathbf{n}}$
\begin{align}\label{ind.x_n}
    \frac{1}{C_0^{k+m+\mathbf{n}}}\leq f(t,x',x_n,v',0)+\sum_{i=0}^{k-1+m+\mathbf{n}}\frac{1}{C_0^i\mathcal{M}},
\end{align}
where 
\begin{align*}
    \mathcal{A}^1_{k,m,\mathbf{n}}&\coloneqq \mathcal{I}^1_{k,m,\mathbf{n}}\times B'_{k(\delta r_0)^3-1-\mathbf{n}\delta^2\eps_0}\times [1-\mathbf{n}(\delta r_0)^3,1]\times B_{1}',\\
    \mathcal{I}^1_{k,m,\mathbf{n}}&\coloneqq [-M^2+((k+m)\delta^2+\mathbf{n}\delta^2)(\eps_0-r_0^2),-M^2+((k+m)\delta^2+\mathbf{n}\delta^2)\eps_0].
\end{align*}
First, we observe from \autoref{lem.auxihar} with $\rho=\delta$, $\widetilde{\mathcal{I}}=\mathcal{I}^1_{k,m,\mathbf{n}}$, $r_1=k(\delta r_0)^3-1-\mathbf{n}\delta^2\eps_0$, $d_1=1-\mathbf{n}(\delta r_0)^3$, $d_2=1$, $r_2=1$, and $r_3=0$ that for any $(t,x',x_n,v')\in \mathcal{A}^1_{k,m,\mathbf{n}}$, there is a point $(t_1,x_1',x_{1,n},v')\in \mathcal{A}^1_{k,m,\mathbf{n}-1}$ such that 
\begin{align*}
    (t,x',x_{n},v',0)\in Q_{\delta r_0}((t_1,x_1',x_{1,n},v',0)\circ (\delta^2 \eps_0,0,0)).
\end{align*}
In addition, we note $Q_{\delta}((t_1,x_{1}',x_{1,n},v',0)\circ (\delta^2\eps_0,0,0))\subset I_{2M}\times \mathrm{H}_{1+\delta}$.
Based on this observation, the estimate \eqref{ind.x_n} with $\mathbf{n}=1$ is obtained by a combination of \eqref{ind.v'} and \eqref{harnack.pre} with $\rho=\delta$ and $z_1=(t_1,x_1',1,v_1',0)$ for any $(t_1,x_1',v_1')\in \mathcal{A}_{k,m}$.  Now, we assume \eqref{ind.x_n} with $\mathbf{n}=l$ for some positive integer $l$. Then we do the same procedure with $z_1=(t_1,x_1',x_{1,n},v_1',0)$ for any $(t_1,x_1',x_{1,n},v_1')\in \mathcal{A}^1_{k,m,l}$, so we get \eqref{ind.x_n} when $\mathbf{n}=l+1$. This proves  \eqref{ind.x_n}.

To conclude this step, we let $\mathbf{n}_1$ be the positive integer $\mathbf{n}_1=\mathbf{n}_1(1/\delta)$ such that
\begin{align}\label{choi.d1}
    \frac14<\mathbf{n}_1(\delta r_0)^3\leq\frac12.
\end{align} 
Thus, we have proved 
\begin{align*}
    \frac{1}{C_0^{k+m+\mathbf{n}_1}}\leq f(t,x',x_n,v',0)+\sum_{i=0}^{k-1+m+\mathbf{n}_1}\frac{1}{C_0^i\mathcal{M}} \quad\text{on }\mathcal{I}^1_{k,m,\mathbf{n}_1}\times B'_{k(\delta r_0)^3-1-\mathbf{n}_1}\times \left[\frac{3}{4},\frac78\right]\times B_{1}'.
\end{align*}

\textbf{Step 3-2.} Next, we show that for any $\mathbf{n}\in\N$
\begin{align}\label{ind.x_n2}
    \frac{1}{C_0^{k+m+\mathbf{n}_1+\mathbf{n}}}\leq f(t,x',x_n,v',0)+\sum_{i=0}^{k-1+m+\mathbf{n}_1+\mathbf{n}}\frac{1}{C_0^i\mathcal{M}}\quad\text{in }\mathcal{A}^0_{k,m,\mathbf{n}},
\end{align}
where 
\begin{align*}
    \mathcal{A}^0_{k,m,\mathbf{n}}&\coloneqq \mathcal{I}^0_{k,m,\mathbf{n}}\times B'_{k(\delta r_0)^3-1-\mathbf{n}_1-\eps_0\widetilde{\rho}_{\mathbf{n}}}\times \left[(\rho_\mathbf{n})^3,7/8\right]\times B_{1}',\\
    \mathcal{I}^0_{k,m,\mathbf{n}}&\coloneqq\left [-M^2+(\eps_0-r_0^2)\left((k+m+\mathbf{n}_1)\delta^2+\widetilde{\rho}_\mathbf{n}\right),-M^2+\eps_0\left((k+m+\mathbf{n}_1)\delta^2+\widetilde{\rho}_\mathbf{n}\right)\right]
\end{align*}
with
\begin{align*}
    \rho_\mathbf{n}\coloneqq \left(\frac3{4}(1-r_0^3)^\mathbf{n}\right)^{\frac13}\quad\text{and}\quad\widetilde{\rho}_\mathbf{n}\coloneqq \sum\limits_{\mathbf{i}=0}^{\mathbf{n}-1}\rho_{\mathbf{i}}^2.
\end{align*}
First, note that for any $(t,x',x_n,v')\in \mathcal{A}^0_{k,m,\mathbf{n}}$, there is  $(t_1,x_1',x_{1,n},v_1')\in \mathcal{A}^0_{k,m,\mathbf{n}-1}$ such that
\begin{align*}
    (t,x',x_n,v',0)\in Q_{\rho_{\mathbf{n}-1} r_0}((t_1,x_1',x_{1,n},v_1',0)\circ ((\rho_{\mathbf{n}-1})^2\eps_0,0,0))
\end{align*}and $ Q_{\rho_{\mathbf{n}-1}}((t_1,x_1',x_{1,n},v_1',0)\circ ((\rho_{\mathbf{n}-1})^2\eps_0,0,0))\subset I_{2M}\times \mathrm{H}_{1+\delta}$, which follows from \autoref{lem.auxihar} with $\rho=\rho_{\mathbf{n}-1}$, $\widetilde{\mathcal{I}}=\mathcal{I}^0_{k,m,\mathbf{n}}$, $r_1=k(\delta r_0)^3-2-\mathbf{n}_1-\mathbf{n}$, $d_1=(\rho_{\mathbf{n}-1})^3$, $d_2=7/8$, $r_2=1$, and $r_3=0$. Therefore, as in \eqref{ind.x_n},
we deduce \eqref{ind.x_n2} with $\mathbf{n}=l+1$ from a combination of \eqref{harnack.pre} with $\rho=\rho_l$ and \eqref{ind.x_n2} with $\mathbf{n}=l$. This proves \eqref{ind.x_n2}.

To conclude this step, we apply \eqref{ind.x_n2} with $\mathbf{n}$ being the positive integer $\mathbf{n}=\mathbf{n}(\varsigma)$ such that
\begin{align}\label{choi.d}
    (\rho_\mathbf{n})^3=\frac3{4}(1-r_0^3)^\mathbf{n}\leq \left(\frac{\varsigma}{4}\right)^3.
\end{align}
Therefore, we have proved
\begin{align}\label{last.tx'xnv'}
    \frac{1}{C_0^k}\frac{1}{c_0}\leq f(t,x',x_n,v',0)+\sum_{i=0}^{k-1+m+\mathbf{n}_1+\mathbf{n}}\frac{1}{C_0^i\mathcal{M}}\quad\text{in }\mathcal{A},
\end{align}
with
\begin{align}\label{const,R0R1}
    \mathcal{A}=[-M^2+(\eps_0-r_0^2)(k\delta^2+R_0),-M^2+\eps_0 (k\delta^2+R_0)]\times B'_{k(\delta r_0)^3-R_1}\times [(\varsigma/4)^3,7/8]\times B_1'
\end{align}
for some constants $R_0,R_1$ depending only on $1/\delta$ and some constant   $c_0=c_0(n,\Lambda,\varsigma,1/\delta)$, where we have used \eqref{m.choi}, \eqref{choi.d1} and the fact that $\delta\leq r_0$. To be more precise, we have chosen
\begin{align*}
    R_0 = (m + \mathbf{n}_1) \delta^2 + \widetilde{\rho}_\mathbf{n}, \quad R_1 = 1 + \mathbf{n}_1, \quad c_0 = C_0^{m + \mathbf{n}_1 + \mathbf{n}}.
\end{align*}
The integer $\mathbf{n}$ depends on $\varsigma$ and 
\begin{align*}
    \widetilde{\rho}_\mathbf{n}=\sum_{\mathbf{i}=0}^{\mathbf{n}-1}(\rho_\mathbf{i})^2\leq \left(\frac{3}4\right)^{\frac23}\frac{1}{1-(1-r_0^3)^{\frac23}},
\end{align*}
which has a universal upper bound independent of $\mathbf{n}$.

\textbf{Step 4.}
We now apply the Harnack inequality once more in order to propagate positivity in the $v_n$-variable. Note that we choose cylinders of radius $\varrho\leq (\frac{\varsigma}{4})^3$, which allows us to control the influence of the velocity variables along the Harnack chain. As a consequence, the positivity of $f$ will be preserved in the region $x_n\in [(\varsigma/2)^3,3/4]$.

Now, using \eqref{last.tx'xnv'}, we will prove for any $j\in\N$,
\begin{align}\label{ind.vn}
    \frac{1}{C_0^{k+j}}\frac1{c_0}\leq f(z)+\sum_{i=0}^{k-1+m+\mathbf{n}_1+\mathbf{n}+j}\frac{1}{C_0^i\mathcal{M}},
\end{align}
where 
\begin{align*}
    \mathcal{B}_j&\coloneqq \mathcal{I}_j\times B'_{k(\delta r_0)^3-R_1-j(\varrho/4)^2\eps_0}\times \left[(\varsigma/4)^3+\widetilde{r}_j,7/8-\widetilde{r}_j\right]\times B_1'\times [-j(\varrho r_0/4),j(\varrho r_0/4)],\\
    \mathcal{I}_j&\coloneqq[-M^2+(\eps_0-r_0^2)(k\delta^2+R_0+j(\varrho/4)^2),-M^2+\eps_0(k\delta^2+R_0+j(\varrho/4)^2)],\\
    \widetilde{r}_j&\coloneqq \eps_0\frac{j(j+1)}2(\varrho /4)^3r_0
\end{align*}
with $\varrho\coloneqq\min\{r_0(\varsigma/4)^3,1/2^{10}\}$.
To do so, we observe that for any $z\in \mathcal{B}_j$, there is a point $z_1\in \mathcal{B}_{j-1}$ such that
\begin{align*}
    z\in {Q}_{\varrho r_0/4}(z_1\circ ((\varrho/4)^2\eps,0,0))\quad\text{and}\quad {Q}_{\varrho /4}(z_1\circ ((\varrho/4)^2\eps,0,0))\subset I_{2M}\times \mathrm{H}_{1+\delta},
\end{align*}
by \autoref{lem.auxihar} with $\rho=\varrho/4$,
$\widetilde{\mathcal{I}}=\mathcal{I}_{j-1}$, $r_1=k(\delta r_0)^3-R_1-(j-1)(\varrho/4)^2(\eps_0-r_0^2)$, $d_1=(\varsigma/4)^3+\widetilde{r}_{j-1}$ $d_2=7/8-\widetilde{r}_{j-1}$, $r_2=1$ and $r_3=(j-1)(\varrho r_0/4)$.
Therefore, by following the same procedure as in \textbf{Step 3}, we deduce \eqref{ind.vn}. 

Now we let $j$ be the positive integer $j=j(\varsigma)$ such that
\begin{align}\label{choi.j}
    \frac12\leq j(\varrho r_0/4) <\frac12+\frac{1}{2^{10}}.
\end{align}
Thus, this together with the fact that $\varrho\leq r_0 (\varsigma/4)^3$ implies 
\begin{equation}\label{j.universal}
\begin{aligned}
    \frac{j(j+1)}2(\varrho /4)^3r_0\leq {j^2}(\varrho r_0/4)^2(\varsigma/4)^3\leq (\varsigma/4)^3\quad\text{and}\quad j(\varrho/4)^2\leq j\varrho r_0/4\leq 1.
\end{aligned}
\end{equation}
Therefore, we have 
\begin{align*}
    \frac{1}{C_0^{k+j}}\frac1{c_0}\leq f(z)+\sum_{i=0}^{k-1+m+\mathbf{n}_1+\mathbf{n}}\frac{1}{C_0^i\mathcal{M}}\quad\text{in }\mathcal{I}_j\times B'_{k(\delta r_0)^3-R_1-1}\times [(\varsigma/2)^3,3/4]\times B_{1/2}
\end{align*}
for some constant $c_1=c_1(n,\Lambda,\varsigma,k,\delta)$.

\textbf{Step 5.} Now, we are in a position to conclude the proof. We let $k$ be the positive integer $k=k(1/\delta)$  such that
\begin{align}\label{k.choi}
    \frac9{8}+R_1\leq k(\delta r_0)^3<2+R_1,
\end{align} 
which gives
\begin{align*}
    k(\delta r_0)^2\geq 1,
\end{align*}
where the constant $R_1=R_1(1/\delta)$ is given in \eqref{const,R0R1}. We are now going to select $M$ and $\mathcal{M}$. First, choosing $M$ by
\begin{equation}\label{choi.M}
    M\coloneqq \sqrt{\eps_0(k\delta^2+R_0+j(\varrho/4)^2)}^{\frac12},
\end{equation}
which depends only on $1/\delta$ by \eqref{k.choi}, \eqref{const,R0R1}, and \eqref{j.universal}, we deduce $  (-1,0]\subset \mathcal{I}_j$,
as $k(\delta r_0)^2\geq1$. Next, we select $\mathcal{M}=\mathcal{M}(n,\Lambda,1/\delta,\varsigma)$ such that
\begin{align}\label{choi.mathcalM}
    \sum_{i=0}^{\infty}\frac{1}{C_0^i\mathcal{M}}=\frac{C_0}{C_0-1}\frac{1}{\mathcal{M}}\leq \frac{1}{2c_1},
\end{align}
where $c_1\coloneqq C_0^{k+j}c_0$, which implies
\begin{align}\label{ineq.last.hopf}
    \inf_{a\in[0,\frac1{16}]}\inf_{H^+_{\frac12,\varsigma}(-a,0,0)}f\geq \frac1c
\end{align}
for some constant $c=c(n,\Lambda,1/\delta,\varsigma)$. Thus, we have proved \eqref{pos.away}, which completes the proof.
\end{proof}

We conclude this section by presenting an auxiliary lemma that will be used in the next section. It shows that positivity propagates even up to the boundary of the solution domain at outgoing boundary points $\gamma_+$. The main idea of the proof is to suitably extend the solution across $\gamma_+$.

\begin{lemma}\label{lem.hargamma+}
  Let $r\in(0,1]$ and let $f\geq0$ be a weak solution to 
    \begin{align*}
        \partial_tf+v\cdot\nabla_xf-\ddiv(A\nabla_v f)=B\cdot\nabla_vf+F\quad\text{in }H_{4r}.
    \end{align*}
    Let $\varsigma\in(0,\frac12]$ and assume
    \begin{align*}
        \inf_{H^+_{2r,\varsigma}(-a,0,0)}f\geq c_0
    \end{align*}
    for some $c_0>0$ and $a\in[0,r^2]$. Then there is a constant $\mathcal{M}=\mathcal{M}(n,\Lambda,\frac{a}{\varsigma^2})\geq1$ such that if
    \begin{align*}
        \|r^2F\|_{L^\infty(H_{4r})}\leq \frac{c_0}{\mathcal{M}},
    \end{align*}
    then 
    \begin{align*}
        f\geq \frac{c_0}{c} \quad\text{in }\mathcal{A}_{r,2\varsigma},
    \end{align*}
    where $c=c(n,\Lambda,\varsigma)$ and we write
    \begin{align}\label{set.mathcalar}
        \mathcal{A}_{r,2\varsigma}\coloneqq {H}_{r}\cap\left[\Big(\{0\leq x_n\leq (2r\varsigma)^3\}\times \{-r\leq v_n\leq-2r\varsigma\}\Big)\cup \Big(\{x_n=(2r\varsigma)^3\}\times \{-r\leq v_n\leq0\}\Big)\right].
    \end{align}
\end{lemma}
\begin{proof}
By scaling, we may assume $r=1$. First, fix a constant $\mathcal{M}\geq1$, which will be determined at the end of the proof. Next note from \cite{Sil22} that $f$ is H\"older continuous up to the boundary, so that we can assume that $f\in C^{\alpha}(\overline{Q})$, where \begin{align*}
        Q\coloneqq \{(t,x',x_n,v',v_n)\,: (t,x',v')\in Q'_3,\, v_n\in[-3-\varsigma,-\varsigma],\,x_n\in[-3,3]\}
    \end{align*} and $f\geq0$.
    Thus, as in the proof of \cite[Cor 2.9]{Zhu24} together with $\nabla_vf\in L^2(Q)$ and $f\in C^\alpha(\overline{Q})$, we have a unique weak solution $f_0\geq0$ to 
\begin{equation}\label{f00.eq}
\left\{
\begin{alignedat}{3}
\partial_tf_0+v\cdot\nabla_xf_0-\ddiv({A}_0\nabla_ v{f_0})&=B_0\cdot\nabla_vf_0+F_0&&\qquad \mbox{in  $Q$}, \\
{f_0}&=f&&\qquad  \mbox{in $\partial_{\mathrm{kin}}Q$},
\end{alignedat} \right.
\end{equation} 
where $A_0=A$, $B_0=B$, $F_0=F$ on $H_4$ and $A_0\equiv I$, $B_0\equiv0$, $F_0\equiv 0$ on $(H_4)^c$.
On the other hand, since $\{x_n=0\}\cap\partial_{\mathrm{kin}}(Q\cap \{x_n>0\})=\emptyset$, we have
\begin{equation*}
\left\{
\begin{alignedat}{3}
\partial_t(f-f_0)+v\cdot\nabla_x(f-f_0)-\ddiv({A}\nabla_ v(f-{f_0}))&=B\cdot\nabla_v(f-f_0)&&\qquad \mbox{in  $Q\cap\{x_n>0\}$}, \\
f-f_0&=0&&\qquad  \mbox{in $\partial_{\mathrm{kin}}(Q\cap \{x_n>0\})$}.
\end{alignedat} \right.
\end{equation*} 
Hence by the uniqueness of solutions to this problem, we get $f=f_0$ in $Q\cap\{x_n\geq0\}$.

Thus, we deduce that $f$ is also a solution to \eqref{f00.eq}, and we will apply Harnack's inequality iteratively to the equation \eqref{f00.eq}. This yields for any $k\in\N$,
     \begin{equation}\label{harnack.ineq10f}
     \begin{aligned}
       f_0(z_1)&\leq c\big(f_0(t_1+\eps_0\rho^2,x_1+v_1\eps_0\rho^2,v_1)+c_0/\mathcal{M}\big)\\
       &\leq \cdots\leq c^kf_0(t_1+k\eps_0\rho^2,x_1+kv_1\eps_0\rho^2,v_1)+c_0\sum_{i=1}^kc^i/\mathcal{M}
    \end{aligned}
    \end{equation}
    for some constants $\eps_0$ and $c=c(n,\Lambda)$,
    whenever $Q_{\rho}(t_1+k\eps_0\rho^2,x_1+kv_1\eps_0\rho^2,v_1)\subset Q\cup H_4$. Next, we choose $\rho\coloneqq\frac{\varsigma}{4}$ so that $Q_{{\rho}}(t_2,x_2,v_2)\subset Q\cup H_4$ for every $(t_2,x_2,v_2)\in \mathcal{A}_{1,2\varsigma}$. Moreover, there is a positive integer $k=k(\frac{a}{\varsigma^2})$ such that $\frac{64}{\eps_0}+\frac{a}{\eps_0}(\frac{4}{\varsigma})^2\leq k\leq \frac{128}{\eps_0}+\frac{a}{\eps_0}(\frac{4}{\varsigma})^2$. Thus, for any $(t_2,x_2,v_2)\in \mathcal{A}_{1,2\varsigma}$, we have
    \begin{align*}
        (t_2-k\eps_0\rho^2,x_2-kv_2\eps_0\rho^2,v_2)\in H_{2,\varsigma}^+(-a,0,0).
    \end{align*}
    In particular, when $a=0$, the positive integer $k$ is a universal constant.
    Thus, by \eqref{harnack.ineq10f} for any $(t_2,x_2,v_2)\in\mathcal{A}_{1,2\varsigma}$, we get
    \begin{align}\label{harnack.ineq1f}
        c_0\leq f_0 (t_2-k\eps_0\rho^2,x_2-kv_2\eps_0\rho^2,v_2)\leq c\big(f_0(t_2,x_2,v_2)+c_0/\mathcal{M}\big)
    \end{align}
    for some constant $c=c(n,\Lambda,\frac{a}{\varsigma^2})$. Now by choosing $\mathcal{M}=\mathcal{M}(n,\Lambda,\frac{a}{\varsigma^2})$ large enough, we get the desired result from the fact that $f_0(t_2,x_2,v_2) = f(t_2,x_2,v_2)$ since $(t_2,x_2,v_2) \in \mathcal{A}_{1,2\varsigma} \subset Q \cap \{ x_n > 0\}$. This completes the proof.
\end{proof}

\section{Hopf lemma} 
\label{sec:Hopf}
In this section, we prove a Hopf lemma for a general class of kinetic equations. We proceed in several steps. First, we show the result for solutions to Kolmogorov's equation (with constant coefficients) and then, in a second step, we develop a perturbation argument which allows us to prove the result for kinetic Fokker-Planck equations with H\"older continuous coefficients and bounded source terms (see \autoref{lem.hop.gen}).

Recall the 1D function ${\varphi_0}$, which is introduced in \eqref{notationvarphi} and solves \eqref{eq.phi0}. In addition, from \cite[Lemma C.1 and Lemma C.5]{KiWe26}, we have
\begin{equation}\label{property.phi_0}
\begin{aligned}
&{\varphi_0}\geq0,\quad\partial_{v}{\varphi_0}\leq 0,\quad {\varphi_0}(x,v)\leq c\max\{|x|^{\frac13},|v|\}^{\frac12}, \\
&{\varphi_0}(x,v)\geq \frac1c\max\{|x|^{\frac13},|v|\}^{\frac12}\quad\text{whenever }v\leq x^3.
\end{aligned}
\end{equation}

\subsection{Hopf lemma for Kolmogorov's equation with constant coefficients}
First, we define a time-dependent kinetic cylinder $\mathcal{H}_{a,b}$ with $a\geq b$, which is has radius $a$ in $t,x',v'$ and $b$ in $x_n,v_n$, and is given by
\begin{equation}\label{defn.sethab}
\begin{aligned}
    &\mathcal{H}_{a,b}\coloneqq \{z\in H_a\,:\,x_n\in[0,b^3],\, v_n\in[-b,b]\} = (-a^2,0] \times (B_{a^3}' \times [0,b^3]) \times (B_a' \times [-b,b]),\\
    &\Gamma_{a,b}\coloneqq \partial_{\mathrm{kin}}\mathcal{H}_{a,b}\cap\Big( \left[\{x_n=b^3\}\times \{-b^3\leq v_n\leq 0\}\right]\cup \left[\{0\leq x_n\leq b^3\}\times \{v_n=-b\}\right]\Big).
\end{aligned}
\end{equation}
In \autoref{lem.exp.posright}, we have proved that for positivity to be propagated up to time level $t=t_0$, we need to assume positivity at time $t=t_0-M^2$ for some constant $M$, which cannot be arbitrarily small. This motivates to introduce the aforementioned cylinders $\mathcal{H}_{a,b}$.

With this cylinder, we construct a barrier that serves as a subsolution.
\begin{lemma}\label{lem.sub.nd}
    There exists a large constant $M=M(n,\Lambda)\geq8$ such that there is a solution $\psi_1$ to
    \begin{equation}\label{eq.subsol0n}
\left\{
\begin{alignedat}{3}
\partial_t\psi_1+v\cdot\nabla_{x}\psi_1-a^{i,j}\partial_{v_i,v_j}\psi_1&\leq 0&&\qquad \mbox{in  $\mathcal{H}_{M,1}$}, \\
\psi_1&\leq 0&&\qquad  \mbox{on $\partial_{\mathrm{kin}}\mathcal{H}_{M,1}\setminus \Gamma_{M,1}$},\\
\psi_1&\leq c&&\qquad\mbox{on $\Gamma_{M,1}$},\\
\psi_1(t,x,v)&\geq {\varphi_0}(x_n,v_n)/2&&\qquad\mbox{in $(t,x,v)\in {H}_{1/2}$}
\end{alignedat} \right.
\end{equation} 
for some constant $c=c(n,\Lambda)$, where $a^{n,n}=1$.
\end{lemma}
\begin{proof}
    We fix $M\geq1$, which will be determined later in \eqref{choi.M0}. Next, we consider cut-off functions $\xi_1=\xi_1(t)$ and $\xi_2=\xi_2(x',v')$ such that
    \begin{align}\label{xi1.hopfn}
        \xi_1(t)=0\text{ for }t\in[-1,0],\quad \xi_1(-M^2)=-8,\quad \xi_1(t)\in[-8,0],\quad|\partial_t\xi_1|\leq \frac{c}{M^2}
    \end{align}
    and 
    \begin{align}\label{xi2.hopfn}
        \xi_2\equiv 0\text{ on }\mathrm{Q}_1',\quad \xi_2\equiv -8\text{ on }\partial\mathrm{Q}'_M,\quad \xi_2\in[-8,0],\quad|D^k\xi_2|\leq \frac{c}{M^k}
    \end{align}
    for some constant $c=c(n)$, where $k\in\{1,2,3\}$ and $\mathrm{Q}'_M\subset \bbR^{n-1}\times \bbR^{n-1}$ is a stationary cylinder defined in \eqref{stationary cylinder}. Then we set 
    \begin{align*}
        \psi_1\coloneqq (\xi_1(t)+\xi_2(x',v')){\varphi_0}+{\varphi_0}(v_n-1)^2+(1-v_n){\varphi_0},
    \end{align*}
    where ${\varphi_0}\equiv{\varphi_0}(x_n,v_n)$, to see that
    \begin{align*}
        &\partial_t\psi_1=(\partial_t\xi_1){\varphi_0},\\
        &v_i\partial_{x_i}\psi_1=v_i(\partial_{x_i}\xi_2){\varphi_0},\quad -a^{i,j}\partial_{v_i,v_j}\psi_1=-a^{i,j}(\partial_{v_i,v_j}\xi_2){\varphi_0}\quad\text{for any }i,j\neq n,\\
        &v_n\partial_{x_n}\psi_1-a^{n,n}\partial_{v_n,v_n}\psi_1=-2{\varphi_0}-4\partial_{v_n}{\varphi_0}(v_n-1)+2\partial_{v_n}{\varphi_0},\\
        &-a^{i,n}\partial_{v_i,v_n}\psi_1=-a^{i,n}(\partial_{v_i}\xi_2)(\partial_{v_n}{\varphi_0}),
    \end{align*}    
    where we have used \eqref{eq.phi0}. Now using this, together with \eqref{xi1.hopfn} and \eqref{xi2.hopfn}, we derive
    \begin{align*}
       \partial_t\psi_1+v\cdot\nabla_{x}\psi_1-a^{i,j}\partial_{v_i,vj}\psi_1\leq -2{\varphi_0}+2\partial_{v_n}{\varphi_0}+c\left(\frac{{\varphi_0}}{M^2}+\frac{|\partial_{v_n}{\varphi_0}|}{M}\right) \quad \text{ in } \mathcal{H}_{M,1}
    \end{align*}
    for some constant $c=c(n,\Lambda)$. Since ${\varphi_0}\geq0$, $\partial_{v_n}{\varphi_0}\leq0$, and using also that $a^{n,n} = 1$, there is a constant $M=M(n,\Lambda)\geq8$ such that 
    \begin{align}\label{choi.M0}
    -2{\varphi_0}+2\partial_{v_n}{\varphi_0}+c\left(\frac{{\varphi_0}}{M^2}+\frac{|\partial_{v_n}{\varphi_0}|}{M}\right)\leq -{\varphi_0}+\partial_{v_n}{\varphi_0}\leq 0,
    \end{align}
     which gives that $\psi_1$ satisfies the first condition given in \eqref{eq.subsol0n}.
     
    Now we will check the remaining three conditions in \eqref{eq.subsol0n}. We observe 
    \begin{align*}
        \psi_1(-M^2,x,v)\leq -8{\varphi_0}+{\varphi_0}(v_n-1)^2+(1-v_n){\varphi_0}\leq -2{\varphi_0} &\le 0, \quad \text{for any } (-M^2,x,v) \in \mathcal{H}_{M,1},\\
        \psi_1(t,x,v)\leq -8{\varphi_0}+{\varphi_0}(v_n-1)^2+(1-v_n){\varphi_0}\leq -2{\varphi_0} &\le 0 \quad\text{in }\{z\in \overline{\mathcal{H}_{M,1}}\,:\,(x',v')\in \partial \mathrm{Q}'_M\},\\
        \psi_1(t,x',0,v',v_n)&=0 \quad\text{in }\{(t,x',0,v',v_n)\in\mathcal{H}_{M,1}\,:\, v_n>0\}\\
        \psi_1(t,x',x_n,v',1) &\leq 0 \quad\text{in }\{(t,x',x_n,v',1)\in\mathcal{H}_{M,1}\},\\
        |\psi_1(z)|\leq 24|{\varphi_0}(x_n,v_n)|&\leq c \quad \text{ in }\mathcal{H}_{M,1},\\
        \psi_1(t,x,v)\geq {\varphi_0}(v_n-1)^2+(1-v_n){\varphi_0} &\geq \frac{{\varphi_0}}{2}\quad\text{in }{H}_{\frac12}
    \end{align*}
    where we have used $\xi_1{\varphi_0},\xi_2{\varphi_0}\leq 0$. This completes the proof.
\end{proof}

With \autoref{lem.sub.nd} at hand, we are now ready to prove the main result in this subsection. This result was already established for stationary equations in 1D in \cite[Lemma 3.7]{KiWe26}. Here, we are using the new barrier function from \autoref{lem.sub.nd} and the propagation of positivity result from \autoref{lem.hargamma+} to extend it to the non-stationary equation in higher dimension.

\begin{lemma}\label{lem.hopf}
    Let $h\geq0$ be a weak solution to 
    \begin{equation}\label{eq.hopf}
\left\{
\begin{alignedat}{3}
\partial_th+v\cdot\nabla_xh-a^{i,j}\partial_{v_i,v_j}{h}&=0&&\qquad \mbox{in  $H_{4r}$}, \\
h&=0&&\qquad  \mbox{in $\{x_n=0\}\times \{v_n>0\}$}
\end{alignedat} \right.
\end{equation} 
with $a^{n,n}=1$. There are constants $c=c(n,\Lambda)$ and $\sigma=\sigma(n,\Lambda)\leq \frac14$  such that 
\begin{align*}
     \inf_{H_{2r,\sigma}^+(-r^2,0,0)}\frac{h}{{\varphi_0}}\leq c\inf_{H_{\sigma r}}\frac{h}{{\varphi_0}}.
\end{align*}
\end{lemma}
\begin{proof}
We may assume $r=1$. First, fix the positive integer $M=M(n,\Lambda)\geq8$ and the function $\psi_1$, which are determined in \autoref{lem.sub.nd}. Then we define $\varphi(z)\coloneqq \psi_1(S_{1/M}z)$ to see that
    \begin{equation}\label{h0.eq}
\left\{
\begin{alignedat}{3}
\partial_t\varphi+v\cdot\nabla_x\varphi-a^{i,j}\partial_{v_i,v_j}{\varphi}&\leq0&&\qquad \mbox{in  $\mathcal{H}_{1,1/M}$}, \\
{\varphi}&\leq0&&\qquad  \mbox{on $\partial_{\mathrm{kin}}\mathcal{H}_{1,1/M}\setminus\Gamma_{1,1/M}$},\\
\varphi&\leq c&&\qquad \mbox{in $\Gamma_{1,1/M}$},\\
\varphi&\geq \frac{{\varphi_0}}{c}&&\qquad\text{in } {H}_{1/(2M)}
\end{alignedat} \right.
\end{equation} 
for some constant $c=c(n,\Lambda)$, as $M$ depends only on $n,\Lambda$, where the sets $\mathcal{H}_{1,1/M}$ and $\Gamma_{1,1/M}$ are defined in \eqref{defn.sethab}. 

Now we fix $\sigma\coloneqq\frac{1}{2M}$. We will prove the lower bound of the function $h$ on $\Gamma_{1,2\sigma}$. Next, we may set $m_0\coloneqq  \inf\limits_{H_{2,\sigma}^{+}(-1,0,0)}\frac{h}{{\varphi_0}}\neq0$. This gives 
    \begin{align}\label{ass.harnack}
        h\geq \frac{m_0}{c_0}\quad\text{in }H_{2,\sigma}^+(-1,0,0)
    \end{align}
    for some constant $c_0=c_0(n,\Lambda)$, as ${\varphi_0}(x,v)\geq \frac1{c_0}$ when $x\geq (2\sigma)^3$ by \eqref{property.phi_0}. By applying \autoref{lem.hargamma+} with $r=1$, $\varsigma=\sigma(=\frac1{2M})$, and $a=1$, we have
    \begin{align}\label{step1.ineq.hopf}
        h \geq \frac{m_0}{c}\quad\text{in } \Gamma_{1,2\sigma}\subset \mathcal{A}_{1,2\sigma}.
    \end{align}
    for some constant $c=c(n,\Lambda)$, as the constant $\sigma$ depends only on $n,\Lambda$.
    
Now, we are ready to apply the maximum principle to get the desired lower bound of $h$ in ${H}_\sigma\cap \{t=0\}$.  
By \eqref{h0.eq} and \eqref{step1.ineq.hopf}, we have that there is a constant $c_3=c_3(n,\Lambda)$ such that $\widetilde{h}\coloneqq\frac{m_0\varphi}{c_3}-h$ is a solution to
\begin{equation*}
\left\{
\begin{alignedat}{3}
\partial_t\widetilde{h}+v\cdot\nabla_{x}\widetilde{h}-\ddiv(A\nabla_v \widetilde{h})&\leq0&&\qquad \mbox{in  $\mathcal{H}_{1,1/M}$}, \\
\widetilde{h}&\leq 0&&\qquad  \mbox{in $\partial_{\mathrm{kin}}\mathcal{H}_{1,1/M}$}.
\end{alignedat} \right.
\end{equation*} 
Thus, by the maximum principle, we have $\frac{m_0}{c_3}\varphi\leq h$ in $\mathcal{H}_{1,1/M}$. In particular, we deduce
\begin{align}\label{last.harnack}
    \left(\inf_{H^+_{2,\sigma}(-1,0,0)}\frac{h}{{\varphi_0}}\right){\varphi_0}(x_n,v_n)\leq ch(t,x,v)\quad\text{for any }(t,x,v)\in H_{\sigma}
\end{align}
for some constant $c=c(n,\Lambda)$, where we have used that $\frac{m_0\varphi}{c_3}\geq \frac{m_0}{c}{\varphi_0}$ in $H_{\sigma}$ for some constant $c=c(n,\Lambda)$, which follows from the last condition of \eqref{h0.eq} and $\sigma=\frac1{2M}$. This completes the proof.
\end{proof}

\subsection{Hopf lemma for a general class of kinetic Fokker-Planck equations}
\label{subsec:Hopf-general}

In this subsection, we consider more general kinetic equations and prove a Hopf lemma of such equations via a perturbation argument. 

The main result of this subsection reads as follows.

\begin{theorem}\label{lem.hop.gen}
    Let $f\geq0$ be a weak solution to 
    \begin{equation*}
\left\{
\begin{alignedat}{3}
\partial_tf+v\cdot\nabla_{x}f-\ddiv(A\nabla_vf)&=B\cdot\nabla_vf+F&&\qquad \mbox{in  ${H}_2$}, \\
f&=0&&\qquad  \mbox{in $\gamma_-\cap H_2$},
\end{alignedat} \right.
\end{equation*}
where $A\in C^{\eps}(H_2)$, $B\in L^\infty(H_2)$, $0\leq F\in L^\infty(H_2)$ for some $\eps>0$. We assume  $(A(0))_{n,n}=1$ and
\begin{align*}
    \|f\|_{L^\infty(H_{2})}\leq 1.
\end{align*}
Then there exist a universal constant $\rho$ and a large constant $\mathcal{M}=\mathcal{M}(n,\Lambda)$ such that if  $f(z_0)=c_0>0$ for some $c_0 > 0$ with $z_0=(t_0,x_0',x_{0,n},v_0)=(-1,0,\rho^3,0)$ and 
\begin{align*}
    \|F\|_{L^\infty(H_2)}\leq \frac{c_0}{\mathcal{M}},
\end{align*} then
\begin{align*}
    \frac{c_0}c\leq \inf_{H_{\frac\rho8}\setminus \mathcal{R}_-}\frac{f}{{\varphi_0}},
\end{align*}
 where $c=c(n,\eps,\Lambda,\|A\|_{C^{\eps}(H_2)},\|B\|_{L^\infty(H_2)})$.
\end{theorem}

We recall the notation $\mathcal{R}_-\coloneqq\{x_n\leq v_n^3\}$. Note that we do not establish the comparability of $f$ and ${\varphi_0}$ inside $\mathcal{R}_-$. This is due to the fact that ${\varphi_0}$ exhibits exponential decay at the boundary inside this region, which is in stark contrast to the polynomial growth away from that set (see \eqref{eq.phi0}). Due to the exponential decay of ${\varphi_0}$, the perturbation argument that our proof is based on no longer works. We do not know, whether the result of \autoref{lem.hop.gen} holds in $\mathcal{R}_-$, when $A$ is non-constant and $B,F \not= 0$.

\begin{lemma}\label{lem.hop.gen0}
    Let $f\geq0$ be a weak solution to 
    \begin{equation*}
\left\{
\begin{alignedat}{3}
\partial_tf+v\cdot\nabla_{x}f-\ddiv(A\nabla_vf)&=B\cdot\nabla_vf&&\qquad \mbox{in  ${H}_2$}, \\
f&=0&&\qquad  \mbox{in $\gamma_-\cap H_2$},
\end{alignedat} \right.
\end{equation*}
where $A\in C^{\eps}(H_2)$ and $B\in L^\infty(H_2)$ with $(A(0))_{n,n}=1$. We assume 
\begin{align*}
    \|f\|_{L^\infty(H_{2})}\leq 1.
\end{align*}
Then there is a universal constant $\rho\leq1$ such that if  $f(z_0)=c_0$ for some constant $c_0>0$ with $z_0=(t_0,x_0',x_{0,n},v_0)=(-1,0,\rho^3,0)$, then
\begin{align*}
    \frac{c_0}c\leq \inf_{H_{\frac\rho8}\setminus \mathcal{R}_-}\frac{f}{{\varphi_0}},
\end{align*}
 where $c=c(n,\eps,\Lambda,\|A\|_{C^{\eps}(H_2)},\|B\|_{L^\infty(H_2)})$.
\end{lemma}

The main idea of the proof is to compare $f$ with a solution $h$ to Kolmogorov's equation with constant coefficients and to apply \autoref{lem.hopf}.

\begin{proof}
Let $M$ be the constant from \autoref{lem.exp.posright} applied with $\delta=1$ and fix $\rho=\frac{1}{M}$. Then by \autoref{lem.exp.posright} applied with $r=\rho$, we have for any
$\varsigma\in(0,\frac1{4})$, 
\begin{align*}
    \inf_{H^+_{\frac\rho2,\varsigma}}{f}\geq \frac{c_0}{c}
\end{align*}
for some constant $c=c(n,\Lambda,\varsigma)$. In addition, by \autoref{lem.hargamma+} applied with $F\equiv0$,  $a=0$, and $r=\frac\rho4$,
\begin{align}\label{fc0c.zero}
    f\geq \frac{c_0}{c}\quad\text{on }H^+_{\frac\rho2,\varsigma}\cup \mathcal{A}_{\frac{\rho}4,2\varsigma}
\end{align}
for some constant $c=c(n,\Lambda,\varsigma)$.

Next, we consider the weak solution $h_1$ to
    \begin{equation}\label{eq.h1.hop.gen0}
\left\{
\begin{alignedat}{3}
\partial_th_1+v\cdot\nabla_{x}h_1-\ddiv(A\nabla_vh_1)&=B\cdot\nabla_vh_1&&\qquad \mbox{in  $\mathcal{H}_{r,\frac{r}N}$}, \\
h_1&=1&&\qquad  \mbox{on $\Gamma_{r,\frac{r}N}$},\\
h_1&=0&&\qquad  \mbox{on $\partial_{\mathrm{kin}}\mathcal{H}_{r,\frac{r}N}\setminus \Gamma_{r,\frac{r}N}$}
\end{alignedat} \right.
\end{equation}
for some small constant $r\in(0,1)$ and large constant $N\geq1$, which will be determined later (see \eqref{choi.r.hopf.gen} and \eqref{choi.N.zero}, respectively). Note that the sets $\mathcal{H}_{r,\frac{r}N}$ and $\Gamma_{r,\frac{r}N}$ are defined in \eqref{defn.sethab}. In particular, we will always choose $N$, $r\leq \frac{\rho}{4}$, and then $\varsigma=\frac{2r}{N\rho}$, so that $\Gamma_{r,\frac{r}N}\subset H^+_{\frac\rho2,\varsigma}\cup \mathcal{A}_{\frac{\rho}4,{2\varsigma}}$. Due to \eqref{fc0c.zero} and the construction of $h_1$, we can apply the maximum principle, which yields
\begin{align}\label{fh1.hop.gen.zero}
    f\geq \frac{c_0}{c}h_1\quad\text{in }\mathcal{H}_{r,\frac{r}{N}}.
\end{align}
The idea of this construction is to reduce the proof of the Hopf lemma for $f$ to the proof of a Hopf lemma for $h_1$, for which the values at $\Gamma_{r,\frac{r}{N}}$ are normalized to one.

\textbf{Step 1.} We will derive a positive lower bound for $h_1$ away from the boundary $\{ x_n = 0 \}$ for each time level $t\in[-(r/4)^2,0]$, which is needed for \textbf{Step 2} (see \eqref{choi.N.zero} below). To do so, we observe from the maximum principle that $h_1\leq 1$ and by \cite{Sil22}, for any $z_1=(t_1,x_{1}',x_{1,n},v_1)=(t_1,0,(\frac{r}{N})^3,0)$ with $t_1\in[-(\frac{r}{2})^2,0]$,
\begin{equation*}
   (r/N)^\beta[h_1]_{C^{\beta}(H_{{r/(2N)}}(z_1))}\leq c,
\end{equation*}
where $c=c(n,\Lambda)$ and $\beta=\beta(n,\Lambda)\in(0,1)$. By choosing a sufficiently small constant $a_0=a_0(n,\Lambda)\in(0,\frac12)$, we derive for any $t_1\in[-(\frac{r}{2})^2,0]$,
\begin{align}\label{h1.12.gen.hopf.zero}
    {h}_1(t_1,0,(r/N)^3(1-a_0),0)\geq {h}_1(z_1)-| {h}(z_1)-{h}_1(t_1,0,(r/N)^3(1-a_0),0)|\geq 1-ca_0^{\frac\beta3}\geq \frac12.
\end{align}
Next, note that $h_1$ solves 
\begin{align}\label{eq.widetildeh1}
    \partial_t{h}_1+v\cdot\nabla_x{h}_1-\ddiv({A}\nabla_v{h}_1)={B}\cdot\nabla_v{h}_1\quad\text{in }I_{\frac{r}{2}}\times \mathrm{H}_{\frac{r}{N}(1-a_0)^{\frac13}(1+\delta_0)}.
\end{align}
Now let $M_0=M_0(1/\delta_0)$ be the large constant determined in \autoref{lem.exp.posright} with $F\equiv 0$, $\delta=\delta_0$ and $\varsigma=\frac12$. By assuming $N\geq 4M_0$ and \eqref{h1.12.gen.hopf.zero}, we have for $\widetilde{t}_1=-(\frac{rM_0}{N})^2(1-a_0)^{\frac23}$
\begin{align*}
    {h}_1(-\widetilde{t}_1,0,(r/N)^3(1-a_0),0))\geq\frac12.
\end{align*}
Thus, by \autoref{lem.exp.posright} with $r$ replaced by $\frac{r}{N}(1-a_0)^{\frac13}$, we have for $\widetilde{\rho}_0\coloneqq \frac{r}{N}(1-a_0)^{\frac13}$,
\begin{align*}
    \frac1{c_1}\leq \inf_{H^+_{\frac{\widetilde{\rho}_0}2,\frac12}}{h}_1,
\end{align*}
where $c_1=c_1(n,\Lambda)$, as the constant $\delta_0$ depends only on $n,\Lambda$. Since we want to get the lower bound for each time level, we now apply the same argument to the function $\widetilde{h}_1(z)\coloneqq h_1(t_0+t,x,v)$ for any $t_1\in[-(\frac{r}4)^2,0)$. More precisely, we observe that \eqref{h1.12.gen.hopf.zero} also holds with $h_1$ replaced by $\widetilde{h}_1$ and that\eqref{eq.widetildeh1} holds with $h_1$, $A$ and $B$ replaced by $\widetilde{h}_1$, $A(t_0+t,x,v)$ and $B(t_0+t,x,v)$. Thus, we deduce 
\begin{align}\label{h_1.hop.gen.zero}
    \frac1{c_1}\leq \inf_{H^+_{\frac{\widetilde{\rho}_0}2,\frac12}(t_0,0,0)}{h}_1
\end{align}
for some constant $c_1=c_1(n,\Lambda)$ and any $t_0\in[-(\frac{r}{4})^2,0]$.

\textbf{Step 2.} In this step, we consider the weak solution $h_2$ of the corresponding equation with frozen coefficients
\begin{equation*}
\left\{
\begin{alignedat}{3}
\partial_th_2+v\cdot\nabla_{x}h_2-\ddiv(A(0)\nabla_vh_2)&=0&&\qquad \mbox{in  $\mathcal{H}_{\frac{r}{4},\frac{r}{4N}}$}, \\
h_2&=c_1h_1&&\qquad  \mbox{in $\partial_{\mathrm{kin}}\mathcal{H}_{\frac{r}4,\frac{r}{4N}}$},
\end{alignedat} \right.
\end{equation*}
where the constant $c_1=c_1(n,\Lambda)$ is given in \eqref{h_1.hop.gen.zero}. The goal of this step is to derive a suitable lower bound for $h_2$ (see \eqref{h2.lower.zero}). Since $h_2$ solves a translation-invariant equation, we can apply the Hopf-type lemma from the previous subsection (see \autoref{lem.hopf}).

First, we note $(t_0,0,(\frac{r}{4N})^3,0)\in H^+_{\frac{\widetilde{\rho}_0}2,\frac12}(t_0,0,0)$ for any $t_0\in [-(\frac{r}8)^2,0]$ as $\widetilde{\rho}_0=\frac{r}{N}(1-a_0)^{\frac13}\in [\frac{r}{2N},\frac{r}N]$ by $a_0\in(0,\frac12)$. Thus, we see 
\begin{align*}
    h_2(t_0,0,(\frac{r}{4N})^3,0)=c_1h_1(t_0,0,(\frac{r}{4N})^3,0)\geq 1
\end{align*}
by \eqref{h_1.hop.gen.zero}. Since $|h_1|\leq 1$, and the constant $c_1$ depends only on $n,\Lambda$, we have $|h_2|\leq c$, hence by \cite{Sil22},
\begin{equation}
   (r/N)^\beta[h_2]_{C^{\beta}(H_{{r/(8N)}}(z_1))}\leq c
\end{equation}
for some constant $c=c(n,\Lambda)$, for any $z_1=(t_1,0,(\frac{r}{4N})^3,0)$ with $t_1\in[-(\frac{r}8)^2,0]$. As in \eqref{h1.12.gen.hopf.zero}, for any $t_1\in[-(r/8)^2,0]$, we have 
\begin{align}\label{upper.h2.zero}
    h_2(t_1,0,(r/(4N))^3(1-a_1),0)\geq \frac12
\end{align}
for some constant $a_1=a_1(n,\Lambda)$. Let $M_1=M_1(1/\delta_1)$ be the constant determined in \autoref{lem.exp.posright} with $F\equiv0$, $\delta=\delta_1$ and $\varsigma=\sigma$, where $\delta_1\coloneqq (1-a_1)^{-\frac13}-1$ and the constant $\sigma=\sigma(n,\Lambda)\leq\frac14$ is given in \autoref{lem.hopf}. By choosing $N\coloneqq4M_0+2M_1$, which depends only on $n,\Lambda$ by the fact that constants $\delta_0$ and $\delta_1$ depend only on $n,\Lambda$, and using \eqref{upper.h2.zero}, we have 
\begin{align}\label{choi.N.zero}
    {h}_2(-(M_1\widetilde{\rho}_1)^2,0,(\widetilde{\rho}_1)^3,0)\geq\frac12,
\end{align}
where $\widetilde{\rho}_1\coloneqq \frac{r}{4N}(1-a_1)^{\frac13}$, as $-(M_1\widetilde{\rho}_1)^2\in [-(\frac{r}8)^2,0]$. Note that $-(M_1\widetilde{\rho}_1)^2$ can be close to $(-r/8)^2$. This is the reason why we had to show the lower bound given in \eqref{h_1.hop.gen.zero} for each time level $t_0\in[-(\frac{r}{4})^2,0]$.

Therefore, by \autoref{lem.exp.posright} with $r=\widetilde{\rho}_1$, $\delta=\delta_1$ and $\varsigma=\sigma$, we have 
\begin{align*}
    \frac1c\leq\inf_{H^+_{\frac{\widetilde{\rho}_1}2,\sigma}(-({\widetilde{\rho}_1}/{4})^2,0,0)}h_2
\end{align*}
for some constant $c=c(n,\Lambda)$. Thus, using \autoref{lem.hopf}, we have 
\begin{align}\label{h2.lower.zero}
    \inf_{H_{\frac{\sigma\widetilde{\rho}_1}4}}\frac{h_2}{{\varphi_0}}\geq \frac1c\inf_{H^+_{\frac{\widetilde{\rho}_1}2,\sigma}(-({\widetilde{\rho}_1}/{4})^2,0,0)}\frac{h_2}{{\varphi_0}}\geq \frac{1}{cr^{\frac12}}
\end{align}
for some constant $c=c(n,\Lambda)$, by \eqref{property.phi_0} and since the constants $N$, $\sigma$ and $a_1$ depend only on $n,\Lambda$.

\textbf{Step 3.} The goal of this step is to prove a lower bound of $h_1$ via a perturbation argument and using the result from \textbf{Step 2}. In this step, for the convenience of the reader, we will say that a constant depends only on the given data, if it depends only on $n,\eps,\Lambda,\|A\|_{C^\eps(H_2)},\|B\|_{L^\infty(H_2)}$. 

First, we use higher order expansions given in \autoref{lem.hol12} to see that there are two constants $C_1, C_2$ such that for any $\varrho\leq \frac{r}{8N}$,
\begin{align}\label{ineq.hopf.gen}
    \|h_1-C_1{\varphi_0}\|_{L^\infty(H_{\varrho})} + \|c_1^{-1}h_2-C_2{\varphi_0}\|_{L^\infty(H_{\varrho})}\leq c\left(\frac{\varrho}{r}\right)^{\frac12+\eps},
\end{align}
as $|h_1| + |h_2|\leq c(n,\Lambda)$ and the constant $N$ determined in \eqref{choi.N.zero} depends only on $n,\Lambda$. The main idea of the proof is to use the fact that the expansions in \eqref{ineq.hopf.gen} are of higher order than $1/2$ in order to deduce the lower bound for $h_1$ from the one of $h_2$.

\textbf{Claim 3.1.} We are going to prove that for any $\delta>0$, there is a sufficiently small $R_0$ depending only on the given data and $\delta$ such that if $r\leq R_0$, then 
\begin{align}\label{c1-c2.hopf.gen}
    |C_1-C_2|\leq {\delta}{{r}^{-\frac12}}.
\end{align}
This is the most delicate part of the proof. To show it, we observe that $h\coloneqq h_1-c_1^{-1}h_2$  solves 
 \begin{equation}\label{eq.h1-h2}
\left\{
\begin{alignedat}{3}
\partial_t h+v\cdot\nabla_{x}h-\ddiv(A(0)\nabla_vh)&=B\cdot \nabla_vh_1-\ddiv((A(0)-A)\nabla_v h_1)&&\qquad \mbox{in  $\mathcal{H}_{\frac{r}{4},\frac{r}{4N}}$}, \\
h&=0&&\qquad  \mbox{in $\partial_{\mathrm{kin}}\mathcal{H}_{\frac{r}{4},\frac{r}{4N}}$}.
\end{alignedat} \right.
\end{equation}
First, using Sobolev's inequality and using energy inequalities by testing $h$ as in \cite[Lemma 4.6]{KLN25}, we deduce 
\begin{align*}
   \dashint_{\mathcal{H}_{\frac{r}{4},\frac{r}{4N}}}|h|^2\,dz\leq  c{r}^{2}\dashint_{\mathcal{H}_{\frac{r}{4},\frac{r}{4N}}}|\nabla_vh|^2\,dz\leq c{r}^{2\eps+2}\dashint_{\mathcal{H}_{\frac{r}{2},\frac{r}{2N}}}|\nabla_vh_1|^2\,dz,
\end{align*}
where we have also used the fact that $\|A-A(0)\|_{L^\infty(H_{{r}})}\leq c{r}^{\eps}$. On the other hand, by testing the equation \eqref{eq.h1.hop.gen0} for $h_1$ with $h_1\psi_0$, where $\psi_0\equiv 1$ on $\mathcal{H}_{\frac{r}{2},\frac{r}{2N}}$ and  $\psi_0\equiv 1$ on $\mathcal{H}_{\frac{3r}{4},\frac{3r}{4N}}$, we derive
\begin{align*}
   r^2\dashint_{\mathcal{H}_{\frac{r}{2},\frac{r}{2N}}}|\nabla_vh_1|^2\,dz\leq c \dashint_{\mathcal{H}_{\frac{3r}{4},\frac{3r}{4N}}}|h_1|^2\,dz\leq c.
\end{align*}
Here we have also used the fact that $|h_1|\leq1$. Therefore, altogether, we have proved
\begin{align}\label{hl2norm.hop.gen}
     \dashint_{I_{\frac{r}4}\times\mathrm{H}_{\frac{r}{4N}}}|h|^2\,dz\leq cr^{2\eps}.
\end{align}
Next, by applying \autoref{lem.bdry.gamma-} to the function $h_1$ and using that $|h_1|\leq 1$, we have for any $z_1\in H_{\frac{r}{4N}}$,
\begin{align*}
    |\nabla_vh_1(z_1)|\leq c\max\{(x_{1,n})^{\frac13},|v_{1,n}|\}^{-\frac12}r^{-\frac12}.
\end{align*} 
Thus, we are able to apply \autoref{lem.point.sol} with $a=\frac{4(1-\eps)}{1+\eps}<4$, $G_1=B\cdot\nabla_vh_1$ and $G_2=(A-A(0))\cdot\nabla_vh_1$ to \eqref{eq.h1-h2} so that we deduce for any $\varrho\leq \frac{r}{8N}$,
\begin{align*}
    |h(0,0,\varrho^3,0)|&\leq c\left(\frac{\varrho}{r}\right)^{\frac{1-\eps}{2}}\left(\left(\dashint_{H_{r/N}}|h|^2\,dz\right)^{\frac12}+cr^{-\frac12+\eps}r^{\frac{1-\eps}2}\right)
    \leq cr^{\frac{\eps}{2}}\left(\frac{\varrho}{r}\right)^{\frac{1-\eps}{2}}
\end{align*}
where we have used the fact that $\|A-A(0)\|_{L^\infty(H_r)}\leq cr^\eps$ and \eqref{hl2norm.hop.gen}.

Therefore, we have for any $\varrho\leq \frac{r}{8N}$ and $z_1=(t_1,x_1',x_{1,n},v_1)=(0,0,\varrho^3,0)$,
\begin{align*}
    |(C_1-C_2){\varphi_0}(z_1)|&\leq |(h_1-C_1{\varphi_0})(z_1)|+|(c_1^{-1} h_2-C_2{\varphi_0})(z_1)|+|h(z_1)|
    \leq c\left(\left(\frac{\varrho}{r}\right)^{\frac12+\eps}+\left(\frac{\varrho}{r}\right)^{\frac{1-\eps}{2}}r^{\frac\eps2}\right).
\end{align*}
Using the fact that ${\varphi_0}(z_1)\eqsim |\varrho|^{\frac12}$ by \eqref{property.phi_0}, we deduce for any $\varrho\leq \frac{r}{8N}$
\begin{align}\label{ineq.c1-c2}
    |C_1-C_2|\leq c\varrho^{-\frac12}\left(\left(\frac{\varrho}{r}\right)^{\frac12+\eps}+\left(\frac{\varrho}{r}\right)^{\frac{1-\eps}{2}}r^{\frac\eps2}\right) \leq cr^{-\frac12}\left(\left(\frac{\varrho}{r}\right)^{\eps}+r^{\frac\eps2}\left(\frac{r}{\varrho}\right)^{\frac{\eps}2}\right).
\end{align}
By choosing $\varrho = \widetilde{\varrho} r$ for some small $\widetilde{\varrho} > 0$ that depends only on the given data and $\delta$, we get
\begin{align}\label{rho0.hopf.gen}
    c\left(\frac{\varrho}{r}\right)^{\eps}=\frac{\delta}{2}.
\end{align}
Next, we choose a small constant $R_0$ depending only on the given data and $\delta$ to show that for any $r\leq R_0$, we have 
\begin{align}\label{choi.r.hopf.gen0}
    cr^{\frac\eps2}\left(\frac{r}{\varrho}\right)^{\frac{\eps}2}\leq \frac{\delta}{2}.
\end{align}
Plugging \eqref{rho0.hopf.gen} and \eqref{choi.r.hopf.gen0} into \eqref{ineq.c1-c2} yields \eqref{c1-c2.hopf.gen} and therefore we have proved Claim 3.1.

Now we are ready to prove a lower bound of $h_1$. 
By using \eqref{h2.lower.zero} and \eqref{ineq.hopf.gen}, we derive
\begin{align*}
    h_1(z)=c_1^{-1} h_2(z)-(c_1^{-1} h_2-h_1)(z)\geq \frac{{\varphi_0}(z)}{cr^{\frac12}}-c\frac{|z|^{\frac12+\eps}}{r^{\frac12+\eps}}-(C_1-C_2){\varphi_0}(z)\quad\text{in }H_{\frac{\sigma\widetilde{\rho}_1}{4}},
\end{align*}
where the constants $C_i$ are determined in \eqref{ineq.hopf.gen} and $\widetilde{\rho}_1=\frac{r}{4N}(1-a_1)^{\frac13}$ with $a_1=a_1(n,\Lambda)\leq \frac12$. 

There is a small constant $\eps_1$ depending only on the data so that we have 
\begin{align}\label{using.r1-}
    h_1(z)\geq \frac{{\varphi_0}(z)}{cr^{\frac12}}-\frac{{\varphi_0}(z)}{2cr^{\frac12}}-(C_1-C_2){\varphi_0}(z)\quad\text{for any }z\in H_{\eps_1\frac{\sigma\widetilde{\rho}_1}{2}}\setminus \mathcal{R}_-,
\end{align}
where we have used the second property given in \eqref{property.phi_0}. Next, note from \eqref{c1-c2.hopf.gen}, that there is a small constant  $R_0$ depending only on the given data so that if we choose $r\coloneqq R_0$, then
\begin{align}\label{choi.r.hopf.gen}
    h_1(z)\geq \frac{{\varphi_0}(z)}{4cr^{\frac12}}\quad\text{for any }z\in H_{\eps_1\frac{\sigma\widetilde{\rho}_1}{2}}\setminus \mathcal{R}_-.
\end{align}

\textbf{Step 4.} Now, we are in a position to conclude the proof. Using \eqref{choi.r.hopf.gen} and \eqref{fh1.hop.gen.zero}, we obtain 
\begin{align}
\label{eq:lower-prelim}
    f\geq \frac{c_0}{c}{\varphi_0} \quad \text{ in }H_{\frac{1}{C}}\setminus \mathcal{R}_-.
\end{align}
Note that this would already prove the desired estimate, however, at this stage, the radius $1/C$ of the cylinder still depends on the data through $A,B$. 

In order to replace $H_{\frac{1}{C}}$ in  \eqref{eq:lower-prelim} by $\rho$, we first apply all the previous steps of the proof to suitable translations of $f$ in $t$. To be precise, given $t_0\in[-(\rho/8)^2,0]$ we denote ${f}_{z_0}(z)\coloneqq f(z_0\circ z)$ where we write $z_0=(t_0,0,0)$. Then, we observe that by \eqref{fc0c.zero}, we get for any $\varsigma \in (0,\frac{1}{4})$,
\begin{align}\label{last.ineq.zero}
    {f}_{z_0}\geq \frac{c_0}{c}\quad\text{on }H^+_{\frac\rho4,2\varsigma}\cap \mathcal{A}_{\frac{\rho}4,{2\varsigma}}
\end{align}
for some constant $c=c(n,\Lambda,\varsigma)$, where $\rho$ is universally determined at the first line of the proof. 

Therefore, by repeating all the previous arguments for ${f}_{z_0}$ instead of $f$, and using that ${\varphi_0}\geq\frac1c$ in $H_{\frac\rho2,\varsigma}^+\cup\mathcal{A}_{\frac\rho4,2\varsigma}$, we get 
\begin{align}\label{last1.ineq.zero}
    f_{z_0}(z)\geq \frac{c_0}{c}{\varphi_0}(z)\quad\text{in }H_{\frac1{C}}\setminus \mathcal{R}_-
\end{align}
for any $t_0\in[-(\frac\rho8)^2,0,0]$. In particular, we have used \eqref{last.ineq.zero} with $\varsigma=\frac{2r}{N\rho}$ so that the constant $c$ given in \eqref{last.ineq.zero} depends only on the given data. 

In order to complete the proof, it remains to derive a lower bound of $f_{z_0}$ in the set 
\begin{align*}
    S\coloneqq \big((I_{\frac1C}\times \mathrm{H}_{\frac{\rho}8} )\setminus H_{\frac1{C}} \big) \setminus \mathcal{R}_-.
\end{align*}
To this end, we apply \eqref{last.ineq.zero} with $\varsigma=\frac{1}{C}$, which yields
\begin{align*}
    f_{z_0}\geq \frac{c_0}{c} \quad \text{on }H^+_{\frac\rho4,\frac2C}\cap \mathcal{A}_{\frac{\rho}4,{\frac2C}},
\end{align*}
where the constant $c$ depends only on the given data, as the constant $C$ depends on the data. Since $S\subset H^+_{\frac\rho4,\frac2C}\cap \mathcal{A}_{\frac{\rho}4,{\frac2C}}$ and ${\varphi_0}\leq c$ in $S$ for some constant $c$ depending only on the data, we have $f_{z_0}\geq \frac{c_0}{c}{\varphi_0}$ in $S$.  
By combination of this estimate with \eqref{last1.ineq.zero}, we have shown altogether that 
\begin{align*}
    f_{z_0}\geq \frac{c_0}{c}{\varphi_0} \quad \text{in } (I_{\frac1C}\times \mathrm{H}_{\frac{\rho}8})\setminus \mathcal{R}_-.
\end{align*}
Since we have defined $f_{z_0}(z)=f(z_0\circ z)$ with $z_0=(t_0,0,0)$ and $t_0\in [-(\rho/8)^2,0]$,  we conclude 
\begin{align*}
    f\geq \frac{c_0}{c}{\varphi_0}\quad\text{in } {H}_{\frac{\rho}8}\setminus \mathcal{R}_-,
\end{align*}
which completes the proof. 
\end{proof}

Now we give the proof of the main result of this subsection. Due to \autoref{lem.hop.gen0}, it only remains to generalize the result to the case $F \not\equiv 0$. We achieve this by applying the  maximum principle.

\begin{proof}[Proof of \autoref{lem.hop.gen}]
First, fix a small constant $\varsigma_0 < \frac1{4}$, which will be determined later (see \eqref{choi.varsigma0}). We now fix a large universal constant $M$ and a large constant $\mathcal{M}_0=\mathcal{M}_0(n,\Lambda,\varsigma_0)$ determined in \autoref{lem.exp.posright} applied with $\varsigma=\varsigma_0$ and
$\delta=1$. Then we fix $\rho\coloneqq\frac{1}{M}$. Then by \autoref{lem.exp.posright} with $r=\rho$, we have for $\mathcal{M}\geq\mathcal{M}_0$,
\begin{align*}
    \inf_{H^{+}_{\frac{\rho}2,\varsigma_0}}f\geq \frac{c_0}{c},
\end{align*}
where $c=c(n,\Lambda,\varsigma_0)$. Moreover, by assuming $\mathcal{M}\geq \mathcal{M}_1$ for some constant $\mathcal{M}_1=\mathcal{M}_1(n,\Lambda)$, we deduce from \autoref{lem.hargamma+} with $a=0$ and $r=\frac{\rho}{4}$ that $f\geq \frac{c_0}{c}$ on $\mathcal{A}_{\frac\rho4,2\varsigma_0}$, 
where the set $\mathcal{A}_{\frac{\rho}4,2\varsigma_0}$ is given in \eqref{set.mathcalar}.
Now we assume $\mathcal{M}\geq\mathcal{M}_0+\mathcal{M}_1$  to see that 
\begin{align}\label{lower.bdd}
    f\geq\frac{c_0}{c}\quad\text{in }H^+_{\frac\rho2,\varsigma_0}\cup \mathcal{A}_{\frac\rho4,2\varsigma_0}
\end{align}
for some constant $c=c(n,\Lambda,\varsigma_0)$. 

Next, we consider a weak solution $h_1$ of 
\begin{equation}\label{eq.h1.hop.gen}
\left\{
\begin{alignedat}{3}
\partial_th_1+v\cdot\nabla_{x}h_1-\ddiv(A\nabla_vh_1)&=B\cdot\nabla_vh_1&&\qquad \mbox{in  $\mathcal{H}_{\frac\rho2 ,\frac{\varsigma_0\rho}2}$}, \\
h_1&=1&&\qquad  \mbox{in $ \Gamma_{\frac\rho2,\frac{\varsigma_0\rho}2}$},\\
h_1&=0&&\qquad  \mbox{in $\partial_{\mathrm{kin}}\mathcal{H}_{\frac\rho2 ,\frac{\varsigma_0\rho}2}\setminus \Gamma_{\frac\rho2 ,\frac{\varsigma_0\rho}2}$},
\end{alignedat} \right.
\end{equation}
where the sets $\mathcal{H}_{\frac\rho2 ,\frac{\varsigma_0\rho}2}$ and $\Gamma_{\frac\rho2 ,\frac{\varsigma_0\rho}2}$ are defined in \eqref{defn.sethab}.
Since $ \frac{c_0}{c}h_1\leq f$  on  $\Gamma_{\frac\rho2 ,\frac{\varsigma_0\rho}2}\subset H^+_{\frac\rho2,\varsigma_0}\cup \mathcal{A}_{\frac\rho4,2\varsigma_0}$ due to \eqref{lower.bdd}, by the maximum principle, we have
\begin{align}\label{ineq2.hop.gen}
    h_1\leq \frac{c}{c_0}f\quad\text{in }\mathcal{H}_{\frac\rho2 ,\frac{\varsigma_0\rho}2}
\end{align}
for some constant $c=c(n,\Lambda,\varsigma_0)$.

Now we will choose a constant $\varsigma_0$ to apply \autoref{lem.hop.gen0} to an appropriately rescaled version of $h_1$. The argument is close to \textbf{Step 1} in the proof of \autoref{lem.hop.gen0}. As in \eqref{h1.12.gen.hopf.zero}, there is a small constant $a_0=a_0(n,\Lambda)\leq\frac12$ such that
for any $z_1=(t_1,x_{1}',x_{1,n},v_1)=(t_1,0,(\frac{\varsigma_0\rho}{2})^3,0)$ with $t_1\in[-(\rho/4)^2,0]$,
\begin{align}\label{h1.12.gen.hopf}
    h_1(t_1,0,({\varsigma_0\rho}/{2})^3(1-a_0),0)\geq\frac12.
\end{align}
Let $M_1=M_1(1/\delta_0)$ be the constant determined in \autoref{lem.exp.posright}, where $\delta_0\coloneqq (1-a_0)^{-\frac13}-1$. 
Now we choose 
\begin{equation}\label{choi.varsigma0}
    \varsigma_0\coloneqq \frac{(1-a_0)^{-\frac13}}{2M_1},
\end{equation}
which gives ${h}_1(-(\widetilde{\rho}_0M_1)^2,0,(\widetilde{\rho}_0)^3,0)\geq\frac12$ with $\widetilde{\rho}_0\coloneqq \frac{\varsigma_0\rho}{2}(1-a_0)^{\frac13}$, by \eqref{h1.12.gen.hopf}. Note that the constant $\varsigma_0$ depends only on $n,\Lambda$, as $a_0$ and $\delta_0$ depend only on $n,\Lambda$. Therefore, applying \autoref{lem.exp.posright} with $r=\widetilde{\rho}_0$, $\delta=\delta_0$ and $\varsigma=\frac1{4M}$,
we have
\begin{align}\label{choi.M0.hop.gen}
    \frac1{c_1}\leq \inf_{H^+_{\frac{\widetilde{\rho}_0}2,\frac1{4M}}}{h}_1,
\end{align}
where  $M$ is the universal constant determined in \autoref{lem.exp.posright} with $\delta=1$ and $c_1=c_1(n,\Lambda)$. Therefore, we define $\overline{h}_1(z)\coloneqq h_1(S_{\frac{\widetilde{\rho}_0}2}z)$ to see that 
 \begin{equation*}
\left\{
\begin{alignedat}{3}
\partial_t\overline{h}_1+v\cdot\nabla_x\overline{h}_1-\ddiv(\overline{A}\nabla_v\overline{h}_1)&=\overline{B}\cdot\nabla_v\overline{h}_1&&\qquad \mbox{in  ${H}_4$}, \\
\overline{h}_1&=0&&\qquad  \mbox{in $\gamma_-\cap H_4$},
\end{alignedat} \right.
\end{equation*}
where $\overline{A}(z)\coloneqq A(S_{\frac{\widetilde{\rho}_0}2}z)$, $\overline{B}(z)\coloneqq B(S_{\frac{\widetilde{\rho}_0}2}z)$. Note from the first lines in the proof of \autoref{lem.hop.gen0} that the constant $\rho=\frac1M$ is the same one as in \autoref{lem.hop.gen0}. Moreover, as $-(\frac{\widetilde{\rho}_0}{2})^2=-(\frac{\varsigma_0\rho}{8})^2(1-a_0)^{\frac23}\geq (-\frac{\rho}{8})^2$ and $\frac{\rho\widetilde{\rho}_0}{2}\in [\frac{\widetilde{\rho}_0}{8M},\frac{\widetilde{\rho}_0}2]$, we have $\overline{h}_1(-1,0,\rho^3,0)\in[\frac1{c_1},1]$ by \eqref{choi.M0.hop.gen}. Therefore, by \autoref{lem.hop.gen0}, we have 
\begin{align*}
    \frac{1}{c}\leq \inf_{H_{\frac{\rho}8}\setminus \mathcal{R}_-}\frac{\overline{h}_1}{{\varphi_0}}
\end{align*}
for some constant $c=c(n,\Lambda,\eps,\|A\|_{C^\eps(H_1)},\|B\|_{L^\infty(H_1)})$. A combination of this and \eqref{ineq2.hop.gen} proves the following fact: There is $C_0=C_0(n,\Lambda)$, such that whenever $\mathcal{M}\geq C_0$, then it holds
\begin{align}\label{ineq.last0.hop.gen}
    \frac1c\leq \inf_{H_{a\rho}\setminus \mathcal{R}_-}\frac{f}{{\varphi_0}}
\end{align}
for some constants $c=c(n,\Lambda,\eps,\|A\|_{C^\eps(H_1)},\|B\|_{L^\infty(H_1)})$ and $a=a(n,\Lambda)\in(0,1)$. Here, we recall that $\widetilde{\rho}_0=\frac{\varsigma_0\rho}{2}(1-a_0)^{\frac13}$ and that constants $a_0$ and  $\varsigma_0$ depend only on $n,\Lambda$. In particular, the constants $a_0$ and $\varsigma_0$ are determined in \eqref{h1.12.gen.hopf} and \eqref{choi.varsigma0}, respectively.

It remains to obtain the estimate \eqref{ineq.last0.hop.gen} with $\frac{\rho}{8}$ instead of $a\rho$.

To this end, we need to prove a lower bound of $f$ in $ (H_{\frac{\rho}8}\setminus H_{a\rho} )\setminus \mathcal{R}_-$. First note that  the constant $a=a(n,\Lambda)$ is fixed and depends only on $n,\Lambda$. Using this, as in \eqref{lower.bdd}, there is a large constant $\mathcal{M}_2=\mathcal{M}_2(n,\Lambda)$ such that if $\mathcal{M}\geq\mathcal{M}_2$, then 
\begin{align}\label{lower.bdd2}
    f\geq\frac{c_0}{c}\quad\text{in }H^+_{\frac\rho2,a}\cup \mathcal{A}_{\frac\rho4,2a}
\end{align}
for some constant $c=c(n,\Lambda)$. Now we fix $\mathcal{M}\coloneqq C_0+\mathcal{M}_2$, which depends only on $n,\Lambda$. Therefore, using \eqref{ineq.last0.hop.gen} and \eqref{lower.bdd2} together with the fact that ${\varphi_0}\eqsim c(n,\Lambda)$ in $ (H_{\frac{\rho}8}\setminus H_{a\rho} )\setminus \mathcal{R}_-$ by \eqref{property.phi_0}, we derive
\begin{align}\label{last2.hop.gen}
    \frac{c_0}c\leq \inf_{(I_{a\rho}\times \mathrm{H}_{\frac{\rho}8})\setminus \mathcal{R}_-}\frac{f}{{\varphi_0}}.
\end{align}
As in the last part given in \autoref{lem.hop.gen0}, by showing  \eqref{last2.hop.gen} for each time level $t\in[-(\frac{\rho}{8})^2,0]$, we deduce the desired estimate. This completes the proof.

\end{proof}

Next we provide the following auxiliary result which was used in the proof of \autoref{lem.hop.gen0}. It shows expansions of order $\frac{a}{4 + a} \in (0,\frac{1}{2})$ at $\gamma_0$ for solutions to equations with source terms that explode mildly at $\gamma_0$.
The proof follows from \autoref{lem.hol12} by a perturbation argument.

\begin{lemma}\label{lem.point.sol}
  Let $A\in C^\eps(H_r)$ with $\eps\in(0,1)$ and $r\in(0,1]$. Let $f$ be a weak solution to 
      \begin{equation*}
\left\{
\begin{alignedat}{3}
\partial_tf+v\cdot\nabla_xf-\ddiv({A}\nabla_ v{f})&=G_0-\ddiv(G_1)&&\qquad \mbox{in  $H_r$}, \\
{f}&=0&&\qquad  \mbox{in $\{x_n=0\}\times \{v_n>0\}$},
\end{alignedat} \right.
\end{equation*} 
where $G_i$ satisfy
\begin{align}\label{ass.G}
    |G_i(z_1)|\leq c_i\max\{(x_{1,n})^{\frac13},|v_{1,n}|\}^{-\frac12}
\end{align}
for some $c_i \ge 0$. Then for any  $a\in(0,4)$ and
$\rho\in(0,\frac{r}{2}]$, we have
\begin{align*}
    |f(0,0,\rho^3,0)|\leq c\left(\frac{\rho}{r}\right)^{\frac{a}{4+a}}\left(\left(\dashint_{H_{r}}|f|^2\,dz\right)^{\frac12}+c_0r^{\frac{a}{4+a}+1}+c_1r^{\frac{a}{4+a}}\right)
\end{align*}
for some constant $c=c(n,\Lambda,\eps,\|A\|_{C^\eps(H_r)},a)$.
\end{lemma}
\begin{proof}
First, we observe from \eqref{ass.G} that
\begin{align}\label{G.int}
    \left(\dashint_{H_r}|G_i|^{4+a}\,dz\right)^{\frac{1}{a+4}}\leq cc_0r^{-\frac{4}{a+4}}\left(\int_{\mathrm{H}_1}\max\{(x_{1,n})^{\frac13},|v_{1,n}|\}^{-\frac{a+4}{2}}\,dx_{n}\,dv_n\right)^{\frac{1}{a+4}}\leq  cc_ir^{-\frac{4}{a+4}}
\end{align}
for some constant $c=c(a)$.
    First, we are going to prove that for any $z_1=(0,0,x_{1,n},0)\in H_{r/2}$ and $\rho\leq r/2$,
    \begin{align}\label{des.est.12delta}
       E(f;H_{\rho}(z_1))\leq c\left(\frac{\rho}{r}\right)^{\frac{a}{4+a}}\left(\left(\dashint_{H_{r}}|f|^2\,dz\right)^{\frac12}+(c_0r+c_1)r^{\frac{a}{a+4}}\right),
    \end{align}
    where $c=c(n,\Lambda,\eps,\|A\|_{C^\eps(H_r)},a)$ and we write
    \begin{align*}
        E(f;H_{\rho}(z_1))\coloneqq  \left(\dashint_{H_\rho(z_1)}|f-(f)_{H_\rho(z_1)}|^2\,dz\right)^{\frac12}.
    \end{align*} To this end, we divide the proof into two cases. Within this proof, all constants $c$ depend only on $n,\Lambda,\eps,\|A\|_{C^\eps(H_r)},a$.
    
    \begin{itemize}
        \item Let $z_1=0$. Let $h$ be a weak solution to 
         \begin{equation*}
\left\{
\begin{alignedat}{3}
\partial_th+v\cdot\nabla_xh-\ddiv({A}\nabla_ v{h})&=0&&\qquad \mbox{in  $H_{r/4}$}, \\
{h}&=f&&\qquad  \mbox{on $\{x_n=0\}\times \{v_n>0\}$}.
\end{alignedat} \right.
\end{equation*} 
Then, by using the Poincar\'e inequality and testing the equation 
\begin{align*}
    \partial_t(f-h)+v\cdot\nabla_x(f-h)-\ddiv({A}\nabla_ v{(f-h)})=G_0+\ddiv(G_1)\quad\text{in }H_{\frac{r}4},
\end{align*}
with  $h-f$, we deduce 
\begin{align*}
    J\coloneqq\dashint_{H_{\frac{r}{4}}}|h-f|^2\,dz\leq c\dashint_{H_{\frac{r}4}}r^2|\nabla_v(h-f)|^2\,dz\leq cr^2\left(\dashint_{H_{\frac{r}4}}|r\nabla_v(h-f)|^2+\dashint_{H_{\frac{r}4}}|rG_0|^2+|G_1|^2\,dz\right).
\end{align*}
Moreover, note that by testing the equation for $h$ with $h-f$, we get
\begin{align*}
    \dashint_{H_{\frac{r}4}}r^2|\nabla_v h|^2\,dz \le c\dashint_{H_{\frac{r}4}}r^2|\nabla_v f|^2\,dz.
\end{align*}
 Next, we use energy estimates of $f$, the H\"older inequality and the observation \eqref{G.int} so that we further estimate $J$ as
\begin{align*}
    J&\leq cr^2\left(\dashint_{H_{\frac{r}{2}}}|f|^2\,dz+\dashint_{H_{\frac{r}2}}|rG_0|^2+|G_1|^2\,dz\right)\\
    &\leq cr^2\left(\dashint_{H_{\frac{r}{2}}}|f|^2\,dz+\left(\dashint_{H_{\frac{r}2}}|rG_0|^{4+a}\,dz\right)^{\frac{2}{a+4}}+\left(\dashint_{H_{\frac{r}2}}|G_1|^{4+a}\,dz\right)^{\frac{2}{a+4}}\right)\\
    &\leq cr^2\left(\dashint_{H_{\frac{r}{2}}}|f|^2\,dz+c_1^2r^{2-\frac{8}{a+4}}+c_0^2r^{-\frac{8}{a+4}}\right).
\end{align*}
Next, as in \eqref{ineq3.gragamma-}, from \autoref{lem.hol12}, we deduce for any $\sigma\in(0,\frac1{16}]$,
\begin{align*}
    \dashint_{H_{{\sigma r}}}|h|^2\leq c\sigma\dashint_{H_{\frac{r}{4}}}|h|^2\,dz.
\end{align*}
Thus, we have 
\begin{align*}
    \dashint_{H_{\sigma r}}|f|^2\,dz&\leq  c\left(\dashint_{H_{\sigma r}}|h|^2\,dz+\sigma^{-(4n+2)}J\right)\\
    &\leq c\sigma\left(\dashint_{H_{\frac{r}{2}}}|f|^2\,dz+ (1 + \sigma)\sigma^{-(4n+2)} J \right)\\
    &\leq c\sigma\left(\dashint_{H_{\frac{r}{2}}}|f|^2\,dz+\sigma^{-(4n+2)}r^2\left[\dashint_{H_{\frac{r}{2}}}|f|^2\,dz+c_0^2r^{2-\frac{8}{a+4}}+c_1^2r^{-\frac{8}{a+4}}\right]\right).
\end{align*}
We can rewrite this as
\begin{align*}
     \int_{H_{\sigma r}}|f|^2\,dz\leq c\left((\sigma^{4n+3}+r^2)\int_{H_{\frac{r}{2}}}|f|^2\,dz+(c_0^2r^2+c_1^2)r^{4n+2+\frac{2a}{a+4}}\right).
\end{align*}
Thus, if $r\leq r_0$ for some small $r_0$, then by \cite[Lemma 3.4]{HaLi11}, we deduce for any $\rho\leq r$,
\begin{align*}
    \int_{H_{ \rho}}|f|^2\,dz\leq c\left(\left(\frac{\rho}{r}\right)^{4n+2+\frac{2a}{a+4}}\int_{H_{\frac{r}{2}}}|f|^2\,dz+(c_0^2r^2+c_1^2)\rho^{4n+2+\frac{2a}{a+4}}\right).
\end{align*}
Here, we also used that $4n+2+\frac{2a}{a+4} \le 4n + 3$ since $a < 4$.

Thus we have for any $\rho \leq r\leq r_0$,
\begin{align}\label{first.goal.potsol}
     \dashint_{H_{ \rho}}|f|^2\,dz\leq c\left(\frac{\rho}{r}\right)^{\frac{2a}{a+4}}\dashint_{H_{{r}}}|f|^2\,dz+(c_0^2r^2+c_1^2)\rho^{\frac{2a}{a+4}}.
\end{align}
When $r>r_0$, then $r\in(r_0,1]$. Thus, we have for any $\rho\leq r_0$
\begin{align*}
     \left(\dashint_{H_{ \rho}}|f|^2\,dz\right)^{\frac12}&\leq c\left(\frac{\rho}{r_0}\right)^{\frac{a}{4+a}}\left(\left(\dashint_{H_{r_0}}|f|^2\,dz\right)^{\frac12}+(c_0r_0+c_1)r_0^{\frac{a}{4+a}}\right)\\
     &\leq c\left(\frac{\rho}{r}\right)^{\frac{a}{4+a}}\left(\left(\dashint_{H_{r}}|f|^2\,dz\right)^{\frac12}+(c_0r+c_1)r^{\frac{a}{4+a}}\right),
\end{align*}
as $r\in(r_0,1]$ and the constant $r_0$ depends only on the given data. Moreover, when $\rho\in(r_0,r/2]$, we directly obtain the above estimates, as $(\rho/r)\eqsim 1$. Therefore, \eqref{des.est.12delta} follows from \eqref{first.goal.potsol}.
\item Let $z_1=(0,0,x_{1,n},0)$ with $x_{1,n}\neq0$. Let $r_1\coloneqq \frac{(x_{1,n})^{\frac13}}{4}$ to see that 
\begin{align}\label{set.inc.potsol}
    H_{r_1}(z_1)\cap \{x_{n}>0\}=\emptyset\quad\text{and}\quad H_{r_1}(z_1)\subset H_{8r_1}.
\end{align}
Suppose $\rho\leq r_1/2$. Then by \autoref{lem.intsch}, $L^\infty-L^2$ estimates by \cite[Theorem 1.1]{Sil22} and \eqref{ass.G}, 
\begin{align*}
    r_1^{\frac12}[f]_{C^{\frac12}(H_{r_1/2}(z_1))}&\leq c\left(\|f\|_{L^\infty(H_{3r_1/4}(z_1))}+r_1^2\|G_0\|_{L^\infty(H_{3r_1/4}(z_1))}+r_1\|G_1\|_{L^\infty(H_{3r_1/4}(z_1))}\right)\\
    &\leq c\left(\left(\dashint_{H_{r_1}(z_1)}|f|^2\,dz\right)^{\frac12}+(c_0r_1+c_1)r_1^{\frac12}\right),
\end{align*}
which implies 
\begin{align*}
    E(f;H_{\rho}(z_1))\leq c\left(\frac{\rho}{r_1}\right)^{\frac12}\left(\left(\dashint_{H_{r_1}(z_1)}|f|^2\,dz\right)^{\frac12}+(c_0r_1+c_1)r_1^{\frac12}\right).
\end{align*}
By \eqref{set.inc.potsol} and \eqref{first.goal.potsol}, we have 
\begin{align*}
    E(f;H_{\rho}(z_1))&\leq c\left(\frac{\rho}{r_1}\right)^{\frac12}\left(\left(\dashint_{H_{8r_1}}|f|^2\,dz\right)^{\frac12}+(c_0r_1+c_1)r_1^{\frac12}\right)\\
    &\leq c\left(\frac{\rho}{r}\right)^{\frac{a}{a+4}}\left(\left(\dashint_{H_{r}}|f|^2\,dz\right)^{\frac12}+(c_0r+c_1)r^{\frac{a}{a+4}}\right).
\end{align*}
On the other hand, when $\rho\geq r_1/2$, we can directly use \eqref{first.goal.potsol} to see that by $H_{\rho}(z_1)\subset H_{8\rho}$,
\begin{align*}
    E(f;H_{\rho}(z_1))
    \leq c\left(\frac{\rho}{r}\right)^{\frac{a}{a+4}}\left(\left(\dashint_{H_{r}}|f|^2\,dz\right)^{\frac12}+(c_0r+c_1)r^{\frac12}\right).
\end{align*}
    \end{itemize}
We have verified \eqref{des.est.12delta}. Thus, we have for any $z_1=(0,0,\rho^3,0)$,
\begin{align*}
    |f(z_1)|\leq \sum_{i=1}^\infty|(f)_{H_{2^{-i}\rho}(z_1)}-(f)_{H_{2^{-(i+1)}\rho}(z_1)}|+|(f)_{H_{\rho/2}(z_1)}|&\leq \sum_{i=1}^\infty E(f;H_{2^{-i}\rho}(z_1))+c|(f)_{H_{4\rho}}|.
\end{align*}
where we have used that $H_{\rho/2}(z_1)\subset H_{4\rho}$. Therefore, a combination of this, \eqref{des.est.12delta} and \eqref{first.goal.potsol} leads to 
\begin{align*}
    |f(z_1)| &\le c \left( \sum_{i=1}^\infty 2^{-i \left( \frac{a}{a+4} \right) } + 1 \right) \left(\left(\dashint_{H_{r}}|f|^2\,dz\right)^{\frac12}+(c_0r+c_1)r^{\frac12}\right)
\end{align*}

This completes the proof.
\end{proof}

\subsection{Hopf lemma in non-flat domains}

We end this section with the proof of the Hopf lemma in a general curved domain (see \autoref{lem.hopf.curve}). 

\begin{proof}[Proof of \autoref{lem.hopf.curve}]
      By \cite[Section 8]{KiWe26}, there is a $C^{k+\eps}$ diffeomorphism $\varphi^{-1}$ such that for $\Psi(t,y,w)\coloneqq (t,\varphi^{-1}(y),\nabla\varphi^{-1}(y)w)$, 
   \begin{align*}
       I_2\times \mathrm{H}_{{R_1}}\subset \Psi(I_2\times (\Omega\cap B_{R_0^3})\times B_{R_0} )\subset I_2\times \mathrm{H}_{2R_1}
   \end{align*}
   for some constants $R_0,R_1\leq1$ depending only on $\Omega$.
   In particular, $\widetilde{f}\coloneqq f\circ \Psi^{-1}$ is a weak solution to 
    \begin{equation*}
\left\{
\begin{alignedat}{3}
\partial_t\widetilde{f}+v\cdot\nabla_{x}\widetilde{f}-\ddiv(\widetilde{A}\nabla_v\widetilde{f})&=\widetilde{B}\cdot\nabla_v\widetilde{f}+\widetilde{F}&&\qquad \mbox{in  $I_2\times \mathrm{H}_{2R_1}$}, \\
\widetilde{f}&=0&&\qquad  \mbox{in $\{x_n=0\}\times \{v_n>0\}$},
\end{alignedat} \right.
\end{equation*} 
where $\widetilde{A}$, $\widetilde{B}$, and $\widetilde{F}$ are defined as in \cite[(8.3)]{KiWe26}. In particular, by \cite[Lemma 2.17]{RoWe25}, we have 
\begin{align*}
    \|\widetilde{A}\|_{C^\eps(H_{R_1})}\leq c\|A\|_{C^\eps(\mathbf{H}_2(z_0))},\quad  \|\widetilde{B}\|_{L^\infty(H_{R_1})}\leq c\|B\|_{L^\infty(\mathbf{H}_2(z_0))},\quad \|\widetilde{F}\|_{L^\infty(H_{R_1})}\leq \|F\|_{L^\infty(\mathbf{H}_{2}(z_0))},
\end{align*}
where $c=c(n,\Lambda,\eps,\Omega)$. In addition, using the scaling argument from \cite[Proposition 5.1]{KiWe26}, by considering $\widetilde{f}_a(z)=\widetilde{f}(t,ax,av)$ for $a\coloneqq \sqrt{(\widetilde{A}(0))_{n,n}}$, we can assume $(\widetilde{A}(0))_{n,n}=1$. Moreover, we now choose a constant $\rho_1=\rho_1(\Omega)$ such that $\Psi(-1,x_0-\rho_1^3n_{x_0},0)=(-1,0,(\rho R_1)^3,0)\in \bbR\times \bbR^{n-1}\times \bbR\times \bbR^n$, where $\rho$ is the universal constant determined in \autoref{thm.mainhar}. Next, by choosing a sufficiently small constant $\mathcal{M}=\mathcal{M}(n,\Lambda,\Omega)$ and using the Harnack chain argument as in \eqref{harnack.ineq10f} with $v_1=0$, we obtain $\widetilde{f}(-(R_1)^2,0,(\rho R_1)^3,0)=\frac{c_0}{c}$ for some constant $c=c(n,\Lambda,\Omega)$. Now we apply a rescaled version of \autoref{lem.hop.gen} to the function $\widetilde{f}$ to get 
\begin{align*}
    \frac{c_0}{c}\leq \inf_{H_{\frac{\rho R_1}8}}\frac{\widetilde{f}}{{\varphi_0}}
\end{align*}
for some constant $c=c(n,\Lambda,\eps,\|\widetilde{A}\|_{C^\eps(H_{R_1})},\|\widetilde{B}\|_{L^\infty(H_{R_1})})$. Now, using the fact that $\Psi^{-1}(\{x_n\leq (v_n)^3\})\subset \mathcal{R}_-$, ${\varphi_0}(\Psi(z))={\varphi_0}(d_\Omega(x),n_x\cdot v)$ by \cite[(8.4)]{KiWe26}, scaling back and \eqref{notationvarphi}, we get the desired result. Since we have considered the function $\widetilde{f}_a(z)(=\widetilde{f}(t,ax,av))$ to assume $(\widetilde{A}(0))_{n,n}=1$, the constant $\rho_1$ eventually depends on $\Omega,\Lambda$. This completes the proof.
\end{proof}

\section{Higher order boundary Harnack estimates}
\label{sec:bdry-Harnack}

In this section, we will prove the following boundary Harnack principle.
\begin{theorem}\label{thm.mainhar}
    Let $k\in\{0,1\}$ $\eps\in(0,1)\setminus \{\frac12\}$ with $k+\eps<\frac32$ and  let $A\in C^{k+\eps}(H_2)$, $B\in C^{k+\eps-1}(H_2)$ $F_i\in C^{k+\eps-\frac12}(H_2)$ for each $i\in\{1,2\}$. Let $f_i$ be weak solutions to 
    \begin{equation}\label{mainhar:eq}
\left\{
\begin{alignedat}{3}
\partial_tf_i+v\cdot\nabla_{x}f_i-\ddiv(A\nabla_vf_i)&=B\cdot\nabla_vf_i+F_i&&\qquad \mbox{in  ${H}_2$}, \\
f_i&=0&&\qquad  \mbox{in $\gamma_-\cap H_2$}.
\end{alignedat} \right.
\end{equation} 
with $f_2,F_2\geq0$, and $(A(0))_{n,n}=1$. Suppose $\|f_i\|_{L^\infty(H_2)},\|F_i\|_{C^{\eps-\frac12}(H_2)}\leq 1$. There exists a universal constant $\rho<1$ and a large constant $\mathcal{M}=\mathcal{M}(n,\Lambda)$ such that if $f_2(z_0)=c_0>0$ for some constant $c_0$ with $z_0=(t_0,x_0',x_{0,n},v_0)=(-1,0,\rho^3,0)$ and 
\begin{align*}
    \|F_2\|_{L^\infty(H_2)}\leq \frac{c_0}{\mathcal{M}},
\end{align*}
then
\begin{align*}
    \left[\frac{f_1}{f_2}\right]_{C^{\min\{1+\eps,3/2\}}(H_{\frac\rho8}\setminus\mathcal{R}_-)}\leq c,
\end{align*}
where $c=c(n,\Lambda,\eps,c_0,\|A\|_{C^{\eps}(H_2)},\|B\|_{C^{\eps-1}(H_2)})$. 
Moreover, when $F_i\equiv 0$, we have 
\begin{align}\label{ineq.main2.mainhar}
     \left[\frac{f_1}{f_2}\right]_{C^{1,\max\{\eps,k\}}(H_{\frac\rho8}\setminus\mathcal{R}_-)}\leq c
\end{align}
for some constant $c=c(n,\Lambda,\eps,c_0,\|A\|_{C^{k+\eps}(H_2)},\|B\|_{C^{k+\eps-1}(H_2)})$, where $\mathcal{R}_-=\{x_n\leq v_n^3\}$.
\end{theorem}

To obtain this result, the key step is to prove an expansion of $f_1/f_2$ at $\gamma_0$ (see  \autoref{lem.exp.bdry.diff}) under the assumption that ${\varphi_0}\lesssim f_2$ in $H_2\setminus \mathcal{R}_-$. Afterwards, we combine this expansion with regularity estimates away from $\gamma_0\cup\gamma_-$ (i.e. interior estimates and boundary estimates up to $\gamma_+$) to derive uniform regularity estimates in $\mathcal{R}_+ \cup \mathcal{R}_0$  (see \autoref{lem.bdryhar}). Finally, we employ \autoref{lem.hop.gen} to replace the  assumption ${\varphi_0}\lesssim f_2$ in $H_2\setminus \mathcal{R}_-$ by the assumption that $f(z_0)=c_0>0$.

\subsection{Auxiliary results}
\label{subsec:aux}
In this section, we provide several auxiliary lemmas that will be used for the proof of the main result.

We first prove the following fact about the functions $\varphi_0, \varphi_1, \varphi_2, \varphi_3$, which are determined in \eqref{notationvarphi}:

\begin{lemma}
\label{lem.relphi0123}
     Assume $a^{n,n} = 1$. Then, for any function $\varphi\in \mathrm{span}\{{\varphi_0}, v_i\partial_{v_n}{\varphi_0},t\partial_{v_n,v_n}{\varphi_0},v_iv_j\partial_{v_n,v_n}{\varphi_0}\}$, there is a function $\phi\in \mathrm{span}\{{\varphi_3},v_iv_j{\varphi_0},v_iv_j{\varphi_1},t{\varphi_0},t{\varphi_1}\}$ such that
\begin{equation}\label{relphi0123}
\left\{
\begin{alignedat}{3}
    \partial_t\phi+v\cdot\nabla_x\phi-a^{i,j}\partial_{v_i,v_j}\phi &= \varphi&& \quad \text{in } \R \times \{ x_n > 0 \} \times \R^n,\\
    \phi &=0&& ~~ \text{ on } \R \times \{ x_n = 0 \} \times \R^n.
\end{alignedat} \right.
\end{equation}
\end{lemma}

\begin{proof}
We deduce from \eqref{eq.psi0}, \eqref{eq.vdphi0}, \eqref{notationvarphi}, and since $a^{n,n} = 1$ that \begin{equation*}
\begin{aligned}
v_n\partial_{x_n}(v_n^2{\varphi_0})-\partial_{v_n,v_n}(v_n^2{\varphi_0}) &=-2{\varphi_0}-4{\varphi_1},\\
v_n\partial_{x_n}(v_n^2{\varphi_1})-\partial_{v_n,v_n}(v_n^2{\varphi_1}) &=-6{\varphi_1}-7v_n^2\partial_{v_n,v_n}{\varphi_0},\\
 \partial_t(t{\varphi_1})+v_n\partial_{x_n}(t{\varphi_1})-\partial_{v_n,v_n}(t{\varphi_1})&={\varphi_1}-3t\partial_{v_n,v_n}{\varphi_0},\\
    v\cdot\nabla_x(v_lv_k{\varphi_0})-a^{i,j}\partial_{v_i,v_j}(v_lv_k{\varphi_0})&=-2a^{l,k}{\varphi_0}-2(a^{l,n}v_k+a^{k,n}v_l)\partial_{v_n}{\varphi_0},\\
    v\cdot\nabla_x(v_lv_k{\varphi_1})-a^{i,j}\partial_{v_i,v_j}(v_lv_k{\varphi_1})&=-3v_lv_k\partial_{v_n,v_n}{\varphi_0}-2a^{l,k}{\varphi_1}\\
&\qquad-2(a^{l,n}v_k+a^{k,n}v_l)(v_n\partial_{v_n,v_n}{\varphi_0}+\partial_{v_n}{\varphi_0}).
\end{aligned}
\end{equation*}
By the aforementioned observations and \eqref{eq.dphi1}, we immediately deduce the desired result.
\end{proof}

\begin{lemma}
For $p_0,p_1,p_2,p_3\in \cP_2$,  we have 
\begin{align}\label{eq.norm}
    \left\|\sum_{i=0}^3p_i{\varphi_i}\right\|_{L^\infty( H_1\setminus \mathcal{R}_-)}\eqsim_n \sum_{i=0}^3\|p_i\|_{L^\infty(H_1)}.
\end{align}
In addition, if ${\varphi_0}/c_2\leq f\leq c_2{\varphi_0}$ in $H_1\setminus \mathcal{R}_-$ for some constant $c_2\geq1$ and $p_0(0)=p_1(0)=0$, then for any $r\leq1$ and $a \in \R$,
\begin{align}\label{eq.norm2}
    \left\|af+\sum_{i=0}^3p_i{\varphi_i}\right\|_{L^\infty(H_r\setminus \mathcal{R}_-)}\eqsim_{n,c_2} r^{\frac12}|a|+\sum_{i=0}^3\|r^{\alpha_i}p_i(S_rz)\|_{L^\infty(H_1)},
\end{align}
where we write
\begin{equation}\label{homo.alphai}
    \mbox{$\alpha_0=\alpha_1=\frac12$, $\alpha_2=2$ and $\alpha_3=\frac52$}.
\end{equation}

\end{lemma}
\begin{proof}
   The first relation \eqref{eq.norm} follows from the fact that 
\begin{align}\label{upperbound}
    \frac1c\leq \| {\varphi_i}\|_{L^\infty( H_1\setminus \mathcal{R}_-)} \leq c
\end{align}
for some constant $c\geq1$, and $\{{\varphi_i}\}_i$ is a linearly independent set. For \eqref{eq.norm2}, first we prove
\begin{align}\label{ineq1.norm2}
    \left\|af+\sum_{i=0}^3p_i{\varphi_i}\right\|_{L^\infty(H_1\setminus \mathcal{R}_-)}\geq \frac1c\left(|a|+ \left\|\sum_{i=0}^3p_i{\varphi_i}\right\|_{L^\infty( H_1\setminus \mathcal{R}_-)}\right)
\end{align}
for some constant $c=c(n,c_2)$. 
First, we assume $ M|a|\leq \left\|\sum\limits_{i=0}^3p_i{\varphi_i}\right\|_{L^\infty( H_1\setminus \mathcal{R}_-)}$
for some constant $M=2c_2\|\varphi_0\|_{L^\infty(H_1)}$.
Then, we have 
\begin{align*}
    \left\|af+\sum_{i=0}^3p_i{\varphi_i}\right\|_{L^\infty(H_1\setminus \mathcal{R}_-)}\geq \left\|\sum_{i=0}^3p_i{\varphi_i}\right\|_{L^\infty(H_1\setminus \mathcal{R}_-)}-\|af\|_{L^\infty(H_1\setminus \mathcal{R}_-)}\geq \frac12\left\|\sum_{i=0}^3p_i{\varphi_i}\right\|_{L^\infty(H_1\setminus \mathcal{R}_-)}.
\end{align*}
Thus, we assume  $  M|a|\geq \left\|\sum\limits_{i=0}^3p_i{\varphi_i}\right\|_{L^\infty( H_1\setminus \mathcal{R}_-)}$, which implies $\|p_i\|_{L^\infty(H_1)}\leq c(n)M|a|$ by \eqref{eq.norm}. Note from \eqref{homogen} and the fact that $p_0(0)=p_1(0)=0$, that there is a small constant $\rho=\rho(n,c_2)$ such that 
\begin{align*}
    \left\|\sum_{i=0}^3p_i{\varphi_i}\right\|_{L^\infty( H_\rho\setminus \mathcal{R}_-)}\leq \sum_{i=0}^3\|p_i\|_{L^\infty(H_\rho)}\|\varphi_i\|_{L^\infty(H_\rho\setminus \mathcal{R}_-)}\leq c(n)M|a|\rho^{\frac32}\leq \frac{|a|}{2c_2}\|\varphi_0\|_{L^\infty(H_\rho\setminus\mathcal{R}_-)},
\end{align*}
as the constant $M$ depends only on $c_2$. Thus, we have
\begin{align*}
    \left\|af+\sum_{i=0}^3p_i{\varphi_i}\right\|_{L^\infty(H_\rho\setminus \mathcal{R}_-)}\geq \|af\|_{L^\infty(H_\rho\setminus \mathcal{R}_-)}-\left\|\sum_{i=0}^3p_i{\varphi_i}\right\|_{L^\infty(H_\rho\setminus \mathcal{R}_-)}&\geq \frac{\|af\|_{L^\infty(H_\rho\setminus \mathcal{R}_-)}}{2}\geq \frac{|a|}{c}
\end{align*}
for some constant $c=c(n,c_2)$. Altogether, we obtain \eqref{ineq1.norm2}. Now, a combination of \eqref{ineq1.norm2} and \eqref{upperbound} leads to \eqref{eq.norm2} with $r=1$. For any $r\in(0,1)$,  we use the scaling argument and \eqref{homogen} to get that
\begin{equation}\label{ineq2.norm2}
\begin{aligned}
\left\|af+\sum_{i=0}^3p_i{\varphi_i}\right\|_{L^\infty(H_r\setminus \mathcal{R}_-)}&=\left\|af(S_rz)+\sum_{i=0}^3p_i(S_rz){\varphi_i}(S_rz)\right\|_{L^\infty(H_1\setminus \mathcal{R}_-)}\\
    &=\left\|r^{\frac12}af_0+\sum_{i=0}^3r^{\alpha_i}p_i(S_rz){\varphi_i}(z)\right\|_{L^\infty(H_1\setminus \mathcal{R}_-)},
\end{aligned}
\end{equation}
where we write $f_0(z)=f(S_rz)/r^{\frac12}$ and it satisfies ${\varphi_0}/c\leq f_0\leq c{\varphi_0}$ in $H_1\setminus \mathcal{R}_-$ for some constant $c=c(c_2)$. Thus, the desired estimate follows from applying \eqref{eq.norm2} with $r=1$ into the last term given in \eqref{ineq2.norm2}. This completes the proof.
\end{proof}

\begin{lemma}\label{lem.aux.direction}
    Let $p\in \cP_2$ and $a_0\in\bbR$. If $p{\varphi_0}+a_0{\varphi_3}$ solves 
    \begin{align}\label{eq.aux.directioin}
        \partial_t(p{\varphi_0}+a_0{\varphi_3})+v\cdot\nabla_x(p{\varphi_0}+a_0{\varphi_3})-\ddiv(A\nabla_v(p{\varphi_0}+a_0{\varphi_3}))=0
    \end{align}
    for some uniformly elliptic constant coefficient matrix $A$ with $(A)_{n,n}=1$, then 
    \begin{align}\label{rel.vector}
        \sum_{j=1}^{n}(A)_{n,j}(\nabla_vp(0))_j=0.
    \end{align}
    In addition, if there is a vector $\mathbf{Q}\in\bbR^n$ and
    a sequence of uniformly elliptic constant coefficient matrices $(A_i)_i$ such that $A_i\to A$ with $(A_{i})_{n,n}=1$ and 
        \begin{align*}
        \sum_{j=1}^{n}(A)_{n,j}\mathbf{Q}_j=0,
    \end{align*}
    then there is a sequence of vectors $(\mathbf{Q}_i)_i$ such that $\mathbf{Q}=\lim\limits_{i\to\infty}\mathbf{Q}_i$ and
    \begin{align*}
        \sum_{j=1}^{n}(A_i)_{n,j}(\mathbf{Q}_i)_j=0.
    \end{align*}
\end{lemma}
\begin{proof}
   Let us write $p_2(z)\coloneqq p(z)-(p(0)+\nabla_vp(0)\cdot v)$ to see that $p_2(0)=\nabla_v p_2(0)=0$. Note that
    \begin{align*}
        &\partial_t(p{\varphi_0}+a_0{\varphi_3})+v\cdot\nabla_x(p{\varphi_0}+a_0{\varphi_1})-\ddiv(A\nabla_v(p{\varphi_0}+a_0{\varphi_3}))\\
        &=-2\sum_{j=1}^{n}(A)_{n,j}(\nabla_vp(0))_j\partial_{v_n}{\varphi_0}-2\sum_{j=1}^{n}(A)_{n,j}\partial_{v_j}p_2\partial_{v_n}{\varphi_0}+c_1{\varphi_0}
    \end{align*}
    for some constant $c_1\in\bbR$, where we have used the fact that $v_n\partial_{x_n}{\varphi_0}-\partial_{v_n,v_n}{\varphi_0}=0$. Thus, by \eqref{eq.aux.directioin} and since all elements of $\{\partial_{v_n}{\varphi_0},v_j\partial_{v_n}{\varphi_0},{\varphi_0}\}$ are linearly independent, we obtain
    \begin{align*}
        -2\sum_{j=1}^{n}(A)_{n,j}(\nabla_vp(0))_j=-2\sum_{j=1}^{n}(A)_{n,j}\partial_{v_j}p_2=c_1=0,
    \end{align*}
    which proves \eqref{rel.vector}.
    
    Next, suppose $(A)_n\cdot \mathbf{Q}=0 $. Now choose $\mathbf{Q}_i=\mathbf{Q}+q_ne_n$, where $e_n$ is a $n$-th unit normal vector and $q_n=((A)_n-(A_i)_n)\cdot \mathbf{Q}$. Then, we see that $(A_i)_n\cdot\mathbf{Q}_i=0$ and $\mathbf{Q}_i\to\mathbf{Q}$, as $(A_i)_{n,n}=1$ and $(A_i)_n\to (A)_n$. This completes the proof.
\end{proof}

\begin{lemma}\label{lem.rep.bdry}
   Let $\eps\in(0,1)$ and let $A\in C^{\eps}(H_1)$, $B\in L^\infty(H_1)$. Let $i\in\{0,1,2,3\}$. Let $f$ be a weak solution to 
    \begin{align*}
        (\partial_t+v\cdot\nabla_x)f-\ddiv(A\nabla_vf)=B\cdot\nabla_vf+F-\ddiv(G{\varphi_i})+g\partial_{v_n}{\varphi_i}\quad\text{in }H_1.
    \end{align*}
    Then we have 
    \begin{align*}
        \|F\|_{L^\infty(H_1)}\leq c(\|f\|_{L^\infty(H_1)}+\|G\|_{C^{\eps}(H_1)}+[F]_{C^\eps(H_1)}+\|g\|_{L^\infty(H_1)})
    \end{align*}
    for some constant $c=c(n,\Lambda,\eps,\|A\|_{C^\eps(H_1)},\|B\|_{L^\infty(H_1)})$.
\end{lemma}
\begin{proof}
    Note that the functions $\varphi_i$ are smooth inside $\{x_n>0\}$. Therefore, by following the proof of \cite[Lemma 5.3]{KiWe26}, we immediately derive the desired estimate.
\end{proof}
Moreover, from \cite[Lemma 5.2, Lemma 5.4]{KiWe26}, we deduce the following.
\begin{lemma}\label{lem.bdry.hol.phi}
  Let $A\in C^\alpha(H_R)$ and $B\in L^\infty(H_R)$ for some $\alpha\in(0,1)$.  Let $f$ be a weak solution to 
    \begin{equation*}
\left\{
\begin{alignedat}{3}
(\partial_t+v\cdot\nabla_x)f-\ddiv(A\nabla_vf)&=B\cdot\nabla_vf+F-\ddiv(G)&&\qquad \mbox{in  ${H}_R$}, \\
f&=0&&\qquad  \mbox{in $\gamma_-\cap {Q}_R$},
\end{alignedat} \right.
\end{equation*} 
where $F\coloneqq \sum\limits_{i=0}^2F_i\partial^{(i)}_{v_n}{\varphi_0}+F_3$, $G\coloneqq \sum\limits_{j=0}^1G_j\partial^{(i)}_{v_n}{\varphi_0}+G_2$ with $F_i\in C^{\lambda_i,\alpha'}$, $G_j\in C^{\mu_j,\alpha'}$ for some $\alpha'\in[0,1)$ and $\lambda_i,\mu_j\in \N\cup\{0\}$ for each $i\in\{0,1,2,3\}$ and $j\in\{0,1,2\}$. Then there is a small constant $\beta=\beta(n,\Lambda,\alpha)$ such that
\begin{align*}
    [f]_{C^{\beta}(H_{R/2})}\leq c\left(\|f\|_{L^\infty(H_R)}+\sum_{i=0}^{3}[F_i]_{C^{\lambda_i,\alpha'}(H_R)}+\sum_{j=0}^{2}[G_j]_{C^{\mu_j,\alpha'}(H_R)}\right),
\end{align*}
where $c=c(n,\Lambda,\lambda_i,\mu_j,\alpha,\alpha',\|A\|_{C^\alpha(H_R)},\|B\|_{L^\infty(H_R)},R)$.
\end{lemma}

\begin{proof}
    The result follows immediately from \cite[Lemma 5.2, Lemma 5.4]{KiWe26}.
\end{proof}

Finally, we provide a Schauder type regularity result away from $\gamma_-\cup\gamma_0$. It follows directly from a combination of \autoref{lem.intsch} and \autoref{lem.bdrysch}.

\begin{lemma}\label{lem.schcomb}
    Let $k\in\{1,2\}$, $\eps\in(0,1)$. Let $A\in C^{k-1,\eps}(H_r(z_0))$, $B\in C^{k-2,\eps}(H_r(z_0))$, $F\in C^{k-2,\eps}(H_r(z_0))$, and $G\in C^{k-1,\eps}(H_r(z_0))$ with $z_0\notin \mathcal{R}_-$ and $r\leq \frac{\mathrm{dist}(z_0,\gamma_0)}{4}$. Assume that $f$ is a weak solution to 
 \begin{align*}
     \partial_tf+v\cdot\nabla_xf-\ddiv(A\nabla_vf)=B\cdot\nabla_vf+F-\ddiv(G)\quad\text{in }H_{r}(z_0).
 \end{align*}
 Then we have 
    \begin{align*}
        [f]_{C^{k,\eps}(H_{r/2}(z_0))}\leq cr^{-(k+\eps)}\left(\|f\|_{L^\infty(H_{r}(z_0))}+r^{k+\eps}[F]_{C^{k-2,\eps}(H_r(z_0))}+r^{k+\eps}[G]_{C^{k-1,\eps}(H_r(z_0))}\right)
    \end{align*}
    for some constant $c=c(n,\Lambda,\eps,\|A\|_{C^{k-1,\eps}(H_r(z_0))},\|B\|_{C^{k-2,\eps}(H_r(z_0))})$.
\end{lemma}
\begin{proof}
Since $H_r(z_0)\cap (\gamma_0\cup \gamma_-)=\emptyset$, combining \autoref{lem.intsch} and \autoref{lem.bdrysch} together with a standard covering argument yields the desired estimate.
\end{proof}

\subsection{Expansions at the grazing set}

Now we are ready to prove expansion estimates in terms of two solutions $f_1,f_2$ of order less than three. Note that these expansions have a very specific form, and we collect all possible information about the prefactors within the expansion estimate. This will be crucial to deduce the boundary Harnack principle of optimal order in the sequel. A crucial observation is that the expansion does not involve the function $\varphi_0$.

\begin{lemma}\label{lem.exp.bdry.diff}
   Let $k\in\{0,1\}$ and  $\eps\in(0,1)\setminus\{\frac12\} $ with $k+\eps<\frac32$. Let $A\in C^{k+\eps}(H_1)$, $B\in C^{k+\eps-1}(H_1)$, $F_i\in C^{k+\eps-\frac12}(H_1)$ for each $i\in\{1,2\}$. Let $f_i$ be weak solutions to 
    \begin{equation*}
\left\{
\begin{alignedat}{3}
\partial_tf_i+v\cdot\nabla_{x}f_i-\ddiv(A\nabla_vf_i)&=B\cdot\nabla_vf_i+F_i&&\qquad \mbox{in  ${H}_1$}, \\
f_i&=0&&\qquad  \mbox{in $\gamma_-\cap H_1$},
\end{alignedat} \right.
\end{equation*} 
where $(A)_{n,n}(0)=1$. Suppose that $f_2\geq0$ and that there is $c_2\geq1$ such that 
\begin{equation}\label{ass0.const.c1}
    \frac{\varphi_0}{c_2}\leq f_2\quad\text{in }H_1\setminus \mathcal{R}_-.
\end{equation}
Then there exist $b_0,b_1\in\bbR$ and $p,q\in \cP_{1+k}$ with $q(0)=\nabla_vq(0)=0$, and 
\begin{align}\label{ass.sol.dp0}
    (A(0))_n\cdot\nabla_vp(0)=0,
\end{align} such that for any $z\in H_r$ with $r\leq\frac12$,
\begin{align}\label{goal.blowup lemma}
    \| f_1-pf_2-q\varphi_1-b_0{\varphi_2}-b_1{\varphi_3} \|_{L^\infty(H_r)}\leq cr^{\frac32+k+\eps}\left(\sum_{i=1}^{2}\|f_i\|_{L^\infty(H_1)}+[F_i]_{C^{k+\eps-\frac12}(H_1)}\right),
\end{align}
where $c=c(n,\Lambda,\eps,\|A\|_{C^{k+\eps}(H_1)},\|B\|_{C^{k+\eps-1}(H_1)},c_2)$. Moreover, if $F_i\equiv0$, then it holds $b_0=0$.
\end{lemma}

\begin{proof}
In order to establish \eqref{goal.blowup lemma}, we will prove that there exists $(a,p_0,p_1,p_2,p_3)\in\bbR\times \cP_{1+k}\times \cP_{1+k}\times \cP_{[k+\eps-\frac12]}\times \cP_{[k+\eps-1]}$ such that 
$p_0(0)=p_1(0)=\nabla_vp_1(0)=0$, $(A(0))_n\cdot \nabla_vp_0(0)=0$ and
\begin{equation}\label{ineq1.blowup}
\begin{aligned}
    \left\|f_1-af_2-\sum_{i=0}^3p_i{\varphi_i}\right\|_{L^\infty(H_r)}\leq cr^{\frac32+k+\eps}\sum_{i=1}^{2}\left(\|f_i\|_{L^\infty(H_1)}+[F_i]_{C^{k+\eps-\frac12}(H_1)}\right)
\end{aligned}
\end{equation}
for some constant $c=c(n,\Lambda,\eps,\|A\|_{C^{k+\eps}(H_1)},\|B\|_{C^{k+\eps-1}(H_1)},c_2)$. 

First, let us prove that \eqref{ineq1.blowup} implies the desired result \eqref{goal.blowup lemma}. To this end, note that by \autoref{lem.hol12},
\begin{align}\label{ineq0.blowup}
    f_2 =a_0{\varphi_0}+(\mathbf{a}_1\cdot v){\varphi_0}+(\mathbf{a}_2\cdot v){\varphi_1}+g,
\end{align}
where $a_0\in\bbR$, $\mathbf{a}_1,\mathbf{a}_2\in\bbR^n$ and 
\begin{align*}
    |g(z)|&\leq c|z|^{\frac12+k+\eps}\left(\sum_{i=1}^{2}\|f_i\|_{L^\infty(H_1)}+[F_i]_{C^{k+\eps-\frac32}(H_1)}\right)
\end{align*}
for some constant $c=c(n,\Lambda,\eps,\|A\|_{C^{k+\eps}(H_1)},\|B\|_{C^{k+\eps-1}(H_1)},c_2)$ with $k+\eps<\frac32$. Moreover, we have $\mathbf{a}_1=\mathbf{a}_2=0$ when $k=0$, and $a_0>0$ by \eqref{ass0.const.c1}. 

Next, recalling that $a,p_0,p_1,p_2,p_3$ are determined in \eqref{ineq1.blowup}, we define polynomials $p$ and $q$ by
\begin{align*}
    &p\coloneqq a+\frac{(\nabla_vp_0(0)\cdot v)}{a_0}+\frac{v^{\mathrm{t}}Qv}{a_0}+\frac{\partial_tp_0}{a_0}t,\\
    &q\coloneqq p_1-(\nabla_vp_0(0)\cdot v)\frac{(\mathbf{a}_2\cdot v)}{a_0},\quad b_0\coloneqq p_2,\quad b_1\coloneqq p_3,\\
    &\text{where} \quad v^{\mathrm{t}}Qv=v^{\mathrm{t}}\nabla_v^2p_0v-(\nabla_vp_0(0)\cdot v)\frac{(\mathbf{a}_1\cdot v)}{a_0}.
\end{align*}
Then, using \eqref{ineq0.blowup} we observe
\begin{align}
\label{f1f2}
\begin{split}
   f_1-pf_2 &= f_1 -\big(p(0) + q_0\big) \big(a_0{\varphi_0} + (\mathbf{a}_1\cdot v){\varphi_0} + (\mathbf{a}_2\cdot v){\varphi_1} + g \big) \\
&= f_1 -p(0)f_2 - a_0q_0{\varphi_0} + q_0 \big((\mathbf{a}_1\cdot v){\varphi_0} + (\mathbf{a}_2\cdot v){\varphi_1}\big) + q_0 g\\
&= f_1-p(0)f_2-a_0q_0{\varphi_0}-q_1((\mathbf{a}_1\cdot v){\varphi_0}+(\mathbf{a}_2\cdot v){\varphi_1})-g_1.
\end{split}
\end{align}
where
\begin{align*}
    q_0(z) &\coloneqq p(z)-p(0), \qquad q_1(z)\coloneqq \nabla_vp(0)\cdot v, \\ g_1(z)&\coloneqq (q_0(z)-q_1(z))((\mathbf{a}_1\cdot v){\varphi_0}+(\mathbf{a}_2\cdot v){\varphi_1})+q_0(z)g(z).
\end{align*}
Thus we have
\begin{align}\label{ineq.g1}
    |g_1(z)|&\leq c|z|^{\frac32+k+\eps}\left(\sum_{i=1}^{2}\|f_i\|_{L^\infty(H_1)}+[F_i]_{C^{k+\eps-\frac12}(H_1)}\right),
\end{align}
i.e. $g_1$ satisfies an error estimate of one order higher than $g$. To prove it, we have also used \cite[Lemma 5.3]{KiWe26}, which is closely related to \autoref{lem.rep.bdry} and yields after H\"older interpolation (see \autoref{lem.interpolhol})
\begin{align*}
    [F_i]_{C^{k+\eps-\frac32}(H_1)} \le c \Vert F_i \Vert_{L^{\infty}(H_1)} + c [F_i]_{C^{k+\eps-\frac12}(H_1)} \le  c \Vert f_i \Vert_{L^{\infty}(H_1)} + c [F_i]_{C^{k+\eps-\frac12}(H_1)}.
\end{align*}
Now, note that by construction,
\begin{align*}
    p(0) = a, \quad \nabla_v p(0) = \frac{\nabla_v p_0(0)}{a_0}, \quad q_0(z) = \frac{\nabla_vp_0(0)\cdot v}{a_0}+\frac{v^{\mathrm{t}}Qv}{a_0}+\frac{\partial_tp_0}{a_0}t, \quad q_1(z) = \frac{\nabla_v p_0(0) \cdot v}{a_0},
\end{align*}
and thus 
\begin{align*}
    q_1 (\mathbf{a}_2 \cdot v) + q &= p_1, \\
    a_0 q_0 + q_1 (\mathbf{a}_1 \cdot v) &= \nabla_vp_0(0)\cdot v + v^{\mathrm{t}}Qv \\
    &\quad + \partial_tp_0 t + (\nabla_v p_0(0) \cdot v) \frac{\mathbf{a}_1 \cdot v}{a_0} = \nabla_vp_0(0)\cdot v + v^{\mathrm{t}}\nabla_v^2p_0v + \partial_tp_0 t = p_0,
\end{align*}
where we used in the last step that $p_0 \in \cP_{1+k}$ is equal to its Taylor expansion around zero and satisfies $p_0(0) = 0$.
Hence, by \eqref{f1f2}, we deduce
\begin{align}
\label{eq:f1f2-rewrite}
\begin{split}
    f_1 & -pf_2-q{\varphi_1}-b_0{\varphi_2}-b_1{\varphi_3} \\
    &= f_1 - a f_2 - (a_0 q_0 + q_1 (\mathbf{a}_1 \cdot v) )\varphi_0 - (q_1 (\mathbf{a}_2 \cdot v) + q) \varphi_1 - b_0 \varphi_2 - b_1 \varphi_3 - g_1 \\
    &=f_1-af_2-\sum_{i=0}^3p_i{\varphi_i} -g_1.
\end{split}
\end{align}
Thus, combining \eqref{ineq1.blowup} and \eqref{ineq.g1} leads to the desired estimate \eqref{goal.blowup lemma}.

Hence, it remains to show \eqref{ineq1.blowup}. We are going to prove it by a contradiction-compactness argument. The proof somewhat follows the scheme of \cite[Proposition 5.1]{KiWe26}, but is more complicated since the blow-up sequence involves two solutions at the same time. We will divide the proof into three steps.

\textbf{Step 1. Constructing the blow-up sequence.} We assume that there are sequences $(A_{l})_l$, $(B_{l})_l$, $(F_{i,l})_l$, and $(f_{i,l})_l$ such that
 \begin{equation*}
\left\{
\begin{alignedat}{3}
\partial_tf_{i,l}+v\cdot\nabla_{x}f_{i,l}-\ddiv(A_{l}\nabla_vf_{i,l})&=B_{l}\cdot\nabla_vf_{i,l}+F_{i,l}&&\qquad \mbox{in  ${H}_1$}, \\
f_{i,l}&=0&&\qquad  \mbox{in $\gamma_-\cap H_1$}
\end{alignedat} \right.
\end{equation*} 
with 
\begin{equation}\label{ass.blow}
\begin{aligned}
    \|A_{l}\|_{C^{k+\eps}(H_1)}+\|B_{l}\|_{C^{k+\eps-1}(H_1)}&\leq \Lambda,\\
    \|f_{i,l}\|_{L^\infty(H_1)}+[F_{i,l}]_{C^{k+\eps-\frac12}(H_1)}&\leq 1,\\
    \frac{\varphi_0}{c_2}&\leq f_{2,l}\quad\text{in }H_1\setminus \mathcal{R}_-
\end{aligned}
\end{equation}
for some $\Lambda\geq1$, $c_2\geq1$, but
\begin{align}\label{ass.cont}
    \sup_{l\in\N}\sup_{(a,p_0,p_1,p_2,p_3)\in\mathcal{T}_l}\sup_{\rho\in[0,\frac12]}\rho^{-(\frac32+k+\eps)}\left\|f_{1,l}-af_{2,l}-\sum_{i=0}^3p_i{\varphi_0}\right\|_{L^\infty(H_\rho)}=\infty,
\end{align}
where the vector spaces $\mathcal{T}_l$ are defined by
\begin{align*}
    \mathcal{T}_l\coloneqq\left\{(a,p_0,p_1,p_2,p_3)\in\bbR\times \cP_{k+1}^2 \times\cP_{[k+\eps-\frac12]}\times \cP_{[k+\eps-1]}\,\bigg|\begin{array}{l} p_0(0)=p_1(0)=Dp_1(0)=0,\\(A_l(0))_n\cdot \nabla_vp_0(0)=0
    \end{array}\right\}.
\end{align*}

Note that $p_2,p_3 \in \R$, and if $k + \eps < \frac{1}{2}$ and $k + \eps < 1$, they are zero, respectively.

For any $\rho>0$ and $l\in\N$, there is a minimizer $(a_{l,\rho},p_{0,l,\rho},p_{1,l,\rho},p_{2,l,\rho},p_{3,l,\rho})\in \mathcal{T}_l$ such that
\begin{equation}\label{l2.proj.blow}
\begin{aligned}
&\left\|f_{1,l}-a_{l,\rho}f_{2,l}-\sum_{i=0}^3p_{i,l,\rho}{\varphi_i}\right\|_{L^2(H_\rho)}\leq \left\|f_{1,l}-af_{2,l}-\sum_{i=0}^3p_{i}{\varphi_i}\right\|_{L^2(H_\rho)},\\
    &\int_{H_\rho}\left(f_{1,l}-a_{l,\rho}f_{2,l}-\sum_{i=0}^3p_{i,l,\rho}{\varphi_i}\right)\left(af_{2,l}+\sum_{i=0}^3p_i{\varphi_i}\right)\,dz=0,
\end{aligned}
\end{equation}
for any $(a,p_0,p_1,p_2,p_3)\in\mathcal{T}_l$.
Now we define
\begin{align}\label{defn.theta.blowup}
    \theta(r)\coloneqq \sup_{l\in\N}\max_{r\leq \rho\leq\frac12}\rho^{-(\frac32+k+\eps)}\left\|f_{1,l}-a_{l,\rho}f_{2,l}-\sum_{i=0}^3p_{i,l,\rho}{\varphi_i}\right\|_{L^\infty(H_\rho)}.
\end{align}
First, we prove $\theta(r)\to\infty$ as $r\to0$. Note from \eqref{defn.theta.blowup} that
\begin{align*}
    \left\|(a_{l,\rho}-a_{l,2\rho})f_{2,l}+\sum_{i=0}^3(p_{i,l,\rho}-p_{i,l,2\rho}){\varphi_i}\right\|_{L^\infty(H_\rho)}\leq c\theta(\rho)\rho^{\frac32+k+\eps}.
\end{align*}
On the other hand, we observe from \autoref{lem.hol12}  that there is a constant $c_{2,m}$ such that $|c_{2,m}|\leq c(n,\Lambda,\eps)$ is uniformly bounded in $m$, and
\begin{align}\label{ineq2.blow}
    \|f_{2,m}-c_{2,m}{\varphi_0}\|_{L^\infty(H_\rho)}\leq c\rho^{\frac12+\eps},
\end{align}
which implies $f_{2,m}\leq L{\varphi_0}$ in $H_{1}\setminus \mathcal{R}_-$ for some constant $L=L(n,\Lambda,\eps)$.
Then, using \eqref{eq.norm2}, we can deduce 
\begin{equation}\label{ineq.const.seq}
\begin{aligned}
&|a_{l,\rho}-a_{l,2\rho}|\leq c\theta(\rho)\rho^{1+k+\eps},\quad
  |p^{\beta}_{i,l,\rho}-p^{\beta}_{i,l,2\rho}|\leq c\theta(\rho)\rho^{1+k+\eps-\beta}\quad\text{for each }i\in\{0,1\},\\
  &|p_{2,l,\rho}-p_{2,l,2\rho}|\leq c\theta(\rho)\rho^{k+\eps-\frac12},\quad |p_{3,l,\rho}-p_{3,l,2\rho}|\leq c\theta(\rho)\rho^{k+\eps-1},
\end{aligned}
\end{equation}
where we write  $p_{i,l,\rho}=\sum\limits_{\beta}p^{\beta}_{i,l,\rho}$ for some monomial $p^{\beta}_{i,l,\rho}$ with kinetic degree $\beta$. 
Therefore, if $\theta(r)\leq N$ for some constant $N$, then as in \cite[Lemma 2.14]{RoWe25}, we have $a_{l,0}\coloneqq \lim\limits_{\rho\to0}a_{l,\rho}$ and ${p}_{i,l,0}\coloneqq \lim\limits_{\rho\to0}p_{i,l,\rho}$ such that
\begin{align*}
    \left\|f_{1,l}-a_{l,0}f_{2,l}-\sum_{i=0}^3p_{i,l,0}{\varphi_i}\right\|_{L^\infty(H_\rho)}\leq c\rho^{\frac32+k+\eps}.
\end{align*}
Since clearly also $(A_l(0))_n\cdot\nabla_vp_{0,l,0}(0)=\lim\limits_{\rho\to0}(A_l(0))_n\cdot\nabla_vp_{0,l,\rho}(0)=0$, this contradicts \eqref{ass.cont}. Thus, we have proved that $\theta(r)\to\infty$, as $r\to0$.

Next, since $\theta(r)$ is non-increasing, there are sequences $l_m$ and $r_m$ such that $r_{l_m}\to0$ as $m\to\infty$ and 
\begin{align}\label{nonzer.tildef}
    \frac{\theta(r_{l_m})}2\leq (r_{l_m})^{-(\frac32+k+\eps)}\left\|f_{1,l_m}-a_{l_m,r_{l_m}}f_{2,l_m}-\sum_{i=0}^3p_{i,l_m,r_{l_m}}{\varphi_i}\right\|_{L^\infty(H_{r_{l_m}})}\leq \theta(r_{l_m}).
\end{align}
Henceforth, we may write $l_m=m$. Now, define the blow-up sequence $\widetilde{f}_m$ by
\begin{equation*}
    \widetilde{f}_m(z)\coloneqq \frac{f_{1,m}(S_{r_m}z)-a_{m,r_{m}}f_{2,m}(S_{r_m}z)-\sum\limits_{i=0}^3p_{i,m,r_{m}}(S_{r_m}z){\varphi_i}(S_{r_m}z)}{r_m^{\frac32+k+\eps}\theta(r_m)}
\end{equation*}
to see that 
\begin{equation*}
\left\{
\begin{alignedat}{3}
\partial_t\widetilde{f}_m+v\cdot\nabla_{x}\widetilde{f}_{m}-\ddiv(\widetilde{A}_{m}\nabla_v\widetilde{f}_{m})&=\widetilde{B}_{m}\cdot\nabla_v\widetilde{f}_{m}+\widetilde{F}_{m}+\widetilde{g}_m&&\qquad \mbox{in  ${H}_{1/r_m}$}, \\
\widetilde{f}_{m}&=0&&\qquad  \mbox{in $\gamma_-\cap H_{1/r_m}$},
\end{alignedat} \right.
\end{equation*} 
where we write
\begin{equation}\label{defn.tildeamg}
\begin{aligned}
    \widetilde{A}_m(z)&\coloneqq A_m(S_{r_m}z),\quad \widetilde{B}_m(z)\coloneqq r_mB_m(S_{r_m}z),\\
     \widetilde{F}_m(z)&\coloneqq \frac{r_m^{\frac12-k-\eps}}{\theta(r_m)}(F_{1,m}(S_{r_m}z)-a_{m,r_m}F_{2,m}(S_{r_m}z)),\\
    \widetilde{g}_m(z)&\coloneqq -\sum_{i=0}^1{\partial_t(\widetilde{p}_{i,m,r_m})}{\widetilde{\varphi_i}}-\sum_{i=0}^3\widetilde{p}_{i,m,r_m}(v\cdot\nabla_x)(\widetilde{\varphi_i})\\
    &\quad+\sum_{i=0}^{3}\left(\ddiv\left(\widetilde{A}_m\nabla_v(\widetilde{p}_{i,m,r_m}\widetilde{\varphi_i}) \right)+\widetilde{B}_m\cdot\nabla_v(\widetilde{p}_{i,m,r_m}\widetilde{\varphi_i})\right),\\
    \widetilde{{\varphi_i}}(x_n,v_n)&\coloneqq {{\varphi_i}(r_m^3x_n,r_mv_n)},\quad \widetilde{p}_{i,m,r_m}(z)\coloneqq \frac{p_{i,m,r_m}(S_{r_m}z)}{r_m^{\frac32+k+\eps}\theta(r_m)}.
\end{aligned}
\end{equation}
In particular, by \eqref{eq.psi0}, \eqref{eq.dphi1}, \eqref{eq.vdphi0} and the fact that 
\begin{align*}
    v\cdot\nabla_x([\nabla_v{\widetilde{p}}_{0,m,r_m}(0)\cdot v]{\widetilde{\varphi}_0})-\ddiv(\widetilde{A}_{m}(0)\nabla_v([\nabla_v{\widetilde{p}}_{0,m,r_m}(0)\cdot v]{\widetilde{\varphi}_0}))=0,
\end{align*}
which follows from $(A_m(0))_n\cdot\nabla_vp_{0,m,r_m}(0)=0$,
we rewrite the function $\widetilde{g}_m$ as 
\begin{equation}\label{rewrite.tildegm}
\begin{aligned}
    \widetilde{g}_m(z)&=-\sum_{i=0}^1{\partial_t(\widetilde{p}_{i,m,r_m})}{\widetilde{\varphi_i}}+\sum_{i=0}^3\left(\ddiv((\widetilde{A}_m-\widetilde{A}_m(0))\nabla_v (\widetilde{p}_{i,m,r_m}\widetilde{\varphi_i})+\widetilde{B}_m\cdot\nabla_v(\widetilde{p}_{i,m,r_m}\widetilde{\varphi_i})\right)\\
    &\quad+\sum_{i=0}^1\left[\ddiv(\widetilde{A}_m(0)\nabla_v(\widetilde{q}_{i,m,r_m})\widetilde{\varphi_i})+\widetilde{A}_m(0)\nabla_v(\widetilde{\varphi_i})\cdot\nabla_v(\widetilde{q}_{i,m,r_m})\right]\\
    &\quad-3r_m^2\widetilde{q}_{1,m,r_m}(\partial_{v_n,v_n}{\varphi_0})(S_{r_m}z)+r_m^2\widetilde{p}_{2,m,r_m}+r_m^{2}m_0\widetilde{p}_{3,m,r_m}\widetilde{\varphi_0},
\end{aligned}
\end{equation}
where we write $\widetilde{q}_{0,m,r_m}\coloneqq\widetilde{p}_{0,m,r_m}-\nabla_v\widetilde{p}_{0,m,r_m}(0)\cdot v$ and $\widetilde{q}_{1,m,r_m}\coloneqq\widetilde{p}_{1,m,r_m}$, which gives $\widetilde{q}_{i,m,r_m}(0)=D\widetilde{q}_{i,m,r_m}(0)=0$ for each $i\in\{0,1\}$.

\textbf{Step 2. Uniform estimates for the blow up sequence.} Now, we are going to prove that
\begin{equation}\label{ineq.goal.blow}
\begin{aligned}
   & \|\widetilde{f}_m\|_{L^\infty(H_{R})}\leq cR^{\frac32+k+\eps}\quad\text{for some constant }c=c(n,\Lambda,\eps,c_2),\\
    &[\widetilde{f}_m]_{C^{\alpha}(H_R)}\leq c\quad\text{for some constant }c=c(n,\Lambda,\eps,R,c_2) \quad \text{and some } \alpha = \alpha(n,\Lambda) > 0,\\
    & [\widetilde{f}_m]_{C^{1,\eps}(H_1(z_0))}\leq c\quad\text{for some constant }c=c(n,\Lambda,\eps,z_0,c_2)\text{ whenever }H_1(z_0)\Subset \{x_n>0\}. 
\end{aligned}
\end{equation}

First, following the same lines as in the proof of \cite[Lemma 2.15]{RoWe25} and using \eqref{ineq.const.seq}, we obtain
the first inequality in \eqref{ineq.goal.blow} and
\begin{align}\label{limit.coeff.blow}
    \frac{|a_{m,r_m}|}{\theta(r_m)}+\sum_{i=0}^3\frac{\|p_{i,m,r_m}\|_{L^\infty(H_1)}}{\theta(r_m)}\to 0\quad\text{as }m\to\infty.
\end{align}
Moreover, in order to establish \eqref{limit.coeff.blow}, we need to estimate the terms $\widetilde{F}_m$ and $\widetilde{g}_m$. 
To this end, we distinguish between two cases.
\begin{itemize}
\item $k=1$. First, using \eqref{homogen} and $A_m\in C^{1,\eps}$, we observe 
\begin{align*}
    \widetilde{g}_m= \sum_{i=0}^{1}F_{m,i,0}{\varphi_i} +\sum_{i=0}^3F_{m,i,1,1}\partial_{v_n}\varphi_i+\sum_{i=0}^{1}F_{m,i,1,2}\partial_{v_n}\varphi_i+\sum_{i=0}^1(F_{m,i,2,1}+F_{m,i,2,2})\partial_{v_n,v_n}\varphi_i+r_m^2\widetilde{p}_{2,m,r_m},
\end{align*}
where $F_{m,1,2,2}=0$. It holds 
\begin{equation}\label{blow.limit.20}
\begin{aligned}
    &\sum_{i=0}^1[F_{m,i,0}]_{C^\eps(H_R)}+\sum_{i=0}^3\|F_{m,i,1,1}\|_{L^{\infty}(H_R)}+\sum_{i=0}^{1}[F_{m,i,1,2}]_{C^{1,\eps}(H_R)}+\sum_{i=0}^1\|F_{m,i,2,1}\|_{L^\infty(H_R)}\\
    &\quad+[F_{m,0,2,2}]_{C^{2,\eps}(H_R)} \leq c\sum_{i=0}^3\frac{\|p_{i,m,r_m}\|_{L^\infty(H_1)}}{\theta(r_m)}
\end{aligned}
\end{equation}
for some constant $c=c(n,R,\Lambda,\eps)$, as well as
\begin{equation}\label{zero2.blowup}
\begin{aligned}
    F_{m,0,1,2}(0)=F_{m,1,1,2}(0)=0,\quad F_{m,0,2,2}(0)=DF_{m,0,2,2}(0)=0.
\end{aligned}
\end{equation}
This can easily be deduce from the following explicit formulas for the functions $F_{m,i,0}$, $F_{m,i,1,1}$, $ F_{m,i,1,2}$, $F_{m,i,2,1}$ and $F_{m,i,2,2}$, namely
\begin{align*}
    F_{m,i,0}&=\frac{-\partial_t\widetilde{p}_{i,m,r_m}+(\widetilde{A}_m-\widetilde{A}_m(0))\nabla_v^2\widetilde{p}_{i,m,r_m}+\nabla_v\widetilde{A}_m\nabla_v\widetilde{p}_{i,m,r_m}+B_m(S_{r_m}z)\cdot \nabla_v\widetilde{p}_{i,m,r_m}}{r_m^{\eps}\theta(r_m)}\\
    &\quad+\frac{A_{m}(0)\nabla_v^2\widetilde{q}_{i,m,r_m}}{r_m^{\eps}\theta(r_m)}+\frac{m_i\widetilde{p}_{3,m,r_m}}{{r_m^{\eps}\theta(r_m)}},\\
    F_{m,i,1,1}&=\frac{(\ddiv\widetilde{A}_m-\ddiv\widetilde{A}_m(0))\widetilde{p}_{i,m,r_m}+(\widetilde{A}_m-\widetilde{A}_m(0)-\nabla_v\widetilde{A}_m(0)\cdot (r_m v))\nabla_v\widetilde{p}_{i,m,r_m}}{{r_m^{1+\eps}\theta(r_m)}},\\
    &\quad+\frac{(B_m(S_{r_m}z)-B_m(0))\widetilde{p}_{i,m,r_m}}{{{r_m^{1+\eps}\theta(r_m)}}}\text{ for }i\in\{0,1\},\\
    F_{m,i,1,1}&=\frac{(\ddiv\widetilde{A}_m)\widetilde{p}_{i,m,r_m}}{{r_m^{\frac32+\eps-\alpha_i}\theta(r_m)}}+\frac{B_m(S_{r_m}z)\widetilde{p}_{i,m,r_m}}{{{r_m^{\frac32+\eps-\alpha_i}\theta(r_m)}}}\text{ for }i\in\{2,3\},\\
     F_{m,i,1,2}&=\frac{(\ddiv\widetilde{A}_m(0)+B_m(0))\widetilde{p}_{i,m,r_m}+\nabla_v\widetilde{A}_{m}(0)\cdot (r_mv)\nabla_v\widetilde{p}_{i,m,r_m}+2\widetilde{A}_m(0)\nabla_v\widetilde{q}_{i,m,r_m}}{r_m^{1+\eps}\theta(r_m)},\\
    F_{m,0,2,1}&=\frac{(\widetilde{A}_m-\widetilde{A}_m(0)-\nabla_v\widetilde{A}_m(0)\cdot (r_m v))\widetilde{p}_{0,m,r_m}}{r_m^{2+\eps}\theta(r_m)},\,
    F_{m,0,2,2}=\frac{\nabla_v\widetilde{A}_m(0)\cdot (r_m v)\widetilde{p}_{0,m,r_m}-3\widetilde{q}_{1,m,r_m}}{r_m^{2+\eps}\theta(r_m)},\\
    F_{m,1,2,1}&=\frac{(\widetilde{A}_m-\widetilde{A}_m(0))\widetilde{p}_{1,m,r_m}}{r_m^{2+\eps}\theta(r_m)},
\end{align*}
where $m_0$ is given in \eqref{rewrite.tildegm}, $m_1\coloneqq0$ and the constants $\alpha_i$ are determined in \eqref{homo.alphai}.
Moreover, by \eqref{defn.tildeamg}, we have 
\begin{align}\label{blow.limit.23}
    [\widetilde{F}_m]_{C^{1+\eps-\frac12}(H_R)}\leq c\frac{(1+|a_{m,r_m}|)}{\theta(r_m)}
\end{align}
for some constant $c=c(n,\Lambda,\eps,R)$. Now employing \autoref{lem.bdry.hol.phi} together with the first inequality in \eqref{ineq.goal.blow}, \eqref{blow.limit.20} and \eqref{blow.limit.23} yields the second inequality in \eqref{ineq.goal.blow}. Moreover, as in the proof of $C^{1,\eps}$ estimates given in \cite[Lemma 2.5]{KiWe26}, we have uniform $C^{1,\eps}$ estimates in \eqref{ineq.goal.blow}. 

    \item $k=0$. In this case, the coefficient $A$ is not differentiable, so we consider divergence type data on the right-hand side. Since the computations are similar to the case $k=1$, we will not give details. By using \eqref{homogen} and the fact that $\widetilde{p}_{0,m,r_m}\in\cP_1$ and $\widetilde{p}_{1,m,r_m}=\widetilde{p}_{3,m,r_m}=0$, we have
\begin{align*}
    \widetilde{g}_m(z)&=-\sum_{i=0}^1\left[\ddiv(G_{m,i,0}{\varphi_{2i}})+\ddiv(G_{m,i,1}\partial_{v_n}{\varphi_{2i}}))+F_{m,i,0}{\varphi_{2i}}+F_{m,i,1}\partial_{v_n}{\varphi_{2i}}\right]+r_m^2\widetilde{p}_{2,m,r_m},
\end{align*}
where 
\begin{equation}\label{zero1.blowup}
    G_{m,0,0}(0)=G_{m,0,1}(0)=G_{m,1,1}(0)=0,\quad G_{m,1,0}=F_{m,1,0}=0
\end{equation}
and for any $R\leq 1/r_m$
\begin{equation}\label{blow.limit11}
\begin{aligned}
    &\|G_{m,0,0}\|_{C^\eps(H_R)}+\sum_{i=0}^1(\|G_{m,i,1}\|_{C^\eps(H_R)}+\|F_{m,i,0}\|_{L^\infty(H_R)})+\|F_{m,0,1}\|_{L^\infty(H_R)}\leq c\frac{\|p_{0,m,r_m}\|_{L^\infty(H_1)}}{\theta(r_m)}
\end{aligned}
\end{equation}
for some constant $c=c(n,\Lambda,\eps,R)$. Recalling \eqref{defn.tildeamg}, we have
\begin{align}\label{blow.limit12}
    [\widetilde{F}_m]_{C^{\eps-\frac12}(H_R)}\leq c\frac{(1+|a_{m,r_m}|)}{\theta(r_m)}
\end{align}
for some constant $c=c(n,R,\Lambda,\eps)$. With these estimates at hand, together with the first inequality given in \eqref{ineq.goal.blow}, we apply \autoref{lem.rep.bdry} to see that
\begin{align*}
    \|\widetilde{F}_m+r_m^2\widetilde{p}_{2,m,r_m}\|_{L^{\infty}(H_R)}\leq c(n,\Lambda,\eps,R).
\end{align*}
Thus, using \autoref{lem.bdry.hol.phi}, the second inequality given in \eqref{ineq.goal.blow} follows. Furthermore, as in the case $k=1$, we obtain the third inequality in \eqref{ineq.goal.blow}.

\end{itemize}
Lastly, using the standard energy estimate and \autoref{lem.rep.bdry} together with \eqref{blow.limit.20}, \eqref{blow.limit.23}, \eqref{blow.limit11} and \eqref{blow.limit12}, we deduce
\begin{align}\label{blow.limit123}
    \int_{H_{R/2}}|\nabla_vf|^2\,dz\leq c(n,\Lambda,R,\eps,k),
\end{align}
which is used for the convergence of the drift term in the next step.

\textbf{Step 3. A limit of the blow up sequence.}
Now we are ready to apply [Lemma D.1, KW 26] together with \eqref{ineq.goal.blow} -\eqref{blow.limit123} to see that $\widetilde{f}_\infty=\lim\limits_{m\to\infty}\widetilde{f}_m$ solves
\begin{equation}\label{eq.limit.blowup}
\left\{
\begin{alignedat}{3}
\partial_t\widetilde{f}_\infty+v\cdot\nabla_{x}\widetilde{f}_{\infty}-\ddiv(\widetilde{A}_{\infty}\nabla_v\widetilde{f}_{\infty})&=m_1+\sum_{i=0}^2p_i\partial^{(i)}_{v_n}{\varphi_0}&&\qquad \mbox{in  $\{x_n>0\}$}, \\
\widetilde{f}_{\infty}&=0&&\qquad  \mbox{in $\{x_n=0\}\times \{v_n>0\}$},\\
\|\widetilde{f}_\infty\|_{L^\infty(H_R)}&\leq cR^{\frac32+k+\eps}&&\qquad  \mbox{for any $R\geq1$},
\end{alignedat} \right.
\end{equation} 
for some constants $m_1,c\in\bbR$ and some polynomials $p_i\in \cP_i$ with $p_1(0)=p_2(0)=Dp_2(0)=0$. In particular, when $k+\eps<\frac12$, then it holds $m_1=0$, and when $k+\eps<1$, then we have $p_0=p_1=p_2=0$.
To prove \eqref{eq.limit.blowup}, we have also used that $\widetilde{A}_\infty=\lim\limits_{m\to\infty}\widetilde{A}_m(0)$, as $\|A_m\|_{C^\eps(H_1)}\leq \Lambda$. For the convergence of the source terms, when $k=1$, we used that by \eqref{blow.limit.20} and \eqref{blow.limit.23}, it holds
\begin{align}\label{formula.limit}
    \widetilde{g}_m+\widetilde{F}_m\to \sum_{i=0}^1F_{\infty,i,0}{\varphi_i}+\sum_{i=0}^1F_{\infty,i,1,2}\partial_{v_n}{\varphi_i}+F_{\infty,0,2,2}\partial_{v_n,v_n}{\varphi_0}+c_\infty
\end{align}
for some constants $F_{\infty,i,0},c_\infty\in \bbR$, some polynomials $F_{\infty,i,1,2}\in \cP_1$, $F_{\infty,0,2,2}\in\cP_2$ for each $i\in\{0,1\}$. Moreover, using \eqref{zero2.blowup}, we derive
\begin{align*}
    F_{\infty,0,1,2}(0)=F_{\infty,1,1,2}(0)=F_{\infty,0,2,2}(0)=DF_{\infty,0,2,2}(0)=D^2F_{\infty,0,2,2}(0)=0.
\end{align*}
Thus, recalling \eqref{notationvarphi} and using that $\partial_{v_n} \varphi_1 = \partial_{v_n} {\varphi_0} + v_n \partial_{v_n,v_n}{\varphi_0}$, we can rewrite \eqref{formula.limit} as the right-hand side given in \eqref{eq.limit.blowup} by choosing $m_1$ and $p_0, p_1, p_2$ appropriately. Similarly, the case $k=0$ can be derived in a similar way using \eqref{blow.limit11}, \eqref{blow.limit12}, and \eqref{zero1.blowup}.

Now, from \eqref{eq.psi0} and \autoref{lem.relphi0123}, we deduce that there are polynomials $q_0,q_1\in \cP_2$ with $q_i(0)=Dq_i(0)=0$ and constants $q_2,q_3\in \bbR$ such that
\begin{align*}
    \partial_t\left(\sum\limits_{i=0}^4q_i{\varphi_i}\right)+v\cdot\nabla_x\left(\sum_{i=0}^4q_i{\varphi_i}\right)-\ddiv\left(\widetilde{A}_\infty\nabla_v\left(\sum_{i=0}^4q_i{\varphi_i}\right)\right)=m_0+\sum_{i=0}^2p_i\partial^{(i)}_{v_n}{\varphi_0}.
\end{align*}
Note that $q_2=0$ when $k+\eps<\frac12$, $q_0=q_1=q_4=0$ when $k+\eps<1$. 

Thus, altogether, we obtain that $\widetilde{f}_\infty-\sum\limits_{i=0}^4q_i{\varphi_i}$ solves \eqref{eq.limit.blowup} with zero right-hand side. Hence, by the Liouville theorem given in \cite[Theorem 4.1]{KiWe26} applied with $p=q=0$, we derive
\begin{align*}
    \widetilde{f}_\infty-\sum\limits_{i=0}^4q_i{\varphi_i}=p{\varphi_0}+p_0{\varphi_3} + \bar{p}_1 + \bar{p}_2 \varphi_2 
\end{align*}
for some $p\in \cP_{[1+k+\eps]}$, $p_0\in\R$, $\bar{p}_1 \in \cP_{[\frac{3}{2} + k+\eps]}$, and $\bar{p}_2 \in \R$.  However, since $\bar{p}_1$ violates the boundary condition and $\bar{p}_2 \varphi_2$ violates the homogeneity of the source term, we conclude that
\begin{align*}
    \widetilde{f}_\infty-\sum\limits_{i=0}^4q_i{\varphi_i}=p{\varphi_0}+p_0{\varphi_3}.
\end{align*}

Now using \autoref{lem.aux.direction}, we have $(\widetilde{A}_\infty)_n\cdot\nabla_vp(0)=0$. In addition, by $\widetilde{A}_m(0)\to \widetilde{A}_\infty$, there is a sequences of vectors $(\mathbf{Q}_m) \subset \bbR^n$ such that $\mathbf{Q}_m\to\nabla_vp(0)\quad\text{and}\quad (\widetilde{A}_m(0))_n\cdot\mathbf{Q}_m=0$. Thus, we have proved that
\begin{align}\label{widetildef.blowup}
    \widetilde{f}_\infty=a_0{\varphi_0}+\sum_{i=0}^{3}q_{\infty,i}{\varphi_i},
\end{align}
where $a_0\in\bbR$, $q_{\infty,0},q_{\infty,1}\in \cP_{k+1}$, $q_{\infty,2}\in \cP_{[k+\eps-\frac12]}$, $q_{\infty,3}\in\cP_{[k+\eps-1]}$ with
\begin{equation}\label{cmlimit.blow}
\begin{aligned}
 &q_{\infty,0}(0)=q_{\infty,1}(0)=Dq_{\infty,1}(0),\\
&\mathbf{Q}_m\to\nabla_vq_{\infty,0}(0)\quad\text{and}\quad (\widetilde{A}_m(0))_n\cdot\mathbf{Q}_m=0.
\end{aligned}
\end{equation}
Recall that by \eqref{ineq2.blow} and the last condition in \eqref{ass.blow}, the limit $c_{2,\infty}\coloneqq \lim\limits_{m\to\infty}c_{2,m}$ exists for a suitable subsequence. Note that for any $z\in H_1$, we have
\begin{equation}\label{ineq3.blow}
\begin{aligned}
    |r_m^{-\frac12}f_{2,m}(S_{r_m}z)-c_{2,\infty}{\varphi_0}(z)|&\leq |r_m^{-\frac12}f_{2,m}(S_{r_m}z)-c_{2,m}{\varphi_0}(z)|+|c_{2,m}-c_{2,\infty}|{\varphi_0}(z)\\
    &\leq r_m^{-\frac12}\|f_{2,m}-c_{2,m}{\varphi_0}\|_{L^\infty(H_{r_m|z|})}+|c_{2,m}-c_{2,\infty}|{\varphi_0}(z),
\end{aligned}
\end{equation}
where we have used the fact that ${\varphi_0}(S_rz)=r^{\frac12}{\varphi_0}(z)$. Therefore using \eqref{ineq2.blow}, \eqref{ineq3.blow} and the fact that $c_{2,m}\to c_{2,\infty}$, we have 
\begin{align}\label{f2m.phi0.blow}
    \lim_{m\to\infty}\|{f_{2,m}(S_{r_m}z)}r_m^{-\frac12}- c_{2,\infty}{\varphi_0}\|_{L^\infty(H_1)}=0.
\end{align}
In particular by the last condition given in \eqref{ass.blow}, we have $c_{2,\infty}>0$.

Therefore, by \eqref{l2.proj.blow}, \eqref{widetildef.blowup}, \eqref{cmlimit.blow}, and \eqref{f2m.phi0.blow}, we have 
\begin{align*}
    0=\lim_{m\to\infty}\int_{H_1}\widetilde{f}_m \left(\frac{a_0}{c_{2,\infty}}r_m^{-\frac12}f_2(S_{r_m}z)+\sum_{i=0}^3q_{m,i}{\varphi_i}(S_{r_m}z)\right)\,dz=\int_{H_1}(\widetilde{f}_\infty)^2,
\end{align*}
where we write 
\begin{align*}
    q_{m,0}(z)&\coloneqq r_m^{-\frac12}(\mathbf{Q}_{m}\cdot v+q_\infty(z)-(\nabla_vq_{\infty,0}(0)\cdot v)),\\
    q_{m,i}(z)&\coloneqq r_m^{-\alpha_i}q_{\infty,i}(z)\quad\text{for each }i\in\{1,2,3\}
\end{align*}
to see that $(\frac{a_0}{c_{2,\infty}}r_m^{-\frac12},q_{m,0},q_{m,1},q_{m,2},q_{m,3})\in \mathcal{T}_m$. Note that the constants $\alpha_i$ are determined in \eqref{homo.alphai}, which ensures $r_m^{-\alpha_i}{\varphi_i}(S_{r_m}z)={\varphi_i}(z)$.

Altogether, it must be $\widetilde{f}_\infty\equiv0$, which contradicts \eqref{nonzer.tildef}. This completes the proof.
\end{proof}

\subsection{Proof of the main result in flat domains}

In this subsection, we will give the proof of the higher order boundary Harnack principle and prove our main result \autoref{thm.mainhar.intro}. A slightly more general version of that result is contained in \autoref{thm.mainhar.gen}.

First, we recall the set $\mathcal{R}^2_-$ defined by 
\begin{align*}
    \mathcal{R}^2_-\coloneqq \{v_n^3\geq 2x_n\}.
\end{align*}

Now we combine the regularity away from $\gamma_0\cup\gamma_-$ and \autoref{lem.exp.bdry.diff} to obtain the following.
\begin{lemma}\label{lem.bdryhar}
    Let $k\in\{0,1\}$ and $\eps\in(0,1)\setminus\{\frac12\}$ with $k+\eps<\frac32$. Let $A\in C^{k+\eps}(H_2)$, $B\in C^{k+\eps-1}(H_2)$, and $F_i\in C^{k+\eps-\frac12}(H_2)$ for each $i\in\{1,2\}$. Let $f_i$ be weak solutions to 
    \begin{equation*}
\left\{
\begin{alignedat}{3}
\partial_tf_i+v\cdot\nabla_{x}f_i-\ddiv(A\nabla_vf_i)&=B\cdot\nabla_vf_i+F_i&&\qquad \mbox{in  ${H}_2$}, \\
f_i&=0&&\qquad  \mbox{in $\gamma_-\cap H_2$},
\end{alignedat} \right.
\end{equation*} 
where $(A)_{n,n}(0)=1$. Suppose that $f_2\geq0$ and that there is $c_2\geq1$ such that 
\begin{equation}\label{c_2.assf.f2}
    \frac{\varphi_0}{c_2}\leq f_2 \quad\text{in }H_2\setminus \mathcal{R}^{2}_-.
\end{equation}
Then for any $z_0\in H_{1}\setminus \mathcal{R}_-$ and $r\coloneqq\frac{\mathrm{dist}(z_0,\gamma_0)}{16}$, 
\begin{align}\label{first.res.bdryhol} 
    [f_1/f_2]_{C^{\min\{1+\eps,3/2\}}(H_{r}(z_0))}\leq c\sum_{i=1}^{2}\left(\|f_i\|_{L^\infty(H_2)}+[F_i]_{C^{\eps-\frac12}(H_2)}\right)
\end{align}
for some constant $c=c(n,\Lambda,\eps,\|A\|_{C^\eps(H_2)},\|B\|_{L^\infty(H_2)},c_2)$. 

Moreover, when $F_i\equiv0$ and $k+\eps>\frac12$, then 
\begin{align}\label{second.res.bdryhol} 
    [f_1/f_2]_{C^{1,\max\{\eps,k\}}(H_{r}(z_0))}\leq c\sum_{i=1}^{2}\|f_i\|_{L^\infty(H_2)},
\end{align}
where $c=c(n,\Lambda,\eps,\|A\|_{C^{k+\eps}(H_2)},\|B\|_{C^{k+\eps-1}(H_2)})$.
\end{lemma}

Note that the condition $k = 1$ is only relevant in case $F_0 \equiv 0$. For the first claim \eqref{first.res.bdryhol}, it suffices to consider $k = 0$ and $\eps \in (0,\frac{1}{2}+\eps')$ for some $\eps' > 0$.

\begin{proof}
Let $z_0\in H_{1}\setminus\mathcal{R}_-$ and set $\rho_0\coloneqq \frac{\mathrm{dist}(z_0,\gamma_0)}{8}$. Then, we have
\begin{align}\label{rel.cylinder.bdryharnack} 
    H_{4\rho_0}(z_0)\cap \gamma_0=\emptyset\quad\text{and}\quad H_{4\rho_0}(z_0)\subset H_{16\rho_0}(\overline{z_0}),
\end{align}
where $\overline{z_0}\coloneqq (t_0,x_0',0,v_0',0)$. Moreover, by \eqref{c_2.assf.f2} and the fact that $H_{\rho_0}(z_0)\subset H_2\setminus \mathcal{R}_-^2$, we have
\begin{align}\label{lowerbdd.f2}
    f_2\geq \frac{\rho_0^{\frac12}}{c}\quad\text{in }H_{\rho_0}(z_0)
\end{align}
for some constant $c=c(c_2)$. Next, we assume 
\begin{align*}
    \sum_{i=1}^{2}\left(\|f_i\|_{L^\infty(H_2)}+[F_i]_{C^{\eps-\frac12}(H_2)}\right)\leq1,
\end{align*}
and we omit the dependency of the constant $c$ whenever it depends on $n,\Lambda,\eps,\|A\|_{C^{k+\eps}(H_1)}$,  $\|B\|_{C^{k+\eps-1}(H_1)}$ and $c_2$. Moreover, by considering $f_0(z)\coloneqq f(\overline{z}_0\circ z)$, we may assume $\overline{z}_0=0$.

By \autoref{lem.exp.bdry.diff},  we have
    \begin{align}\label{ineq1.bdryhar}
        \|f_1-pf_2-q{\varphi_1}-b_0{\varphi_2}-b_1{\varphi_3}\|_{L^\infty(H_r)}\leq cr^{\frac32+k+\eps}
    \end{align}
    for some $p,q\in \cP_{k+1}$ and $b_0,b_1\in\bbR$, 
    where 
    \begin{equation}\label{p.sol.exp}
        (A(0))_n\cdot\nabla_vp(0)=0,
    \end{equation}
 $b_0=0$ when $\eps<\frac12$ and $q=b_1=0$ when $\eps+k<1$. We now split the proof into several steps.

\textbf{Step 1.} We identify the equation that is satisfied by $\widetilde{f}\coloneqq f_1-pf_2-q{\varphi_1}-b_0{\varphi_2}-b_1{\varphi_3}$. 

To this end, we observe from \autoref{lem.hol12} that 
\begin{align*}
    f_2=c_0{\varphi_0}+(\mathbf{C}_0\cdot v){\varphi_0}+(\mathbf{C}_1\cdot v){\varphi_1}+g
\end{align*}
for some $c_0\in\bbR$, $\mathbf{C}_0,\mathbf{C}_1\in\bbR^n$, where $\mathbf{C}_0=\mathbf{C}_1=0$ when $k=0$, and such that
\begin{align*}
    |c_0|+|\mathbf{C}_0|+|\mathbf{C}_1|\leq c, \quad |g(z)|\leq c|z|^{\frac12+k+\eps}.
\end{align*}
Thus, we have 
\begin{align*}
    \widetilde{f}=f_1-p(0)f_2-(\nabla_vp\cdot v)c_0{\varphi_0}-(p(z)-p(0))g(z)-\sum\limits_{i=0}^3q_i{\varphi_i}
\end{align*}
for some polynomials $q_i\in\cP_3$ with $q_0(0)=q_1(0)=Dq_0(0)=Dq_1(0)=0$ and 
\begin{align*}
    \sum_{i=0}^3\|q_i\|_{L^\infty(H_1)}\leq c.
\end{align*}
The proof of this fact follows in the same way as the proof of \eqref{eq:f1f2-rewrite}.

Next, we observe that for $\widetilde{p}(z)\coloneqq p(z)-p(0)$, $\widetilde{p}_{0}(z)\coloneqq c_0+\mathbf{C}_0\cdot v$ and $\widetilde{p}_{1}(z)\coloneqq \mathbf{C}_0\cdot v$,
     \begin{equation}\label{ineq.widetildepg}
     \begin{aligned}
         &\partial_t(\widetilde{p}g)+v\cdot\nabla_x(\widetilde{p}g)-\ddiv(A\nabla_v(\widetilde{p}g))-B\cdot\nabla_v(\widetilde{p}g)\\
         &=\partial_t\widetilde{p}g+\widetilde{p}(\partial_tg+v\cdot\nabla_xg-\ddiv(A\nabla_vg)-B\cdot\nabla_vg)-2A\nabla_vg\nabla_v\widetilde{p}-\ddiv(A\nabla_v\widetilde{p})g-(B\cdot\nabla_v\widetilde{p})g\\
         &=\partial_t\widetilde{p}g+\widetilde{p}\left(F_2-\sum_{i=0}^1\ddiv((A-A(0))\nabla_v(\widetilde{p}_i{\varphi_i}))-B\cdot\nabla_v(\widetilde{p}_i{\varphi_i})+m_0\partial_{v_n}{\varphi_0}+\widetilde{p}_2\partial_{v_n,v_n}{\varphi_0}\right)\\
         &\quad-2A\nabla_vg\nabla_v\widetilde{p}-\ddiv(A\nabla_v\widetilde{p})g-(B\cdot\nabla_v\widetilde{p})g
     \end{aligned}
     \end{equation}
     for some $\widetilde{p}_2\in\cP_1$, $m_0\in\bbR$ with $\widetilde{p}_2(0)=0$, where we have used \eqref{eq.vdphi0}.
     We further estimate the right hand side as 
     \begin{align*}
         &\partial_t\widetilde{p}g+\widetilde{p}F_2+\sum_{i=0}^1\left[(\nabla_v\widetilde{p})^{\mathrm{t}}(A-A(0))\nabla_v(\widetilde{p}_i{\varphi_i})-\ddiv(\widetilde{p}(A-A(0))\nabla_v(\widetilde{p}_i{\varphi_i}))\right]\\
         &\quad-\widetilde{p}\left(\sum_{i=0}^1B\cdot\nabla_v(\widetilde{p}_i{\varphi_i}) - m_0\partial_{v_n}{\varphi_0} - \widetilde{p}_2\partial_{v_n,v_n}{\varphi_0}\right)-A\nabla_vg\nabla_v\widetilde{p}-\ddiv(A\nabla_v\widetilde{p}g )-(B\cdot\nabla_v\widetilde{p})g\\
         &\eqqcolon \widetilde{F}_g-\ddiv(\widetilde{G}_g),
     \end{align*}
     where we have used $\ddiv(Fg)=F\cdot\nabla_vg+\ddiv(F)g$ and we write
     \begin{align*}
         \widetilde{F}_g&\coloneqq \partial_t\widetilde{p}g+\widetilde{p}F_2+\sum_{i=0}^1(\nabla_v\widetilde{p})^{\mathrm{t}}(A-A(0))\nabla_v(\widetilde{p}_i{\varphi_i})-\sum_{i=0}^1\widetilde{p}B\cdot\nabla_v(\widetilde{p}_i{\varphi_i})+\widetilde{p}(m_0\partial_{v_n}{\varphi_0}+\widetilde{p}_2\partial_{v_n,v_n}{\varphi_0})\\
         &\quad-A\nabla_vg\nabla_v\widetilde{p}-(B\cdot\nabla_v\widetilde{p})g,\\
         \widetilde{G}_g&\coloneqq\sum_{i=0}^1\widetilde{p}(A-A(0))\nabla_v(\widetilde{p}_i{\varphi_i})+A\nabla_v\widetilde{p}g.
     \end{align*}
    Now using these definitions, as well as \eqref{eq.psi0} and \eqref{p.sol.exp}, we get 
\begin{align}\label{eq.widetildef}
\partial_t\widetilde{f}+v\cdot\nabla_x\widetilde{f}-\ddiv(A\nabla_v\widetilde{f})=B\cdot\nabla_v\widetilde{f}+\widetilde{F}-\ddiv(\widetilde{G})\quad\text{in }H_{2\rho_0}(z_0),
\end{align}
where we write 
\begin{align*}
    \widetilde{F}\coloneqq F_1-p(0)F_2-q_2-\widetilde{F}_g-h,\quad \widetilde{G}\coloneqq -\widetilde{G}_g-(A-A(0))\nabla_v((\nabla_vp\cdot v)c_0{\varphi_0}+q_2{\varphi_2})
\end{align*}
with $h$ given by
\begin{align*}
    h\coloneqq\partial_t\left(\sum_{i\neq2}q_i{\varphi_i}\right)+v\cdot\nabla_x\left(\sum_{i\neq2}q_i{\varphi_i}\right)-\ddiv\left(A\nabla_v\left(\sum_{i\neq2}q_i{\varphi_i}\right)\right)-B\cdot\nabla_v\left(\sum_{i\neq2}q_i{\varphi_i}\right).
\end{align*}

Note that this yields
\begin{equation}\label{est.h}
    [h]_{C^{\eps}(H_{4\rho_0}(z_0))}\leq [h]_{C^{\eps}(H_{16\rho_0})}\leq c\rho_0^{\frac12-\eps}.
\end{equation}
Indeed, the functions $q_0,q_1,q_3$ are nonzero and $A\in C^{1,\eps}$, when $k=1$. Therefore, using the fact that $q_0{\varphi_0},q_1{\varphi_1}$, and ${\varphi_2}$ have degree $\frac52$, and \eqref{rel.cylinder.bdryharnack} with $\overline{z}_0=0$, we derive \eqref{est.h}.

\textbf{Step 2. }Now, we estimate the $L^\infty$ norms and H\"older semi-norms of $\widetilde{F}_g$ and $\widetilde{G}_g$. First, we will show 
\begin{equation}\label{step2.goal.bdryhar}
\begin{aligned}
    &\rho_0^{2}\|\widetilde{F}_g\|_{L^{\infty}(H_{\rho_0}(z_0))}+\rho_0^{1+\eps}[\widetilde{G}]_{C^{\eps}(H_{\rho_0}(z_0))}\leq c\rho_0^{\frac32+\eps}\quad\text{if }k=0,\\
    &\rho_0^{2+\eps}[\widetilde{F}_g]_{C^\eps(H_{\rho_0}(z_0))}+\rho^{2+\eps}[\widetilde{G}]_{C^{1,\eps}(H_{\rho_0}(z_0))}\leq c\rho_0^{\frac52} \quad\text{if }k=1.
\end{aligned}
\end{equation}
To do so, we need the regularity of $g$ and $\nabla_vg$, so we investigate the following equation
\begin{align*}
         \partial_tg+v\cdot\nabla_xg-\ddiv(A\nabla_vg)&=B\cdot\nabla_vg+F_2-\sum_{i=0}^1\left[\ddiv((A-A(0))\nabla_v(\widetilde{p}_i{\varphi_i}))-B\cdot\nabla_v(\widetilde{p}_i{\varphi_i})\right]\\
         &\quad+m_0\partial_{v_n}{\varphi_0}+\widetilde{p}_{2}\partial_{v_n,v_n}{\varphi_0},
     \end{align*}
     which follows from the third line of \eqref{ineq.widetildepg}. By \autoref{lem.schcomb}, we have
     \begin{equation}\label{gradg.est}
     \begin{aligned}
         &\|\nabla_vg\|_{L^\infty(H_{\rho_0}(z_0))}+\rho_0^{\eps}[\nabla_vg]_{C^{\eps}(H_{\rho_0}(z_0))}\\
         &\leq \frac{c}{\rho_0}\left(\|g\|_{L^\infty(H_{2\rho_0}(z_0))}+\rho_0^{1+\eps}\sum_{i=0}^1[(A-A(0))\nabla_v(\widetilde{p}_i{\varphi_i})]_{C^\eps(H_{2\rho_0}(z_0))}\right)+\rho_0(\|F_2\|_{L^\infty(H_{2\rho_0}(z_0))}+\rho_0^{-\frac12})\\
         &\leq c\rho_0^{\min\{-\frac12+k+\eps,\frac12\}},
     \end{aligned}
     \end{equation}
     where we have used the fact that $H_{2\rho_0}(z_0)\subset H_{16\rho_0}$ by \eqref{rel.cylinder.bdryharnack} with $\overline{z}_0=0$, \eqref{psi0.dd} and 
     \begin{equation}\label{est.ggraphi_0A}
     \begin{aligned}
         &\|g\|_{L^\infty(H_{2\rho_0}(z_0))}\leq \|g\|_{L^\infty(H_{16\rho_0})}\leq c\rho_0^{k+\frac12+\eps},\\
         &\|A-A(0)\|_{L^\infty(H_{2\rho_0}(z_0))}\leq \|A-A(0)\|_{L^\infty(H_{16\rho_0})}\leq c\rho_0^{\min\{k+\eps,1\}}.
     \end{aligned}
     \end{equation}
     Similarly, we have 
     \begin{align}\label{cepsg.est}
         [g]_{C^\eps(H_{\rho_0}(z_0))}\leq c\rho_0^{\frac12}.
     \end{align}
Now using \eqref{gradg.est}, \eqref{est.ggraphi_0A}, and \eqref{cepsg.est} together with the fact that 
\begin{equation*}
    \|\widetilde{p}\|_{C^\delta(H_{2\rho_0}(z_0))}= \|{p}(z)-p(0)\|_{C^\delta(H_{2\rho_0}(z_0))}\leq \|{p}(z)-p(0)\|_{C^\delta(H_{16\rho_0})}\leq c\rho_0^{1-\delta}\quad\text{for any }\delta\in[0,1],
\end{equation*}
we derive \eqref{step2.goal.bdryhar}. In particular, for $C^{1,\eps}$ estimates of the function $\widetilde{G}$, we have used \autoref{lem.holspace} and the product rule for the H\"older space.

\textbf{Step 3.} Now we are ready to apply \autoref{lem.schcomb} to the solution of \eqref{eq.widetildef}. To this end, we divide the proof into four cases according to the range of $\eps$ and the value of $k$.
\begin{itemize}
    \item Let $\eps\in(0,\frac12)$ and $k=0$. Then we have $q=b_0=b_1=0$. Thus, by \autoref{lem.schcomb} together with \eqref{ineq1.bdryhar} and \eqref{step2.goal.bdryhar}, we have 
    \begin{align*}
        [\widetilde{f}]_{C^{1+\eps}(H_{\rho_0/2}(z_0))}\leq \frac{c}{\rho_0^{1+\eps}}\left(\|\widetilde{f}\|_{L^\infty(H_{\rho_0}(z_0))}+\rho_0^2\|\widetilde{F}\|_{L^\infty(H_{\rho_0(z_0))}}+\rho_0^{1+\eps}[\widetilde{G}]_{C^{\eps}(H_{\rho_0}(z_0))}\right)\leq c\rho_0^{\frac12}.
    \end{align*}
    Moreover, by \eqref{rel.cylinder.bdryharnack} with $\overline{z}_0=0$ and \eqref{ineq1.bdryhar}, we have $\|\widetilde{f}\|_{L^\infty(H_{\rho_0}(z_0))}\leq \rho_0^{\frac32+\eps}$. Thus, by \autoref{lem.interpolhol} we get for any $l\in\{0,1\}$ and $\delta\in(0,1)$ with $l+\delta\leq 1+\eps$,
    \begin{align}\label{interpo.1}
        [\widetilde{f}]_{C^{l+\delta}(H_{\rho_0/2}(z_0))}\leq c\rho_0^{\frac32+\eps-(l+\delta)}.
    \end{align}
\item Let $\eps\in(\frac12,1)$ and $k=0$. Similarly, we use \autoref{lem.schcomb} along with $\delta=\frac32$ and \eqref{ineq1.bdryhar}, \eqref{step2.goal.bdryhar}, to see that
 \begin{align*}
        [\widetilde{f}]_{C^{\frac32}(H_{\rho_0/2}(z_0))}&\leq \frac{c}{\rho_0^{1+\frac12}}\left(\|\widetilde{f}\|_{L^\infty(H_{\rho_0}(z_0))}+\rho_0^2\|\widetilde{F}\|_{L^\infty(H_{\rho_0(z_0))}}+\rho_0^{1+\frac12}[\widetilde{G}]_{C^{\frac12}(H_{\rho_0}(z_0))}\right)\\
        &\leq \frac{c}{\rho_0^{1+\frac12}}\left(\rho_0^{\frac32+\eps}+\rho_0^{2}+\rho_0^{1+\eps}[\widetilde{G}]_{C^{\eps}(H_{\rho_0}(z_0))}\right),
    \end{align*}
    where we have used the fact that 
    \begin{align}\label{ineq.gra.psi}
    [(A-A(0))\nabla_v{\varphi_2}]_{C^{\eps}(H_{\rho_0}(z_0))}\leq c\rho_0,\quad
        [\widetilde{G}]_{C^{\frac12}(H_{\rho_0}(z_0))}\leq c\rho_0^{\eps-\frac12}[\widetilde{G}]_{C^{\eps}(H_{\rho_0}(z_0))}.
    \end{align}
    In particular, for the first inequality given in \eqref{ineq.gra.psi}, we have used \eqref{psi0.dd}.
    Therefore, we have $[\widetilde{f}]_{C^{\frac32}(H_{\rho_0}(z_0))}\leq c\rho_0^{\frac12}$. As in the previous case, we also derive \eqref{interpo.1} with $\eps=\frac12$ and $l+\delta\leq \frac32$.
    \item Let $\eps\in(\frac12,1)$ and assume $F_i\equiv0$. In that case we have $q_2=0$. Hence, as in the proof of \eqref{interpo.1}, we also deduce
    \begin{align*}
        [\widetilde{f}]_{C^{1+\eps}(H_{\rho_0/2}(z_0))}\leq \frac{c}{\rho_0^{1+\eps}}\left(\|\widetilde{f}\|_{L^\infty(H_{\rho_0}(z_0))}+\rho_0^2\|\widetilde{F}\|_{L^\infty(H_{\rho_0(z_0))}}+\rho_0^{1+\eps}[\widetilde{G}]_{C^{\eps}(H_{\rho_0}(z_0))}\right)\leq c\rho_0^{\frac12}
    \end{align*}
    and, whenever $l+\delta<2$,
    \begin{align}\label{ineq.inter2}
        [\widetilde{f}]_{C^{l+\delta}(H_{\rho_0/2}(z_0))}\leq c\rho_0^{\frac32+\eps-(l+\delta)}.
    \end{align}

    \item Let $k=1$ and $\eps\in(0,\frac12)$. Since $q_2=0$, we deduce from \autoref{lem.schcomb}, \eqref{ineq1.bdryhar}, and \eqref{step2.goal.bdryhar} that
    \begin{align*}
        [\widetilde{f}]_{C^{2+\eps}(H_{\rho_0/2}(z_0))}\leq \frac{c}{\rho_0^{2+\eps}}(\|\widetilde{f}\|_{L^\infty(H_{\rho_0}(z_0))}+\rho_0^{2+\eps}[\widetilde{F}]_{C^\eps(H_{\rho_0}(z_0))}+\rho_0^{2+\eps}[\widetilde{G}]_{C^{1+\eps}(H_{\rho_0}(z_0))})\leq c\rho_0^{\frac12-\eps}.
    \end{align*}
    Moreover, as $\|\widetilde{f}\|_{L^\infty(H_{\rho_0/2}(z_0))}\leq c\rho_0^{\frac52+\eps}$ by \eqref{ineq1.bdryhar}, using \autoref{lem.interpolhol}, we deduce that  for any $l+\delta\leq 2+\eps$,
    \begin{align}\label{interpol2}
        [\widetilde{f}]_{C^{l+\delta}(H_{\rho_0}(z_0))}\leq c\rho_0^{\frac52-(l+\delta)}.
    \end{align}
    \end{itemize}
  \textbf{Step 4.} We are ready to prove the main estimates. 
First, note that 
\begin{align}\label{ineq6.bdryhol}
    [f_2]_{C^{l,\delta}(H_{\rho_0}(z_0))}\leq c\rho_0^{-(l+\delta)}
\end{align}
  for any $l\in\N\cup\{0\}$ and $\delta\in[0,1]$ with $l+\delta\leq 1+\eps$ when $k=0$, and $l+\delta<2+\eps$ when $k=1$ by \autoref{lem.schcomb}.
  
  Thus, using this together with \autoref{lem.holspace}, the product rule and \eqref{c_2.assf.f2}, we have 
\begin{align}\label{ineq7.bdryhol}
    [1/f_2]_{C^{l,\delta}(H_{\rho_0}(z_0))}\leq c\rho_0^{-(l+\delta)-\frac12},
\end{align}
by \autoref{lem.hol12} and due to the lower bound \eqref{lowerbdd.f2}. Now we write $r\coloneqq \frac{\mathrm{dist}(z_0,\gamma_0)}{16} =\frac{\rho_0}{2}$.

For $k=0$, we observe 
    \begin{align*}
        [f_1/f_2]_{C^{1,\min\{\eps,\frac32\}}(H_{r}(z_0))}=[(\widetilde{f}+q{\varphi_1}+b_0{\varphi_2}+b_1{\varphi_3})/f_2]_{C^{1,\min\{\eps,\frac32\}}(H_{r}(z_0))}.
    \end{align*}
    Using \autoref{lem.holspace}, \autoref{lem.grabdd.hol}, the product rule, \eqref{interpo.1}, \eqref{ineq6.bdryhol}, and \eqref{ineq7.bdryhol}, we get
    \begin{equation}\label{ineq8.bdryhol}
    \begin{aligned}
        [\widetilde{f}/f_2]_{C^{1,\min\{\eps,\frac32\}}(H_{r}(z_0))}&\leq c [\nabla_v(\widetilde{f}/f_2)]_{C^{\min\{\eps,\frac32\}}(H_{r}(z_0))}+c [\widetilde{f}/f_2]_{C_{\mathcal{T}}^{{1+\min\{\eps,\frac32\}}}(H_{r}(z_0))}\\
        &\quad+c[\widetilde{f}/f_2]_{C_{x}^{{1+\min\{\eps,\frac32\}}}(H_{r}(z_0))}
        \leq c,
    \end{aligned}
    \end{equation}
    where we have used $[f]_{{C_{\mathcal{T}}^{{1+\min\{\eps,\frac32\}}}(H_{r}(z_0))}} + [f]_{C_{x}^{{1+\min\{\eps,\frac32\}}}(H_{r}(z_0))}\leq c[f]_{C^{1+\min\{\eps,\frac32\}}(H_{r}(z_0))}$ by the definition of the kinetic H\"older space given in \eqref{defn.hol.xt}. In addition, by \eqref{property.phi_0}, \eqref{psi0.dd}, and \eqref{psi0.dd}, we have 
    \begin{align*}
        [(q{\varphi_1}+b_0{\varphi_2}+b_1{\varphi_3})/f_2]_{C^{1,\min\{\eps,\frac32\}}(H_{r}(z_0))}\leq c.
    \end{align*}
    Altogether, this proves \eqref{first.res.bdryhol}. 
    
    For \eqref{second.res.bdryhol}, we assume $F_i\equiv 0$, $\eps\in(\frac12,1)$, and $k=0$. Then we have $q\in\cP_2$ with $q(0)=Dq(0)=0$. Thus we have
    \begin{align*}
        [(q{\varphi_1}+b_1{\varphi_3})/f_2]_{C^{1,\eps}(H_{r}(z_0))}\leq c.
    \end{align*}
    Moreover, as in \eqref{ineq8.bdryhol} together with \eqref{ineq.inter2}, we have $[\widetilde{f}/f_2]_{C^{1,\eps}(H_r(z_0))}\leq c$. Thus, we get \eqref{second.res.bdryhol} when $\eps\in(\frac12,1)$.

    Now we assume $F_i\equiv 0$, $k=0$, and $\eps\in(0,\frac12)$. We will prove $[f_1/f_2]_{C^{2,\eps}(H_r(z_0))}\leq cr^{-\eps}$, which gives $[f_1/f_2]_{C^{1,1}(H_r(z_0))}\leq c$ using \autoref{lem.interpolhol} together with $\|f_1/f_2\|_{L^\infty(H_r(z_0))}\leq cr^{2}$. Note that the $L^\infty$-estimate follows from \eqref{ineq1.bdryhar}, \eqref{c_2.assf.f2}, and the fact that $q{\varphi_1}$ and ${\varphi_3}$ have homogeneity $\frac52$, by \eqref{psi0.dd} and the fact that $q(0)=Dq(0)=0$.

    As in \eqref{ineq8.bdryhol} together with \eqref{interpol2} instead of \eqref{interpo.1}, we derive $[\widetilde{f}/f_2]_{C^{2,\eps}(H_r(z_0))}\leq cr^{-\eps}$. Moreover using the homogeneity of $q{\varphi_1}$ and ${\varphi_3}$, we deduce $[\widetilde{f}/f_2]_{C^{2,\eps}(H_r(z_0))}\leq cr^{-\eps}$. Thus, as in the above argument, we derive the desired estimate \eqref{second.res.bdryhol} when $k=1$, which completes the proof.
\end{proof}

Now we combine \autoref{lem.hop.gen}, \autoref{lem.bdryhar} and the covering argument given in \cite[Lemma 2.9]{RoWe25} to prove the main result of this article in the half-space.

\begin{proof}[Proof of \autoref{thm.mainhar}]
First, we claim that the Hopf lemma \autoref{lem.hop.gen} remains true in larger domains $H_{\frac{\rho}{8}}\setminus\mathcal{R}^M_-$ instead of $H_{\frac{\rho}{8}}\setminus\mathcal{R}_-$, where $\mathcal{R}_-^M\coloneqq \{Mx_n\leq v_n^3\}$ for any $M\geq1$. Indeed, we get \eqref{using.r1-} with $\mathcal{R}_-$ replaced by $\mathcal{R}^M_-$, as $\frac1c|z|^{\frac12}\leq {\varphi_0}(z)\leq c |z|^{\frac12}$ in the complement of $\mathcal{R}^M_-$ for some constant $c=c(M)$, and then we just follow the lines of the remaining proof given in \autoref{lem.hop.gen}.
 Therefore, we have in particular,
    \begin{align*}
        \frac{c_0}{c}\leq \inf_{H_{\frac{\rho}{8}}\setminus\mathcal{R}^2_-}\frac{f}{{\varphi_0}},
    \end{align*}
    where $c=c(n,\Lambda,\eps,\|A\|_{C^{k+\eps}(H_2)},\|B\|_{C^{k+\eps-1}(H_2)})$. By applying a re-scaled version of \autoref{lem.bdryhar}, for any $z_0\in H_{\frac{\rho}8}\setminus \mathcal{R}^2_-$ and $r\coloneqq \frac{\mathrm{dist}(z_0,\gamma_0)}{16}$, we deduce 
\begin{align*}
    [f_1/f_2]_{C^{1,\min\{\eps,\frac12\}}(H_r(z_0))}\leq c\sum_{i=1}^{2}\left(\|f_i\|_{L^\infty(H_2)}+[F_i]_{C^{\eps-\frac12}(H_2)}\right)\leq c
\end{align*}
for some constant $c=c(n,\Lambda,\eps,c_2,\|A\|_{C^\eps(H_2)},\|B\|_{L^\infty(H_2)})$, as $\|f_i\|_{L^\infty(H_2)}+\|F_i\|_{C^{\eps-\frac12}(H_2)}\leq 1$. Next, by the covering argument given in \cite[Lemma 2.9]{RoWe25}, we derive 
\begin{align*}
    [f_1/f_2]_{C^{1,\min\{\eps,\frac12\}}(H_{\frac{\rho}8}\setminus \mathcal{R}_-)}\leq c.
\end{align*}
When $F_i\equiv0$, the arguments are exactly the same so that we also obtain \eqref{ineq.main2.mainhar}. 
\end{proof}

\subsection{Boundary Harnack in non-flat domains}

We conclude this subsection by providing the boundary Harnack principle in general domains. Since the same flattening argument is already explained in the proof of \autoref{lem.hopf.curve}, we do not give all details here. 
\begin{theorem}\label{thm.mainhar.gen}
   Let $\Omega$ be a bounded domain of class $C^{\frac{k+4+\eps}2}$ for some $k\in\{0,1\}$, $\eps\in(0,1)\setminus \{\frac12\}$ with $k+\eps<\frac32$ and $z_0\in \gamma_0$. Let $A\in C^{k+\eps}(\mathbf{H}_2(z_0))$, $B\in C^{k+\eps-1}(\mathbf{H}_2(z_0))$, $F_i\in C^{k+\eps-\frac12}(\mathbf{H}_2(z_0))$ for each $i\in\{1,2\}$, where $\mathbf{H}_2(z_0)\coloneqq H_2(z_0)\cap \{x\in\Omega\}$. Let $f_i$ be weak solutions to 
    \begin{equation*}
\left\{
\begin{alignedat}{3}
\partial_tf_i+v\cdot\nabla_{x}f_i-\ddiv(A\nabla_vf_i)&=B\cdot\nabla_vf_i+F_i&&\qquad \mbox{in  $\mathbf{H}_2(z_0)$}, \\
f_i&=0&&\qquad  \mbox{in $\gamma_-\cap \mathbf{H}_2(z_0)$}.
\end{alignedat} \right.
\end{equation*} 
with $f_2,F_2\geq0$. Suppose $\|f_i\|_{L^\infty(\mathbf{H}_2(z_0))},\|F_i\|_{C^{\eps-\frac12}(\mathbf{H}_2(z_0))}\leq 1$. Then, there exist a small constant $\rho_1=\rho_1(\Lambda,\Omega)$ and a large constant $\mathcal{M}=\mathcal{M}(n,\Lambda,\Omega)$ such that if $f_2(z_1)=c_0>0$ for some $c_0 > 0$ with $z_1\coloneqq z_0\circ (-1,-\rho_1^3n_{x_0},0)$ and 
\begin{align*}
    \|F_2\|_{L^\infty(\mathbf{H}_2(z_0))}\leq \frac{c_0}{\mathcal{M}},
\end{align*}
then
\begin{align*}
    \left[\frac{f_1}{f_2}\right]_{C^{\min\{1+\eps,3/2\}}(\mathbf{H}_{\frac{\rho_1}8}(z_0)\setminus\mathcal{R}_-)}\leq c,
\end{align*}
where $c=c(n,\Lambda,\eps,c_0,\|A\|_{C^{\eps}(\mathbf{H}_2(z_0))},\|B\|_{C^{\eps-1}(\mathbf{H}_2(z_0))},|z_0|,\Omega)$ and $\mathcal{R}_-\coloneqq \{d_\Omega(x)\leq (n_x\cdot v)^3\}$. 

Moreover, when $F_i\equiv 0$, we have 
\begin{align*}
     \left[\frac{f_1}{f_2}\right]_{C^{1,\max\{\eps,k\}}(\mathbf{H}_{\frac{\rho_1}8}(z_0)\setminus\mathcal{R}_-)}\leq c
\end{align*}
for some constant $c=c(n,\Lambda,\eps,c_0,\|A\|_{C^{k+\eps}(\mathbf{H}_2(z_0))},\|B\|_{C^{k+\eps-1}(\mathbf{H}_2(z_0))},|z_0|,\Omega)$.
\end{theorem}
\begin{proof}
   We may assume $z_0=0$ as $z_0\in\gamma_0$. As in the proof of \autoref{lem.hopf.curve}, $\widetilde{f}_i\coloneqq f_i\circ \Psi^{-1}$ is a weak solution to 
    \begin{equation*}
\left\{
\begin{alignedat}{3}
\partial_t\widetilde{f}_i+v\cdot\nabla_{x}\widetilde{f}_i-\ddiv(\widetilde{A}\nabla_v\widetilde{f}_i)&=\widetilde{B}\cdot\nabla_v\widetilde{f}_i+\widetilde{F}_i&&\qquad \mbox{in  $I_2\times \mathrm{H}_{2R_1}$}, \\
\widetilde{f}_i&=0&&\qquad  \mbox{in $\{x_n=0\}\times \{v_n>0\}$},
\end{alignedat} \right.
\end{equation*} 
where $\widetilde{A}$, $\widetilde{B}$ and $\widetilde{F}$ are defined as in \cite[(8.3)]{KiWe26}. Similarly, as in the proof of \autoref{lem.hopf.curve}, we may assume $(\widetilde{A}(0))_{n,n}=1$ and we choose $\rho_1=\rho_1(\Omega)$ such that $\Psi(-1,x_0-\rho_1^3n_{x_0},0)=(-1,0,(\rho R_1)^3,0)\in \bbR\times \bbR^{n-1}\times \bbR\times \bbR^n$, where $\rho$ is the universal constant determined in \autoref{thm.mainhar} to have $f_2(-1,0,(\rho R_1)^3,0)=c_0$. Next, selecting a small constant $\mathcal{M}=\mathcal{M}(n,\Lambda,\Omega)$, we have $\widetilde{f}_2(-(R_1)^3,0,(\rho R_1)^3,0)=\frac{c_0}{c}$ for some constant $c=c(n,\Lambda,\Omega)$. From \cite[Lemma 2.17]{RoWe25}, $\widetilde{A}\in C^{k+\eps}, \widetilde{B}\in C^{k+\eps-1}$, and $\widetilde{F}\in C^{k+\eps-\frac12}$ as $\partial\Omega\in C^{\frac{k+4+\eps}2}$. Thus, by applying a rescaled version of \autoref{thm.mainhar}, we deduce H\"older estimates of $\widetilde{f}_1/\widetilde{f}_2$. Finally, scaling back and using the fact that $\Psi^{-1}(\{x_n\leq (v_n)^3\})\subset \mathcal{R}_-$, we obtain the desired result. Lastly, since we have assumed $(\widetilde{A}(0))_{n,n}=1$ by considering $\widetilde{f}_a(z)(=\widetilde{f}(t,ax,av))$ as in the proof of \autoref{lem.hopf.curve}, the constant $\rho_1$ eventually depends on $\Omega,\Lambda$. This completes the proof.
\end{proof}

\begin{proof}[Proof of \autoref{thm.mainhar.intro}]
   \autoref{thm.mainhar.intro} is a special case of \autoref{thm.mainhar.gen}.
\end{proof}

\subsection{Sharpness of our main result}

We discuss the sharpness of our main result. First, we give a remark on the possibility to enlarge the domain in which the boundary Harnack estimate is valid.

\begin{remark}\label{rmk.bdryhar}
 We note that our result also holds for a larger domain, whenever the domain is away from $\gamma_-$. More precisely, we can obtain the H\"older estimates of the quotient in $H_{\rho/8}\setminus \mathcal{R}_-^M$ for any $M>1$, where 
    \begin{align*}
        \mathcal{R}_-^M=\{Mx_{n}\leq v_n^3\}.
    \end{align*}
    To do this, as we pointed out in the proof of \autoref{thm.mainhar}, we can derive
    \begin{align*}
        \frac{c_0}{c}\leq \inf_{H_{\frac{\rho}8}\setminus \mathcal{R}_-^M}\frac{f}{{\varphi_0}}
    \end{align*}
    for any $M\geq 1$, where $c=c(n,\Lambda,\eps,\|A\|_{C^{k+\eps}(H_2)},\|B\|_{C^{k+\eps-1}(H_2)},M)$. Thus, by the same argument as in the proof of \autoref{lem.bdryhar}, we have \eqref{first.res.bdryhol} and \eqref{second.res.bdryhol} whenever $z_0\in H_{\rho/8}\setminus \mathcal{R}^M_-$. Therefore, we can derive the H\"older estimates given in \autoref{thm.mainhar} with $\mathcal{R}_-$ replaced by $\mathcal{R}_-^M$. In that case, the constant $c$ depends on  $n,\Lambda,\eps,c_0, \|A\|_{C^{k+\eps}(H_2)},\|B\|_{C^{k+\eps-1}(H_2)}$, and on $M$.
\end{remark}

Now, we  discuss the optimality of the regularity exponents in \autoref{thm.mainhar.gen}.

\begin{example}
\label{ex:sharpness}
First note that $\phi_0=\phi_0(x_n,v_n)$ solves \eqref{mainhar:eq} with $A=I$, $B=0$, and $F_2=0$. Moreover, we have $\phi_0(z_0)=c_0 > 0$ for some universal constant $c_0$.

Now we construct suitable solutions to verify the sharpness of the result given in \autoref{thm.mainhar.gen}.

\begin{itemize}
    \item Let $F_1\equiv 1$. We use the function $\psi_0(x_n,v_n)$ that solves the equation \eqref{mainhar:eq} with $A=I$, $B=0$, and $F_1=1$ (see \eqref{eq.psi0}). However, by \cite[Lemma C.1, (C.5)]{KiWe26} and  \cite[Lemma C.6]{KiWe26}, we deduce $\phi_0(0,x_n)=c_0x_n^{\frac16}$ and $\psi_0(0,x_n)=c_1x_n^{\frac23}$, which implies 
    $\psi_0/\phi_0\notin C^{\frac32+\eps}(H_1\setminus \mathcal{R}_-)$ for any $\eps>0$. 
\item Let $F_1\equiv 0$. We observe from \eqref{eq.dphi1} that $f_1\coloneqq \frac{v_1^2}{2m_0}\phi_0(x_n,v_n)-\partial_{v_n}\phi_1(x_n,v_n)$ solves \eqref{mainhar:eq} with $A=I$, $B=0$, and $F_1=0$. Moreover, from \cite[Lemma C.1 and (C.1)]{KiWe26}, we have 
    \begin{align*}
        \frac{f_1}{\phi_0}=\frac{v_1^2}{2m_0}+c_2v_n^{2}\frac{U(-\frac16,\frac53,-\frac{v_n^3}{9x_n})}{U(-\frac16,\frac23,-\frac{v_n^3}{9x_n})}
    \end{align*}
    for some constant $c_2\neq0$ when $v_n\leq0$. For any $c>0$, we can find sequences $(a_i)_i$, $(b_i)_i$ such that $-\frac{(a_{i})^3}{b_{i}}\to c$ with  $a_{i},b_{i}\to0$ as $i\to\infty$. In particular, by choosing $|a_{i}|^3>b_{i}$ and $z_i=(t_i,x_i',x_{i,n},v_i',v_{i,n})\coloneqq (0,0,b_i,0,a_i)$, we have 
    \begin{align*}
        \lim_{i\to\infty}\frac{f_1(z_i)}{\phi_0(z_i)}\frac{1}{|v_n|^2}=c_2\frac{U(-\frac16,\frac53,c)}{U(-\frac16,\frac23,c)},
    \end{align*}
    which implies that $\lim\limits_{z\to0}\partial_{v_n,v_n}(f_1/\phi_0)(z)$ is not unique. Thus, $f_1/\phi_0\notin C^{2}(H_1\setminus \mathcal{R}_-)$.
\end{itemize}

\end{example}

\subsection{Behavior away from the grazing set}

Finally, we discuss the boundary behavior of $f_1/f_2$ away from $\gamma_0$.
The following proposition shows that $f_1/f_2$ is smooth near $\gamma_+$ but away from $\gamma_0$. 

\begin{proposition}\label{lem.bdry+}
    Let $\Omega$ be a bounded domain of class $C^{\frac{k+4+\eps}2}$ for some $k\in \N\cup\{0\}$ and $\eps\in(0,1)$, $z_0\in\gamma_+$ and $\mathbf{H}_2(z_0)\cap\gamma_0=\emptyset$. Let $A\in C^{k-1,\eps}(H_2(z_0))$, $B\in C^{k-2,\eps}(H_2(z_0))$ and $F\in C^{k-2,\eps}(H_2(z_0))$. Let $f_i$ be a weak solution to 
     \begin{equation*}
\partial_tf_i+v\cdot\nabla_{x}f_i-\ddiv(A\nabla_vf_i)=B\cdot\nabla_vf_i+F_i\quad \mbox{in  $\mathbf{H}_2(z_0)$}
\end{equation*} 
with $f_2,F_2\geq0$. Suppose $\|f_i\|_{L^\infty(\mathbf{H}_2(z_0))},\|F_i\|_{C^{k-2,\eps}(\mathbf{H}_2(z_0))}\leq 1$. Then, there exist a small constant $\rho_1=\rho_1(\Omega)$ and a large constant $\mathcal{M}=\mathcal{M}(n,\Lambda,\Omega)$ such that if $f_2(z_1)=c_0>0$ for some $c_0 > 0$ with $z_1\coloneqq z_0\circ (-1,0,0)$ and 
\begin{align}\label{bdry+:ass.F2}
    \|F_2\|_{L^\infty(\mathbf{H}_2(z_0))}\leq \frac{c_0}{\mathcal{M}},
\end{align}
then
\begin{align*}
    \left[\frac{f_1}{f_2}\right]_{C^{k,\eps}(\mathbf{H}_{\rho_1}(z_0))}\leq c,
\end{align*}
where $c=c(n,\Lambda,\eps,c_0,\|A\|_{C^{k-1,\eps}(\mathbf{H}_2(z_0))},\|B\|_{C^{k-2,\eps}(\mathbf{H}_2(z_0))},|z_0|,\Omega)$.
\end{proposition}
\begin{proof}
   Since the flattening argument given in the proof of \autoref{lem.hopf.curve} preserves the regularity of the coefficients $A,B$ and the right-hand side $F$,  by \autoref{lem.bdrysch} together with the assumption $\|f_i\|_{L^\infty(\mathbf{H}_2(z_0))},\|F_i\|_{C^{k-2,\eps}(\mathbf{H}_2(z_0))}\leq 1$ and the interpolation in \autoref{lem.interpolhol}, we have $\|f_i\|_{C^{k,\eps}(\mathbf{H}_1(z_0))}\leq c$, where $c=c(n,\Lambda,\eps,\|A\|_{C^{k-1,\eps}(\mathbf{H}_2(z_0))},\|B\|_{C^{k-2,\eps}(\mathbf{H}_2(z_0))},|z_0|,\Omega)$. 
   
   We are now going to prove $f_2(z)\geq \frac{c_0}{c}$ in $\mathbf{H}_{\rho_1}(z_0)$ for some constant $c=c(n,\Lambda,c_0,\Omega)$ provided that $f_2(z_1)=c_0$ with $z_1=z_0\circ (-1,0,0)$ and \eqref{bdry+:ass.F2} for some suitable constants $\rho_1\leq1$ and $\mathcal{M}\geq1$.
   
   As the flattening procedure is well explained in \autoref{lem.hopf.curve}, we may assume $\mathbf{H}_2(z_0)=H_2(z_0)$, $f_2(z_1)=c_0$ and $\|F_2\|_{L^\infty(H_2(z_0))}\leq \frac{c_0}{\mathcal{M}}$. Thus, it suffices to show that $f_2(z)\geq\frac{c_0}{c}$ in $H_{\rho}(z_0)$ for some universal constant $\rho$. By following the extension argument as in \eqref{f00.eq}, we also assume that $f_2$ is actually a weak solution to 
   \begin{align*}
       \partial_tf_2+v\cdot\nabla_xf_2-\ddiv(A\nabla_vf_2)=B\cdot \nabla_vf_2+F_2,\quad\text{in }Q_2(z_0),
   \end{align*}
   with ${f}_2(z_0\circ (-1,0,0))=c_0$ and $\|F_2\|_{L^\infty(Q_2(z_0))}\leq \frac{c_0}{\mathcal{M}}$. Now, taking $\mathcal{M}=\mathcal{M}(n,\Lambda)$ sufficiently small and following \eqref{harnack.ineq10f} with $z_1=z_0\circ (-1,0,0)$, we deduce  ${f}_2(z_0\circ (-\eps_0^2,0,0))\geq \frac{c_0}{c}$. Next, by applying \autoref{lem.har} with $z_1=z_0\circ (-\eps_0^2,0,0)$ and $\rho=1$, we have 
   $ \inf\limits_{Q_{r_0}(z_0)}{f}_2(z)\geq \frac{c_0}{c}$
   for some constant $c=c(n,\Lambda)$ and universal constant $r_0$, whenever $\mathcal{M}=\mathcal{M}(n,\Lambda)$ is sufficiently large. By selecting $\rho=r_0$, we have the desired lower bound of $f_2$ in $H_{\rho}(z_0)$, which completes the proof.
\end{proof}

The following example shows that our main result \autoref{thm.mainhar} together with \autoref{rmk.bdryhar} is sharp, in the sense that there is a solution $f$ such that the $L^\infty$-norm of $f/\phi_0$ in $H_1\setminus \mathcal{R}_-^M$ explodes when $M\to\infty$. Moreover, we observe that the boundary Harnack principle in $\mathcal{R}_-$ fails for any reference point in $\mathcal{R}_-$.
\begin{example}\label{ex.r-}
    Let $f\coloneqq {\phi}_{-1}\xi$ with $\xi=\xi(v)\in C^\infty$ such that $\xi(v)\equiv 0$ for $v\leq \frac{1}{32}$ and $\xi(v)\equiv1$ for $v>\frac{1}{16}$, where $\phi_{-1}$ is given in \cite[Lemma C.1]{KiWe26}. Moreover, by \cite[Lemma C.1]{KiWe26}, we have that $\phi_{-1}$ is a weak solution to $(v\partial_x-\partial_{vv})\phi_{-1}=0$ in $\{ x > 0 \} \times \{ v>\frac{1}{16} \}$. Thus, $f$ is a weak solution to 
     \begin{equation*}
\left\{
\begin{alignedat}{3}
\partial_tf+v\partial_xf-\partial_{vv}f&=F&&\qquad \mbox{in $H_{1}$}, \\
f&=0&&\qquad  \mbox{on $\{x_n=0\}\times \{v_n>0\}$},
\end{alignedat} \right.
\end{equation*} 
where $F=-\partial_{vv}\xi \phi_{-1}-2\partial_v\xi\partial_v\phi_{-1}$. In addition, from the computations given in \cite[Lemma C.1]{KiWe26}, we deduce $F\in C^\infty(H_1)$, as the support of $F$ is away from the origin. Moreover, we have $\phi_0(-1,\rho^3,0)=\phi_0(\rho^3,0)=c\rho^{\frac12}>0$, where $\rho$ is the universal constant determined in \autoref{thm.mainhar}. However, we observe from \cite[Lemma C.5]{KiWe26} that
\begin{align*}
    \|f/\phi_0\|_{L^\infty(H_1\cap \mathcal{R}_-)}\geq \|(x^{-1}v^{\frac12})/(xv^{-\frac52})\|_{L^\infty(H_1\cap \mathcal{R}_-)}=\infty.
\end{align*}
Moreover, we have 
\begin{align*}
    \|f/\phi_0\|_{L^\infty(H_1\setminus \mathcal{R}^M_-)}\eqsim M^2.
\end{align*}
In addition, for any $C > 0$, we can find $(x_1,v_1) , (x_2,v_2) \in \mathcal{R}_-$ such that  
    \begin{align*}
       \frac{\phi_0(x_1,v_1)}{\phi_{-1}(x_1,v_1)} \frac{\phi_{-1}(x_2,v_2)}{\phi_0(x_2,v_2)} = \frac{x_1^2 v_1^{-3}}{x_2^2 v_2^{-3}} \ge C.
    \end{align*}
    For instance, this can be achieved by choosing $(x_1,v_1) = (1/10,1/2)$ and $(x_2,v_2) = (1/(10\sqrt{C}),1/2)$. 
This indicates that we cannot chose any reference point so that the boundary Harnack principle holds in $H_1 \cap \mathcal{R}_-^M$.
\end{example}

\section{Appendix}
\label{sec:Appendix}

This section is motivated by \cite[Theorem 1.7]{KiWe26} (see also \cite[Theorem 8.6]{KiWe26}), where we constructed a function $\Phi$, which behaves like the square-root of the kinetic distance function to $\gamma_0$ (see \eqref{defn.Phi} for a precise definition of $\Phi$), such that $f/\Phi \in C^{0,1}$ in $H_1 \setminus \mathcal{R}_-$. Here, we construct an explicit example indicating that the Lipschitz regularity obtained in \cite{KiWe26} was optimal. Thereby, it becomes apparent that the higher regularity of order $C^{3/2}$ and $C^{1,1}$ obtained in this article is a specific feature of the quotient of two solutions.

\begin{example}\label{ex:Phi}
    Let us fix a domain $H_{\frac14}=(0,1/4^3)\times(-1/4,1/4)$. Define
    \begin{align*}
        a(v)\coloneqq 1+v,\,\varphi(x,v)\coloneqq v\phi_0-v^2\partial_v\phi_0+5\phi_0
    \end{align*}
    to see that 
    \begin{equation*}
\left\{
\begin{alignedat}{3}
v\partial_x\varphi-\partial_{v}(a(v)\partial_v\varphi)&=b\partial_v\varphi+F&&\qquad \mbox{in  $H_{\frac14}$}, \\
{\varphi}_1&=0&&\qquad  \mbox{in $\{x_n=0\}\times \{v_n>0\}$}
\end{alignedat} \right.
\end{equation*} 
for some $F\in L^\infty(H_{1/4})$, where $b\coloneqq (-\partial_v(a(v))$.
To this end, we note that
\begin{align*}
    &v\partial_x\phi_0-\partial^2_{v}\phi_0=0,\, v\partial_x(v\phi_0)-\partial^2_{v}(v\phi_0)=-2\partial_{v}\phi_0,\\
    &v\partial_x(v^2\partial_v\phi_0)-\partial^2_{v}(v^2\partial_v\phi_0)=-2\partial_{v}\phi_0-5v\partial_v^2\phi_0,\, \partial^2_{v}(v\phi_0)=2\partial_v\phi_0+v\partial^2_v\phi_0,\\
    &\partial^2_{v}(v^2\partial_v\phi_0)=2\partial_v\phi_0+4v\partial^2_v\phi_0+v^2\partial_v^3\phi_0.
\end{align*}
Thus, we have 
\begin{align*}
    v\partial_x\varphi-\partial_{v}(a(v)\partial_v\varphi)-b\partial_v=v\partial_x\varphi-a(v)\partial^2_v\varphi&=v\partial_x\varphi-\partial^2_v\varphi+(1-a(v))\partial_v^2\varphi\\
    &=5v\partial_v^2\phi_0+(-v)(-3v\partial_v^2\phi_0-v^2\partial_v^3\phi_0+5\partial_v^2\phi_0)\\
    &=3v^2\partial_v^2\phi_0+v^3\partial_v^3\phi_0\eqqcolon F(x,v)
\end{align*}
and the function $F$ is bounded. Now we claim that for 
\begin{align}\label{defn.Phi}
    \Phi(x,v)\coloneqq \phi_0(x/{\mathbf{a}(v)},v/\mathbf{a}(v))
\end{align}
with $\mathbf{a}(v)\coloneqq \sqrt{a(v)}$, 
\begin{align}\label{goal.ex}
    \frac{\varphi}{\Phi}\notin C^1(\overline{H_{1/4}}\setminus \mathcal{R}_-),
\end{align}
where $\mathcal{R}_-=\{x\leq v^3\}$.
To this end, we fix two sequences $(x_i)$ and $(v_i)$ such that $x_i\to 0$ and $v_i\to0$ with $(v_i)^3\leq x_i$,  $-\frac{(v_i)^3}{9x_i}\to z$ for some constant $z$. 

First, we prove 
\begin{align}\label{limit.Phiandphi}
    \lim_{i\to\infty}\frac{\Phi(x_i,v_i)}{|\rho_i|^{\frac12}}=\lim_{i\to\infty}\frac{\phi_0(x_i,v_i)}{|\rho_i|^{\frac12}},
\end{align}
where we write $\rho_i\coloneqq\max\{|x_i|^{\frac13},|v_i|\}$. To this end, note that by the fundamental theorem of calculus
\begin{align*}
     \phi_0(x_i/\mathbf{a}(v_i),v_i/\mathbf{a}(v_i))-\phi_0(x_i,v_i)&=\int_{0}^{1}\partial_x\phi_0(x_i(\xi/\mathbf{a}(v_i)+(1-\xi)),v_i/\mathbf{a}(v_i))\times x_i(1/\mathbf{a}(v_i)-1)\,d\xi\\
     &\quad+\int_{0}^1\partial_v\phi_0(x_i,v_i(\xi/\mathbf{a}(v_i)+1-\xi))\times v_i(1/\mathbf{a}(v_i)-1)\,d\xi.
\end{align*}
Now using this together the fact that $|x\partial_x\phi_0(x,v)|,|v\partial_v\phi_0(x,v)|\leq c(|x|^{\frac16}+|v|^{\frac12})$ and $|1/\mathbf{a}(v)-1|\leq c|v|$ for some constant $c$, we deduce
\begin{align*}
   \frac{ |\Phi(x_i^3,v_i)-\phi(x_i^3,v_i)|}{|\rho_i|^{\frac12}}\leq c|v_i|.
\end{align*}
With this estimate, we can derive \eqref{limit.Phiandphi}. Now using the explicit formula of $\phi_0$ given in \cite[Appendix C]{KiWe26}, we have when $-\frac{(v_i)^3}{9x_i}\to z\geq0$ with $v_i\leq0$,
\begin{align}\label{limit.Phi}
   0<g_0(z)\coloneqq \lim_{i\to\infty}\frac{\Phi(x_i,v_i)}{|\rho_i|^{\frac12}}=\lim_{i\to\infty}\frac{\phi_0(x_i,v_i)}{|\rho_i|^{\frac12}}=\begin{cases}
        c_0U\left(-\frac16,\frac23,z\right)&\quad\text{if }z\leq \frac19,\\
        c_0\left(\frac{1}{9z}\right)^{\frac16}U\left(-\frac16,\frac23,z\right)&\quad\text{if }z\geq\frac19.
    \end{cases}
\end{align}

Next, we prove 
\begin{equation}\label{limits3}
\begin{aligned}
    &\lim_{i\to\infty}\partial_v\left(\frac{v\phi_0}{\Phi}\right)(x_i,v_i)=1,\\
    &\lim_{i\to\infty}\partial_v\left(\frac{v^2\partial_v\phi_0}{\Phi}\right)(x_i,v_i)=\lim_{i\to\infty}\left[2\left(\frac{v\partial_v\phi_0}{\phi_0}\right)(x_i,v_i)+\left(\frac{v^2\partial^2_v\phi_0}{\phi_0}\right)(x_i,v_i)-\left(\frac{v\partial_v\phi_0}{\phi_0}\right)^2(x_i,v_i)\right],\\
    &\lim_{i\to\infty}\partial_v\left(\frac{\phi_0}{\Phi}\right)(x_i,v_i)=\lim_{i\to\infty}\left[\frac23\left(\frac{v^2\partial^2_v\phi_0}{\phi_0}\right)(x_i,v_i)+\frac23\left(\frac{v\partial_v\phi_0}{\phi_0}\right)^2(x_i,v_i)\right].
\end{aligned}
\end{equation}
To do this, we observe 
\begin{align*}
    \partial_v\left(\frac{v\phi_0}{\Phi}\right)=(\phi_0+v\partial_v\phi_0)\Phi^{-1}-\phi_0(v\partial_v\Phi)\Phi^{-2}.
\end{align*}
As in \eqref{limit.Phiandphi}, we get
\begin{align*}
    \lim_{i\to\infty}\frac{(v\partial_v\phi_0)(x_i,v_i)}{|v_i|^{\frac12}}=\lim_{i\to\infty}\frac{(v\partial_v\Phi_0)(x_i,v_i)}{|v_i|^{\frac12}}.
\end{align*}
Using this and \eqref{limit.Phi}, we have 
\begin{align*}
    \lim_{i\to\infty}\partial_v\left(\frac{v\phi_0}{\Phi}\right)(x_i,v_i)=1.
\end{align*}

Next, we observe 
\begin{align*}
      \partial_v\left(\frac{v^2\partial_v\phi_0}{\Phi}\right)=(2v\partial_v\phi_0+v^2\partial^2_v\phi_0)\Phi^{-1}-(v\partial_v\phi_0)(v\partial_v\Phi)\Phi^{-2},
\end{align*}
which implies
\begin{align*}
    \lim_{i\to\infty}\partial_v\left(\frac{v^2\partial_v\phi_0}{\Phi}\right)(x_i,v_i)=\lim_{i\to\infty}\left[2\left(\frac{v\partial_v\phi_0}{\phi_0}\right)(x_i,v_i)+\left(\frac{v^2\partial^2_v\phi_0}{\phi_0}\right)(x_i,v_i)-\left(\frac{v\partial_v\phi_0}{\phi_0}\right)^2(x_i,v_i)\right].
\end{align*}

Now we are going to estimate 
\begin{align*}
    \partial_v\left(\frac{\phi_0}{\Phi}\right)=(\partial_v\phi_0)\Phi^{-1}-\phi_0(\partial_v\Phi)\Phi^{-2}=\frac{(\Phi-\phi_0)(\partial_v\phi_0)}{\Phi^2}-\frac{\phi_0\partial_v(\Phi-\phi_0)}{\Phi^2}\eqqcolon J_1+J_2.
\end{align*}
To show the limit of $J_1+J_2$, we observe
\begin{equation}\label{est.j1+j2}
\begin{aligned}
    \phi_0(x,v)-\Phi(x,v)&= \phi_0(x,v)-\phi_0(x/\mathbf{a},v)+\phi_0(x/\mathbf{a},v)-\phi_0(x/\mathbf{a},v/\mathbf{a})\\
    &=\int_0^1\partial_x\phi_0(\xi x+(1-\xi)x/\mathbf{a},v)\times x(1-1/\mathbf{a})\,d\xi\\
    &\quad\qquad+\int_0^1\partial_v\phi_0(x/\mathbf{a},\xi v+(1-\xi)v/\mathbf{a})\times v(1-1/\mathbf{a})\,d\xi
\end{aligned}
\end{equation}
and
\begin{equation*}
    \lim_{v\to0}\frac{1-1/\mathbf{a}(v)}{v}=1.
\end{equation*}
As in \eqref{limit.Phiandphi}, we deduce
\begin{align*}
    \lim_{i\to\infty}\frac{\left[(\phi_0-\Phi)(\partial_v\phi_0)\right](x_i,v_i)}{|\rho_i|}=\lim_{i\to\infty}\frac{\left[(x\partial_x\phi_0+v\partial_v\phi_0)(v\partial_v\phi_0)\right](x_i,v_i)}{|\rho_i|^{}}.
\end{align*}
Similarly, as in \eqref{est.j1+j2} with $\phi_0$ and $\Phi$ replaced by $\partial_v\phi_0$ and $\partial_v\Phi$, we have 
\begin{align*}
    \lim_{i\to\infty}\frac{\left[\phi_0\partial_v(\phi_0-\Phi)\right](x_i,v_i)}{|\rho_i|}=\lim_{i\to\infty}\frac{\left[(vx\partial_x\partial_v\phi_0+v^2\partial^2_v\phi_0)\phi_0\right](x_i,v_i)}{|\rho_i|^{}}.
\end{align*}
Now, we combine the previous two limits and the fact that 
\begin{align*}
    3x\partial_x\phi_0+v\partial_v\phi_0=\frac12\phi_0
\end{align*}
by the homogeneity of $\phi_0$, in order to obtain
\begin{align*}
    \lim_{i\to\infty}\frac{\left(\Phi\partial_v\phi_0-\phi_0\partial_v\Phi\right)(x_i,v_i)}{|\rho_i|}=\lim_{i\to\infty}\frac{\left[(\frac16\phi_0+\frac23v\partial_v\phi_0)(v\partial_v\phi_0)+\left(-\frac16v\partial_v\phi_0+\frac23v^2\partial_v^2\phi_0\right)\phi_0\right](x_i,v_i)}{|\rho_i|^{}}.
\end{align*}
So far, we have verified \eqref{limits3}. 

Furthermore, we note from the computations given in \cite[Proof of Lemma C.1]{KiWe26} and (C.1) in \cite[Appendix C]{KiWe26} that if $-\frac{v_i^3}{x_i}\to z\geq0$ with $v_i\to0$ and $x_i\to0$, then 
\begin{align*}
    g_1(z)&\coloneqq\lim_{i\to\infty}\frac{(v\partial_v\phi_0)(x_i,v_i)}{|\rho_i|^{\frac12}}=\begin{cases}
        c_0{\frac{z}{2}U\left(\frac56,\frac53,z\right)}&\quad\text{if }z\leq \frac19,\\
        c_0\left(\frac{1}{9z}\right)^{\frac16}\frac{z}{2}U\left(\frac56,\frac53,z\right)&\quad\text{if }z\geq \frac19
    \end{cases} \\
   g_2(z)&\coloneqq\lim_{i\to\infty}\frac{(v^2\partial^2_v\phi_0)(x_i,v_i)}{|\rho_i|^{\frac12}}=\begin{cases}
        c_0\left({zU\left(\frac56,\frac53,z\right)}-z^2\frac54U\left(\frac{11}6,\frac83,z\right)\right)&\quad\text{if }z\leq \frac19,\\
        c_0\left(\frac{1}{9z}\right)^{\frac16}\left({zU\left(\frac56,\frac53,z\right)}-z^2\frac54U\left(\frac{11}6,\frac83,z\right)\right)&\quad\text{if }z\geq \frac19.
    \end{cases} 
\end{align*}
Therefore, applying these limits to \eqref{limits3} leads to 
\begin{align*}
\lim_{i\to\infty}\partial_v\left(\frac{\varphi}{\Phi}    \right)(x_{i},v_{i})=1-2\frac{g_1(z)}{g_0(z)}    -\frac13\frac{g_2(z)}{g_0(z)}+\frac53\left(\frac{g_1(z)}{g_0(z)}\right)^2.
\end{align*}
whenever $-\frac{v_i}{9x_i}\to z$ with $v_i\leq0$, $v_i,x_i\to0$.
Now, note from \cite[(C.5)-(C.8) in Appendix C]{KiWe26} that
\begin{align*}
    \frac{g_1}{g_0}(0)=\frac{g_2}{g_0}(0)=0\quad\lim_{z\to\infty}\frac{g_1}{g_0}(z)=\frac12\quad\text{and}\quad \lim_{z\to\infty}\frac{g_2}{g_0}(z)=-\frac14.
\end{align*}
Thus, using this and the continuity of the function $g_i(z)$ together with $g_0(z)>0$, there are sequences $\{(x_{1,i},0)\}_i$ and $\{(x_{2,i},v_{2,i})\}_i$ such that
\begin{align*}
\lim_{i\to\infty}\partial_v\left(\frac{\varphi}{\Phi}    \right)(x_{1,i},0)=1>\lim_{i\to\infty}\partial_v\left(\frac{\varphi}{\Phi}    \right)(x_{2,i},v_{2,i}),
\end{align*}
where $x_{1,i},x_{2,i},v_{2,i}\to0$ with $-\frac{(v_{2,i})^3}{x_{2,i}}\to c$ for large constant $c>1$. This verifies \eqref{goal.ex}.
\end{example}

\bibliographystyle{alpha}
\bibliography{literature}

\end{document}